\documentclass[12pt]{article}
\usepackage{amsmath, amssymb, amsthm, amsfonts}
\usepackage{tikz}
\usepackage[matrix,arrow,curve,frame]{xy}

\usepackage[colorlinks=true,citecolor=red,linkcolor=blue]{hyperref}

\input xy
\xyoption{all}

\swapnumbers
\numberwithin{equation}{section}

\theoremstyle{plain}
\newtheorem{theorem}[subsubsection]{Theorem}
\newtheorem{lemma}[subsubsection]{Lemma}
\newtheorem{prop}[subsubsection]{Proposition}
\newtheorem{cor}[subsubsection]{Corollary}

\theoremstyle{definition}
\newtheorem{defn}[subsubsection]{Definition}

\newtheorem{remark}[subsubsection]{Remark}

\def\CC{\mathbb{C}}

\def\NN{\mathbb{N}}

\def\RR{\mathbb{R}}

\def\ZZ{\mathbb{Z}}

\newcommand\cA{\mathcal{A}}
\newcommand\cB{\mathcal{B}}
\newcommand\cC{\mathcal{C}}

\newcommand\cE{\mathcal{E}}
\newcommand\cF{\mathcal{F}}
\newcommand\cG{\mathcal{G}}
\newcommand\cH{\mathcal{H}}

\newcommand\cM{\mathcal{M}}

\newcommand\cO{\mathcal{O}}
\newcommand\cP{\mathcal{P}}

\newcommand\cR{\mathcal{R}}

\newcommand\cU{\mathcal{U}}
\newcommand\cV{\mathcal{V}}
\newcommand\cW{\mathcal{W}}
\newcommand\cX{\mathcal{X}}
\newcommand\cY{\mathcal{Y}}
\newcommand\cZ{\mathcal{Z}}

\newcommand\frh{\mathfrak{h}}

\newcommand\tilW{\widetilde{W}}

\newcommand\id{\textup{id}}

\newcommand{\Ind}{\textup{Ind}}

\newcommand\Rep{\textup{Rep}}

\newcommand\Hom{\textup{Hom}}

\renewcommand\sl{\mathfrak{sl}}

\newcommand{\btimes}{\boxtimes}

\newcommand\quash[1]{}

\newcommand\one{\mathbf{1}}

\renewcommand\a\alpha
\renewcommand\b\beta
\newcommand\g\gamma
\renewcommand\d\delta
\newcommand\D\Delta

\renewcommand{\l}{\lambda}

\newcommand{\om}{\omega}

\begin{document}
    \bibliographystyle{alpha}

\title{On ribbon categories for singlet vertex algebras}
\author{Thomas Creutzig, Robert McRae and Jinwei Yang}
\date{}

\maketitle

\begin{abstract}
We construct two non-semisimple braided ribbon tensor categories of modules for each singlet vertex operator algebra $\mathcal{M}(p)$, $p\geq 2$. The first category consists of all finite-length $\mathcal{M}(p)$-modules with atypical composition factors, while the second is the subcategory of modules that induce to local modules for the triplet vertex operator algebra $\mathcal{W}(p)$. We show that every irreducible module has a projective cover in the second of these categories, although not in the first, and we compute all fusion products involving atypical irreducible modules and their projective covers.
\end{abstract}

\tableofcontents

\section{Introduction}

The singlet vertex algebras $\cM(p)$, $p \in \mathbb Z_{\geq 2}$, first appeared in the physics literature in the early 1990s \cite{Ka}. Together with their simple current extensions, the triplet algebras $\cW(p)$, they are the first examples of vertex algebras associated to logarithmic conformal field theories. While the triplet algebras have been thoroughly studied \cite{AM1, AM2, CF, FHST, FGST1, FGST2, NT, TW}, especially the monoidal structure on their module categories \cite{TW}, the representation theory of the singlet algebras is not yet completely understood, although see \cite{A, AM_log_intw, AM5, CM1, CMR, CGR}. The triplet algebras are $C_2$-cofinite \cite{AM1, CF} and hence have only finitely many inequivalent simple modules; moreover every (grading-restricted generalized) $\cW(p)$-module has finite length and the full category of $\cW(p)$-modules is a braided tensor category \cite{H5}. The singlet algebras, on the other hand, are not $C_2$-cofinite. They have uncountably many inequivalent simple modules, and indecomposable modules do not necessarily have finite length. These properties make understanding the representation theory of the singlet algebras a rather challenging problem.
This problem is not only interesting in its own right: representation categories of singlet algebras have important connections and applications to higher-dimensional supersymmetric gauge theories, subregular $W$-algebras at admissible levels, and $3$-manifold invariants. We now describe our main results and then comment on the implications of our work to these applications.

\subsection{Main results}

In this paper, we study two locally-finite categories of grading-restricted generalized $\cM(p)$-modules. To define these categories, let $\cO_p$ denote the category of $C_1$-cofinite grading-restricted generalized modules for the Virasoro vertex operator algebra $L(c_p,0)$ at central charge $c_p=13-6p-6p^{-1}$. The category $\cO_p$ and its direct limit completion $\Ind(\cO_p)$ are  braided tensor categories  \cite{CJORY, CMY}. Then the singlet algebra $\cM(p)$ is a commutative algebra in the braided tensor category $\Ind(\cO_p)$ \cite{HKL}, and we have an associated braided tensor category $\Rep^0\,\cM(p)$ of generalized $\cM(p)$-modules which are objects of $\Ind(\cO_p)$ when viewed as $L(c_p,0)$-modules \cite{KO,CKM1}. The irreducible $\cM(p)$-modules in $\Rep^0\,\cM(p)$ are precisely the atypical irreducibles $\cM_{r,s}$, indexed by $r\in\ZZ$, $1\leq s\leq p$.

We can then view the triplet algebra $\cW(p)$ as a commutative algebra in $\Rep^0\,\cM(p)$, and we have an associated tensor category $\Rep\,\cW(p)$ of not-necessarily-local $\cW(p)$-modules which are objects of $\Rep^0\,\cM(p)$ when viewed as $\cM(p)$-modules. As shown in \cite{KO,CKM1}, there is an induction tensor functor $\cF_{\cW(p)}: \Rep^0\,\cM(p)\rightarrow\Rep\,\cW(p)$. Now we can define our two locally-finite categories of $\cM(p)$-modules:
\begin{itemize}
 \item The category $\cC_{\cM(p)}$ is the full subcategory of finite-length grading-restricted generalized $\cM(p)$-modules whose composition factors come from the $\cM_{r,s}$ for $r\in\ZZ$, $1\leq s\leq p$.

 \item The category $\cC_{\cM(p)}^0$ is the full subcategory of generalized $\cM(p)$-modules in $\Rep^0\,\cM(p)$ that induce to (local) grading-restricted generalized $\cW(p)$-modules.
\end{itemize}
Our main results are summarized in the following theorem:
\begin{theorem}${}$
 \begin{itemize}
  \item[(1)] The category $\cC_{\cM(p)}^0$ is a proper subcategory of $\cC_{\cM(p)}$ with the same irreducible objects as $\cC_{\cM(p)}$.

  \item[(2)] The categories $\cC_{\cM(p)}$ and $\cC_{\cM(p)}^0$ are braided ribbon tensor categories, with the vertex algebraic braided tensor category structure of \cite{HLZ1}-\cite{HLZ8} and with duals given by the contragredient modules of \cite{FHL}.

  \item[(3)] Every irreducible module $\cM_{r,s}$ has a projective cover $\cP_{r,s}$ in $\cC_{\cM(p)}^0$.

  \item[(4)] The fusion rules of the irreducible modules in $\cC_{\cM(p)}$ and $\cC_{\cM(p)}^0$ are as follows:
  \begin{equation*}
   \cM_{r,s}\btimes \cM_{r',s'}   = \bigg(\bigoplus_{\substack{\ell = |s-s'|+1 \\ \ell+s+s' \equiv 1\; (\mathrm{mod}\; 2)}}^{{\rm min}\{s+s'-1, 2p-1-s-s'\}}\cM_{r+r'-1, \ell}\bigg) \oplus \bigg(\bigoplus_{\substack{\ell = 2p+1-s-s'\\ \ell+s+s' \equiv 1\; (\mathrm{mod}\; 2)}}^{p}\cP_{r+r'-1, \ell}\bigg)
  \end{equation*}
for $r,r'\in\ZZ$, $1\leq s,s'\leq p$, with sums taken to be empty if the lower bound exceeds the upper bound.
 \end{itemize}
\end{theorem}

For the fusion products of the projective covers $\cP_{r,s}$ with the irreducible modules and with each other, see Theorem \ref{generalfusionrules} below. Note that although the modules $\cP_{r,s}$ are projective in $\cC_{\cM(p)}^0$, they are not projective in the larger tensor category $\cC_{\cM(p)}$; this is one reason for introducing the smaller category $\cC_{\cM(p)}^0$.

The proof of our main theorem uses the existence of braided tensor category structure on the direct limit completion $\Ind(\cO_p)$ \cite{CJORY, CMY} as well as
three important results on the triplet $W$-algebras. First, we use the construction of certain logarithmic modules
$\mathcal R_{r, p-1}$ for $r=1, 2$ \cite{AM2}; it turns out that they contain as $\cM(p)$-module summands the projective covers $\cP_{r,p-1}$ for $r\in\ZZ$.
Secondly, \cite[Theorem 5.9]{NT} says that the $\mathcal R_{r, p-1}$ are projective $\cW(p)$-modules; we use this to show that the $\cM(p)$-modules $\cP_{r,p-1}$ are projective in $\cC_{\cM(p)}^0$. Finally, we use the fusion rules calculated in \cite{TW} for the simple $\cW(p)$-module $\mathcal W_{1, 2}$ with the remaining simple $\cW(p)$-modules; these fusion rules are needed for computing fusion products of $\cM(p)$-modules, for proving rigidity, and for constructing the remaining projective covers in $\cC_{\cM(p)}^0$. Note that our main theorem does not require the additional logarithmic projective covers of simple $\cW(p)$-modules discussed in \cite{NT}.
Also, the results we need from \cite{AM2, NT, TW} have recently been given independent proofs in \cite[Section 7]{MY}, using the structure of the Virasoro tensor category $\cO_p$, results on the category of $\cW(p)$-modules proved in \cite{AM1}, and an induction functor from $\cO_p$ to modules for $\cW(p)$.


To prove that our tensor categories of $\cM(p)$-modules are rigid, we first prove rigidity for the simple modules $\cM_{r,s}$. This part of the argument is essentially the same as the rigidity proof for $\cW(p)$ in \cite{TW}: we use BPZ equations to show that $\cM_{1,2}$ is rigid and then use fusion rules to get rigidity for the remaining simple modules. Then we prove rigidity for arbitrary finite-length modules using the following general theorem (Theorem \ref{thm:rigid_simple_to_fin_len} in the main text), which we expect to have many future applications to non-semisimple module categories for vertex operator algebras. Indeed, as we mention in Remark \ref{rem:O1_rigid} below, this theorem, combined with results from \cite{McR1, CJORY}, implies rigidity for the category $\cO_1$ of $C_1$-cofinite Virasoro modules at central charge $1$:
\begin{theorem}
 Assume that $V$ is a self-contragredient vertex operator algebra and $\cC$ is a category of grading-restricted generalized $V$-modules such that:
 \begin{itemize}
 \item[(1)] The category $\cC$ is closed under submodules, quotients, and contragredients, and every module in $\cC$ has finite length.

  \item[(2)] The category $\cC$ has braided tensor category structure as in \cite{HLZ8}.

  \item[(3)] Every simple module in $\cC$ is rigid.
 \end{itemize}
Then $\cC$ is a rigid tensor category.
\end{theorem}

We should mention that our category $\cC_{\cM(p)}$ of singlet modules is not the full category of finite-length $\cM(p)$-modules since it lacks most irreducible typical $\cM(p)$-modules (which are Fock modules for the rank-one Heisenberg extension of $\cM(p)$). Conjecturally, the category of finite-length $\cM(p)$-modules agrees with the $C_1$-cofinite module category (see for example \cite{CMR}), in which case this category would have braided tensor category structure via the methods of \cite{CJORY, CY}. In Proposition \ref{prop:C_embeds} below, we show that if the entire category of $C_1$-cofinite $\cM(p)$-modules indeed has braided tensor category structure, then $\cC_{\cM(p)}$ embeds as a tensor subcategory. In particular, the fusion rules of our main theorem do not depend on our choice of subcategory.

\subsection{Connections to topology, representation theory, and physics}

There is an intimate connection between topology, geometry, quantum groups, representation theory of vertex algebras, and physics. The singlet algebras  and their representation categories serve as an important modern example:

Motivated by three-dimensional $\mathcal N=2$ supersymmetric gauge theories, new and only conjectural invariants of $3$-manifolds, denoted by $\hat{Z}_a(q)$, were introduced in \cite{GPV}.  The $\hat{Z}_a(q)$ are $q$-series labeled by abelian flat connections, and they often turn out to be essentially mock modular forms or false theta functions. But the characters of the singlet modules $ \cM_{r,s}$ are of the false theta function type, and indeed it turns out that in some cases they coincide with these new conjectural $3$-manifold invariants \cite{CCFGH}.

In addition, physics associates a tensor category to any $3$-manifold. This tensor category should be an appropriate non-semisimple, non-finite generalization of a modular tensor category, presumably a locally-finite rigid braided tensor category that is non-degenerate in the sense that its M\"uger center is trivial.
Extended topological field theories in dimension $1+1+1$ and non-semisimple quantum invariants of closed $3$-manifolds are constructed from such non-semisimple tensor categories \cite{CGP1, D}, especially the categories of weight modules of unrolled quantum groups at odd roots of unity \cite{DGP}.
In the vertex algebra setting, already the simplest invariants are meaningful, that is, a general expectation is that open Hopf links are related to analytic properties of characters of modules for a vertex operator algebra \cite{CG}. This is useful as it provides a way to compute fusion rules:
for strongly-rational vertex operator algebras, this is the celebrated Verlinde formula \cite{V} proven by Huang \cite{H3, H4}. For the singlet algebra, the conjectural Verlinde formula involves regularized false theta functions \cite{CM}, and the normalized $S$-matrix coefficients coincide with open Hopf links of the unrolled restricted quantum group of $\mathfrak{sl}_2$ at a $2p$th root of unity \cite{CMR}. Our fusion rule computations confirm the Verlinde conjecture of \cite{CM} for the subcategory $\cC_{\cM(p)}$, and they coincide with tensor products of modules for the unrolled restricted quantum group of $\mathfrak{sl}_2$ at a $2p$th root of unity computed in \cite{CGP}. Conjecturally, this is not a coincidence:

Two major sources of braided tensor categories are modules for quantum groups and for vertex operator algebras. It is then natural to ask if there are equivalences of categories associated to quantum groups and vertex operator algebras, and this firmates under the name Kazhdan-Lusztig correspondence since they proved a braided equivalence of ordinary highest-weight modules of affine vertex algebras at generic level with corresponding quantum group modules \cite{KL1}--\cite{KL5}. The first representation theory statements concerning the triplet algebras $\cW(p)$ were made under the assumption that there is an equivalence with the category of weight modules for the restricted quantum group of $\mathfrak{sl}_2$ at a $2p$th root of unity \cite{FHST, FGST1, FGST2}, and an equivalence of abelian categories was stated to be true in \cite{NT}. A braided equivalence must fail, however, since the quantum group category turns out to be non-braidable \cite{KS}. On the other hand, the category of weight modules of the unrolled restricted quantum group of $\mathfrak{sl}_2$ at $2p$th root of unity
is conjecturally equivalent to a category of generalized modules for the $\cM(p)$-algebra \cite{CGP, CMR}, and the triplet algebra is a simple current extension of the singlet algebra. Translating back to the quantum group side has led to a quasi-Hopf algebra whose underlying algebra is the restricted quantum group and whose representation category is a finite tensor category \cite{CGR}.

Given that singlet and triplet algebras are by far the best understood vertex algebras with non-semisimple representation theory, it is fair to say that a major problem in this context is the conjectural correspondence with (quasi-Hopf modifications of the) quantum groups. A first step towards a proof is the existence of a tensor category on the vertex algebra side, that is, our main theorem. Next, one would like to prove that our category $\cC_{\cM(p)}^0$ is braided equivalent to the corresponding category of the unrolled quantum group. Uprolling, that is, performing the simple current extension, then immediately gives the correspondence between the representation category of $\cW(p)$ and the quasi-Hopf modification of the restricted quantum group; see \cite{CGR} for details on this idea.

A second family of vertex algebras with usually non-semisimple representation theory is affine vertex algebras and $W$-algebras at non-positive-integer levels, for example admissible but non-integral levels. The best-understood examples are the affine vertex algebras of $\mathfrak{sl}_2$ at admissible levels \cite{AM4, CR2}. Even for $\mathfrak{sl}_2$, the generic module has neither lower-bounded conformal weights nor finite-dimensional conformal weight spaces, that is, two essential finiteness conditions needed for the existence of tensor category structure in \cite{HLZ1}-\cite{HLZ8} fail. As a consequence, braided tensor category structure is only known to exist on the subcategory of ordinary grading-restricted modules \cite{CHY, CY}, and rigidity for this subcategory is known only in the simply-laced case \cite{CHY, C}.

Our results now allow the study of tensor categories that include relaxed highest-weight modules of special subregular $W$-algebras, namely the simple subregular $W$-algebras of $\mathfrak{sl}_{p-1}$ at level $k = - (p-1) + \frac{p-1}{p}$. These coincide with the $\mathcal B_p$-algebras (for $p\geq 3$) of \cite{CRW} by \cite[Corollary 16]{ACGY} together with \cite[Theorem 5]{ACKR}. These $\mathcal B_p$-algebras are defined as extensions of $\cM(p)$ tensored with a rank-one Heisenberg algebra. One can thus apply the theory of vertex algebra extensions \cite{CKM1} to our category $\cC_{\cM(p)}^0$ tensored with a category of Fock modules for the Heisenberg algebra to obtain a braided tensor category of  $\mathcal B_p$-algebra modules. The procedure will be very analogous to \cite{ACKR}, and as a consequence one finally has examples of rigid braided tensor categories of $W$-algebra modules that include relaxed highest-weight modules. We note that the $\mathcal B_3$-algebra is the simple affine vertex algebra of $\mathfrak{sl}_2$ at level $-\frac{4}{3}$ \cite{Ad2}, and a $\mathbb Z_2$-orbifold of the $\mathcal B_2$-algebra is the simple affine vertex algebra of $\mathfrak{sl}_2$ at level $-\frac{1}{2}$. Another family of $W$-superalgebras that have $\cM(p)$ as a Heisenberg coset are the simple principal $W$-superalgebras of $\mathfrak{sl}_{p-1|1}$ at level $k = - (p-2)+ \frac{p}{p-1}$ \cite{CGN}, so it is now also possible to study braided tensor categories for these superalgebras.
The $\mathcal B_p$-algebras are also important in physics since they appear as chiral algebras of certain four-dimensional supersymmetric gauge theories, called the $(A_1, A_{2p-3})$-Argyres-Douglas theories, by  \cite[Corollary 16]{ACGY} together with \cite[Theorem 4.1]{C2}; representation theory data of the vertex algebras relate to interesting gauge theory data (see for example \cite{BN, CS}).

Finally, we mention that there are higher-rank analogues of triplet, singlet, and $\mathcal B_p$-algebras \cite{FT, CM2, C3}, and they are expected to enjoy similar relations to topology, physics, and quantum groups; see \cite{AMW, BMM, C3, CM2, CR3, FL, L, P, R, S} for some results. To extend our work to higher rank, one first needs to show that the category of $C_1$-cofinite modules for principal $W$-algebras at appropriate levels have vertex tensor category structure, that is, one needs to generalize \cite{CJORY} beyond the Virasoro case.

\subsection{Outline}

In Section \ref{sec:background}, we provide background on the representation theory of the singlet and triplet vertex operator algebras. In
Section \ref{subsec:Vir}, we obtain  singlet module categories from the direct limit completion of the category of $C_1$-cofinite Virasoro modules, and then we compute some fusion rules in Section \ref{subsec:fusion}. In Section \ref{subsec:loc_fin_sing_cats}, we show that the categories $\cC_{\cM(p)}$ and $\cC_{\cM(p)}^0$ of singlet modules are braided tensor categories, and then we obtain some projective modules in $\cC_{\cM(p)}^0$. We establish rigidity for both $\cC_{\cM(p)}^0$ and $\cC_{\cM(p)}$ in Section \ref{sec:rigidity}. Finally, in the last section, we use rigidity to finish the proof of our main theorem, that is, we construct the remaining projective covers in $\cC_{\cM(p)}^0$ and compute the remaining fusion rules.

\vspace{5mm}

\noindent {\bf Acknowledgements}

\noindent TC acknowledges support from NSERC discovery grant RES0048511. RM thanks the University of Alberta for its hospitality during the visit in which this work was begun.

\section{The singlet and triplet \texorpdfstring{$W$}{W}-algebras}\label{sec:background}

In this section, we recall the definitions of the singlet and triplet $W$-algebras, as well as results from the representation theory of the singlet and triplet that we will use later. For more details, see for example the references \cite{A, AM1, AM2, NT, TW, CRW}.

\subsection{Setting and definitions}

For an integer $p \geq 2$, fix a rank-one lattice $L=\ZZ \a$ with
\[
\langle \a, \a\rangle = 2p.
\]
Also fix $h=\frac{\alpha}{\sqrt{2p}}\in\RR\alpha$, so that $\langle h,h\rangle =1$.

We denote the lattice vertex operator algebra associated to $L$ by $(V_L, Y, \one, \om)$. As a vector space,
\[
V_L = \cU(\widehat{\frh}_{<0})\otimes \CC[L],
\]
where $\CC[L]$ is the group algebra of $L$ and $\widehat{\frh}$ is the affinization of the abelian Lie algebra $\frh=\CC\a$. The vacuum vector of $V_L$ is $
\one = 1 \otimes 1$,
and we use the modified conformal vector
\begin{equation}\label{conformalweight}
\om = \frac{1}{4p}\a(-1)^2\one + \frac{p-1}{2p}\a(-2)\one =\frac{1}{2}h(-1)^2\one+\frac{p-1}{\sqrt{2p}}h(-2)\one.
\end{equation}
The vertex operator algebra $V_L$ has finitely many irreducible modules up to equivalence, parametrized by cosets in $L^\circ/L$, where  $L^{\circ} = \ZZ \frac{\a}{2p}$ is the dual lattice of $L$. Specifically, for $\lambda+L\in L^\circ/L$,
\[
V_{\l + L} = \cU(\widehat{\frh}_{<0})\otimes e^{\l}\CC[L]
\]
has the structure of an irreducible $V_L$-module. Taking $\l = 0$ recovers $V_L$ itself, while the full space $V_{L^{\circ}}=\bigoplus_{\lambda+L\in L^\circ/L} V_{\lambda+L}$ has the structure of a generalized vertex algebra \cite{DL}. The Virasoro algebra acts on each $V_{\l + L}$ with the central charge
\[
c_{p} := 13-6p-6p^{-1} = 1- 6\frac{(p-1)^2}{p}.
\]

Let $\cH$ be the Heisenberg vertex operator algebra associated with $\widehat{\frh}$, with the same conformal vector \eqref{conformalweight}. For $\l \in \CC$, let $\cF_{\l}$ denote the irreducible Fock $\cH$-module generated by a highest-weight vector $v_{\l}$ such that
\[
h(n) v_{\l} = \delta_{n,0}\l v_{\l}, \;\;\; n>0.
\]
In particular, $\cF_0=\cH$ itself. As a vector space,
\[
\cF_{\l} = \cU(\widehat{\frh}_{<0})\otimes \CC v_{\l},
\]
and the lowest conformal weight of $\cF_{\l}$ is
\begin{equation}\label{conformalgrading}
h_{\l} = \frac{1}{2}\l(\l-\a_0),
\end{equation}
where
\[
\a_0 = \sqrt{2p}-\sqrt{2/p}.
\]
The lattice vertex operator algebra $V_L$ is an extension of $\cH$ and decomposes as an infinite direct sum of Fock spaces as an $\cH$-module:
\[
V_L = \bigoplus_{\mu \in L}\cF_{\mu},
\]
where we have identified $\mu=n\alpha\in L$ with $n\sqrt{2p}\in\RR$. Similarly,
\[
V_{\l + L} = \bigoplus_{\mu \in \l + L}\cF_{\mu}
\]
for $\l \in L^{\circ}$.

Now define the screening operator
\begin{equation*}
 e^{-\alpha/p}_0=\mathrm{Res}_x\,Y(e^{-\alpha/p},x)
\end{equation*}
where $Y$ is the vertex operator for the generalized vertex algebra $V_{L^\circ}$. Then the singlet vertex operator algebra $\cM(p)$ is the vertex operator subalgebra
\[
\ker|_{\cF_0}\,e^{-\alpha/p}_0
\]
of $\cF_0$, and the triplet vertex operator algebra $\cW(p)$ is the vertex operator subalgebra
\[
\ker|_{V_L}\,e^{-\alpha/p}_0
\]
of $V_L$. By \cite[Theorem 3.2]{A}, the singlet $\cM(p)$ is generated as a vertex algebra by $\omega$ and $H=S_{2p-1}(\alpha)\one$, where $S_k(\alpha)$ is the Schur polynomial in the variables $\alpha(-1), \alpha(-2),\ldots$ defined by
\begin{equation*}
 \exp\left(\sum_{n=1}^{\infty} \frac{\alpha(-n)}{n}x^n\right)=\sum_{k=0}^\infty S_k(\alpha) x^k.
\end{equation*}

Introduce $\a_+ = \sqrt{2p}$, $\a_{-} = -\sqrt{2/p}$ (corresponding to $\alpha, -\frac{1}{p}\alpha\in L^\circ$), and define
\[
\a_{r,s} = \frac{1-r}{2}\a_+ + \frac{1-s}{2}\a_{-},
\]
for $r,s\in\ZZ$. Note that $\a_{r,s}$ is periodic: $\a_{r+1, s+p}=\alpha_{r,s}$ and that $\a_0 =\a_+ + \a_-$. Defining $h_{r,s}$ to be the conformal weight $h_{\alpha_{r,s}}$ of \eqref{conformalgrading}, we calculate
\begin{equation}\label{eqn:h_rs}
h_{r,s} = \frac{r^2-1}{4}p - \frac{rs-1}{2} + \frac{s^2-1}{4}p^{-1}.
\end{equation}
Now from \cite[Theorem 2.1]{A}, the singlet $\cM(p)$ decomposes into an infinite direct sum of irreducible modules for its Virasoro subalgebra:
\begin{equation}\label{sing}
\cM(p) =\bigoplus_{n =0}^{\infty} L(c_{p}, h_{2n+1, 1}),
\end{equation}
where $L(c,h)$ is the irreducible Virasoro module of central charge $c$ and lowest conformal weight $h$. By \cite[Theorem 1.1]{AM1}, the triplet $\cW(p)$ decomposes into an infinite direct sum of the same irreducible Virasoro modules, but with different multiplicities:
\begin{equation}\label{tripdecomp}
\cW(p) = \bigoplus_{n = 0}^{\infty}(2n+1)L(c_{p,1}, h_{2n+1, 1}).
\end{equation}

\subsection{Representations of the triplet algebra}

The triplet $W$-algebra $\cW(p)$ has $2p$ simple modules up to isomorphism, labeled by $\cW_{r,s}$ for $r = 1, 2$ and $1\leq s \leq p$. Recalling the correspondence
\begin{equation}\label{eqn:L_into_RR}
 \alpha_{r,s}\in\RR \longleftrightarrow \left( p(1-r)-(1-s)\right)\frac{\alpha}{2p}\in L^\circ,
\end{equation}
we have
\[
\cW_{r,p} = V_{\a_{r,p}+L},
\]
while for $1\leq s \leq p-1$, there are non-split short exact sequences
\begin{equation}\label{exactvrs}
0 \longrightarrow \cW_{r,s} \longrightarrow V_{\a_{r,s}+L} \longrightarrow \cW_{3-r,p-s} \longrightarrow 0.
\end{equation}
The lowest conformal weight of $\cW_{r,s}$ is $h_{r,s}$, and the lowest conformal weight space is one-dimensional for $r=1$ and two-dimensional for $r=2$. As $L(c_p,0)$-modules, we have
\begin{equation}\label{tripmoduledecomp}
\cW_{r, s} = \bigoplus_{n = 0}^\infty (2n+r)L(c_{p}, h_{2n+r, s})
\end{equation}
for $r = 1,2$ and $1\leq s \leq p$.

Let $\cC_{\cW(p)}$ denote the category of grading-restricted generalized (that is, logarithmic) $\cW(p)$-modules. Since $\cW(p)$ is $C_2$-cofinite \cite[Theorem 2.1]{AM1}, every simple module $\cW_{r,s}$ in $\cC_{\cW(p)}$ has a projective cover $\cR_{r,s}$ \cite[Theorem 3.23]{H5}. The lattice modules $\cW_{r,p}=V_{\alpha_{r,p}+L}$ are their own projective covers; this was first stated in \cite[Section 5]{NT}, but the proof depends only on results from \cite{AM1}. Essentially, any extension
\begin{equation*}
 0\longrightarrow\cW_{r',s'}\longrightarrow\cX\longrightarrow \cW_{r,p}\longrightarrow 0
\end{equation*}
for $(r',s')\neq(r,p)$ is non-logarithmic, so the exact sequence splits by the block decomposition of the category of ordinary $\cW(p)$-modules given in \cite[Theorem 4.4]{AM1}. On the other hand, self-extensions of $\cW_{r,p}$ split by the structural results on the Zhu algebra $A(\cW(p))$ obtained in \cite[Theorem 5.9]{AM1}. For a fuller argument, see \cite{NT} or \cite[Proposition 7.7]{MY}.


The projective covers $\cR_{r,p-1}$ were constructed explicitly in \cite{AM2}, and we recall some details of this construction. Set
$$V =V_{\alpha_{1,1}+L}\oplus V_{\a_{2,p-1}+L} =V_L\oplus V_{-\alpha/p+L}.$$
By \cite{Li}, $V$ has a vertex operator algebra structure with vertex operator
\begin{equation}\label{eqn:YV_def}
 Y_V\big((u_1,u_2),x\big)(v_1,v_2) = \big(Y(u_1,x)v_1, Y(u_1,x)v_2+Y(u_2,x)v_1\big),
\end{equation}
where $Y$ is the vertex operator for the generalized vertex algebra $V_{L^\circ}$. Moreover,
$$W = V_{\a_{2, 1}+L}\oplus V_{\a_{1,p-1}+L} =V_{\alpha/2+L}\oplus V_{(p-2)\alpha/2p+L}$$
is a $V$-module with vertex operator
\begin{equation}\label{eqn:YW_def}
 Y_W\big((v_1,v_2),x\big)(w_1,w_2) = \big(Y(v_1,x)w_1, Y(v_1,x)w_2+Y(v_2,x)w_1\big)
\end{equation}
for $v_1\in V_L$, $v_2\in V_{-\alpha/p+L}$, $w_1\in V_{\alpha_{2,1}+L}$, and $w_2\in V_{\alpha_{1,p-1}+L}$. Now we deform the vertex operators $Y_W$ and $Y_V$ using Li's $\Delta$-operators \cite{Li2}: for $u\in V_{(1)}$, set
\begin{equation*}
\Delta(u, x) = x^{u_0}\exp\left(\sum_{n=1}^{\infty}\frac{u_n}{-n}(-x)^{-n}\right),
\end{equation*}
and then define
\begin{equation*}
\widetilde{Y}_W(v,x) = Y_W(\Delta(e^{-\a/p},x)v,x),\qquad \widetilde{Y}_V(v,x)=Y_V(\Delta(e^{-\alpha/p},x)v,x)
\end{equation*}
for $v \in V$. We restrict $\widetilde{Y}_W$ to $\cW(p)$ to get a $\cW(p)$-module $\cR_{1,p-1}$ with underlying vector space $W$ and $Y_{\cR_{1,p-1}}=\widetilde{Y}_W\vert_{\cW(p)}$, and we restrict $\widetilde{Y}_V$ to $\cW(p)$ to get a $\cW(p)$-module $\cR_{2,p-1}$ with underlying vector space $V$ and $Y_{\cR_{2,p-1}}=\widetilde{Y}_V\vert_{\cW(p)}$.

Since $Y_V\vert_{V_{-\alpha/p+L}\otimes V_{-\alpha/p+L}}=0$ and $\cW(p)=\ker\vert_{V_L}\,e^{-\alpha/p}_0$, we have
\begin{align*}
 Y_{\cR_{1,p-1}}(v,x) & = Y_W(v,x)+\sum_{n=1}^{\infty} \frac{(-1)^{n+1}}{n} x^{-n}Y_W(e^{-\alpha/p}_n v,x)\nonumber\\
 Y_{\cR_{2,p-1}}(v,x) & = Y_V(v,x)+\sum_{n=1}^{\infty} \frac{(-1)^{n+1}}{n} x^{-n}Y_V(e^{-\alpha/p}_n v,x)
\end{align*}
for $v\in\cW(p)$. From this together with \eqref{eqn:YV_def} and \eqref{eqn:YW_def}, it is clear that for $r=1,2$, there is an exact sequence of $\cW(p)$-modules
\begin{equation*}
0 \rightarrow V_{\a_{r,p-1}+L} \rightarrow \cR_{r, p-1} \rightarrow V_{\a_{3-r, 1}+L} \rightarrow 0.
\end{equation*}
These exact sequences are non-split, and the $\cR_{r,p-1}$ are logarithmic $\cW(p)$-modules, because
\begin{align}\label{newvirasoro}
  Y_{\cR_{1, p-1}}(\om, x)(w_1,w_2) & = \big(Y_W(\om, x) + x^{-1}Y_W(e^{-\a/p},x)\big)(w_1,w_2)\nonumber\\
  & =\big( Y(\omega,x)w_1, Y(\omega,x)w_2+x^{-1}Y(e^{-\alpha/p},x)w_1\big),
\end{align}
and similarly for $Y_{\cR_{2,p-1}}$. In particular,
\begin{equation*}
 L(0)(w_1,w_2)=\big(L(0)w_1,L(0)w_2+e^{-\alpha/p}_0w_1\big),
\end{equation*}
which is non-semisimple (see the discussion in \cite[Section 4]{AM2} for more details).

In \cite[Section 5]{AM2}, it is shown that $\cR_{r,p-1}$ has Loewy diagram
\begin{equation*}
 \begin{tikzpicture}[->,>=latex,scale=1.5]
\node (b1) at (1,0) {$\cW_{r,p-1}$};
\node (c1) at (-1, 1){$\cR_{r, p-1}$:};
   \node (a1) at (0,1) {$\cW_{3-r, 1}$};
   \node (b2) at (2,1) {$\cW_{3-r, 1}$};
    \node (a2) at (1,2) {$\cW_{r,p-1}$};
\draw[] (b1) -- node[left] {} (a1);
   \draw[] (b1) -- node[left] {} (b2);
    \draw[] (a1) -- node[left] {} (a2);
    \draw[] (b2) -- node[left] {} (a2);
\end{tikzpicture}.
\end{equation*}
In particular, there are non-split exact sequences
\begin{equation}\label{eqn:yrw_seq}
 0 \longrightarrow \cY_{r,p-1}\longrightarrow\cR_{r,p-1} \longrightarrow\cW_{r,p-1}\longrightarrow 0
\end{equation}
and
\begin{equation}\label{eqn:wyww_seq}
 0\longrightarrow\cW_{r,p-1}\longrightarrow\cY_{r,p-1}\longrightarrow\cW_{3-r,1}\oplus\cW_{3-r,1}\longrightarrow 0,
\end{equation}
where $\cY_{r,p-1}$ is the maximal submodule generated by $V_{\alpha_{r,p-1}+L}$ and $\cW_{3-r,1}\subseteq V_{\alpha_{3-r,1}+L}$.

According to \cite[Theorem 5.9]{NT}, $\cR_{r,p-1}$ for $r=1,2$ is projective in $\cC_{\cW(p)}$ and is a projective cover of $\cW_{r,p-1}$; see \cite[Theorem 7.9]{MY} for an independent proof. For the case $p=2$, one can also show these modules are projective using the isomorphism of $\cW(2)$ with the even subalgebra of the vertex operator superalgebra of one pair of symplectic fermions; this identifies the $\cR_{r,1}$ with projective modules for the symplectic fermion superalgebra. A construction of the remaining projective covers of irreducible $\cW(p)$-modules was outlined in \cite[Section 4.1]{NT}, and an alternate construction was given in \cite[Theorem 7.9]{MY} using the tensor structure on a module category for the Virasoro subalgebra $L(c_p,0)$. However, we will not need these modules here.

%
%

Now because $\cW(p)$ is $C_2$-cofinite, \cite[Theorem 4.13]{H5} shows that $\cC_{\cW(p)}$ has braided tensor category structure as developed in \cite{HLZ1}-\cite{HLZ8}. In \cite{TW}, Tsuchiya and Wood determined fusion products in $\cC_{\cW(p)}$ and showed that it is a rigid tensor category; see also \cite[Theorems 7.5 and 7.6]{MY} for an independent derivation of these results using the tensor structure on a category of $L(c_p,0)$-modules. Here we summarize some fusion rules in $\cC_{\cW(p)}$ that we will need for studying the singlet algebra:
\begin{theorem}[\cite{TW}]\label{TW} \hspace{2em}
\begin{enumerate}
\item[(1)] The simple module $\cW_{2,1}$ is a self-dual simple current with
\begin{equation}\label{walgebra: fr2}
\cW_{2,1}\btimes \cW_{r,s} = \cW_{3-r, s}
\end{equation}
for $r=1,2$ and $1\leq s\leq p$.

\item[(2)] The simple module $\cW_{1,2}$ is rigid with fusion rules
\begin{equation}\label{walgebra: fr3}
\cW_{1,2}\btimes \cW_{r,s} =
\begin{cases}
\cW_{r, 2} \;\;\; &\mbox{if}\;\; s = 1\\
\cW_{r, s-1}\oplus \cW_{r, s+1}\;\;\; &\mbox{if}\;\; 2 \leq s \leq p-1\\
\cR_{r, p-1} \;\;\; &\mbox{if}\;\; s = p
\end{cases}
\end{equation}
for $r=1,2$.
\end{enumerate}
\end{theorem}

\begin{remark}
Since the tensor product of a rigid with a projective object in any tensor category is projective,  the $s=p$ case of \eqref{walgebra: fr3} gives a proof that $\cR_{r,p-1}$ is projective, provided one has not used projectivity of $\cR_{r,p-1}$ to prove \eqref{walgebra: fr3}. This is in fact the argument used in \cite[Theorem 7.9]{MY} to prove $\cR_{r,p-1}$ is projective, but note that the proof of \eqref{walgebra: fr3} given in \cite{TW} does use projectivity of $\cR_{r,p-1}$. Note also that projectivity of $\cW_{r,p}$ and $\cR_{r,p-1}$ are the only essential results from \cite{NT} used in \cite{TW} to prove Theorem \ref{TW}.
\end{remark}


\subsection{Representations of the singlet algebra}

The singlet vertex operator algebra $\cM(p)$ has infinitely many simple modules, first classified in \cite{A}. Here, we use \cite[Section 2]{CRW} as a reference. For $\lambda\in\CC\setminus L^\circ$, the Heisenberg Fock module $\cF_\lambda$ restricts to an irreducible $\cM(p)$-module. For $r\in\ZZ$ and $1\leq s\leq p$ on the other hand, $\cF_{\alpha_{r,s}}$ is not usually irreducible as an $\cM(p)$-module, but it contains the irreducible $\cM(p)$-module $\cM_{r,s}=\mathrm{Soc}(\cF_{\alpha_{r,s}})$. We have
\[
\cM_{r,p} = \cF_{\a_{r,p}},
\]
while for $1\leq s\leq p-1$, there are non-split short exact sequences
\begin{equation}\label{exactsing}
0 \longrightarrow \cM_{r,s} \longrightarrow \cF_{\a_{r,s}} \longrightarrow \cM_{r+1, p-s} \longrightarrow 0.
\end{equation}
These exact sequences arise from the action of a certain screening operator $Q^{[s]}_-$ that maps the Fock spaces $\cF_{\alpha_{r,s}}$ for $r\in\ZZ$ to $\cF_{\alpha_{r+1,p-s}}$, that is, $\cM_{r,s}=\ker Q^{[s]}_-\vert_{\cF_{\alpha_{r,s}}}$ and $\cM_{r+1,p-s}=\mathrm{im}\, Q^{[s]}_-\vert_{ \cF_{\alpha_{r,s}}}$. Moreover, $\cM_{r,s}=\cW_{\bar{r},s}\cap\cF_{\alpha_{r,s}}$ where $\bar{r}=1$ or $2$ according as $r$ is odd or even, and the diagram
\begin{equation}\label{diag:screening}
\xymatrixcolsep{4pc}
 \xymatrix{
  \cM_{r,s} \ar[r] \ar[d] & \cF_{\alpha_{r,s}} \ar[r]^(.3){Q^{[s]}_-\vert_{\cF_{\alpha_{r,s}}}} \ar[d] & \cM_{r+1,p-s}\subseteq\cF_{\alpha_{r+1,p-s}} \ar[d] \\
  \cW_{\bar{r},s} \ar[r] & V_{\alpha_{\bar{r},s}+L} \ar[r]^(.33){Q^{[s]}_-\vert_{V_{\alpha_{\bar{r},s}+L}}} & \cW_{3-\bar{r},p-s}\subseteq V_{\alpha_{3-\bar{r},p-s}+L} \\
 }
\end{equation}
commutes, where the unlabeled arrows are inclusions.

The lowest conformal weight of $\cM_{r,s}$ is $h_{r,s}$ if $r\geq 1$ and $h_{2-r,s}$ if $r\leq 0$ (note that these agree if $s=p$). As Virasoro modules,
\begin{align}
\cM_{r,s} & = \bigoplus_{n = 0}^{\infty} L(c_{p}, h_{r+2n, s}) \qquad (r \geq 1,\; 1\leq s \leq p)\label{eqn:Mrs_decomp_pos}\\
\cM_{r+1,p-s} & = \bigoplus_{n = 0}^{\infty} L(c_{p}, h_{r-2n, s}) \qquad (r \leq 0,\; 1\leq s \leq p)\label{eqn:Mrs_decomp_neg}
\end{align}
(note that these expressions for $r=1$ and $r=0$ agree).

For $r \in \ZZ$, we also have logarithmic $\cM(p)$-modules $\cP_{r,p-1}$ which satisfy a non-split exact sequence
\begin{equation}\label{exactproj}
0 \longrightarrow \cF_{\a_{r,p-1}} \longrightarrow \cP_{r,p-1} \longrightarrow \cF_{\a_{r-1, 1}} \longrightarrow 0.
\end{equation}
As vector spaces,
\begin{align*}
 \cP_{r,p-1}=\cF_{\alpha_{r,p-1}}\oplus\cF_{\alpha_{r-1,1}} \subseteq V_{\alpha_{r,p-1}+L}\oplus V_{\alpha_{r-1,1}+L} = V_{\alpha_{\bar{r},p-1}+L}\oplus V_{\alpha_{3-\bar{r},1}+L} =\cR_{\bar{r},p-1},
\end{align*}
where $\bar{r}=1$ or $2$ according as $r$ is odd or even. The vertex operator for $\cP_{r,p-1}$ is given by
\begin{equation*}
 Y_{\cP_{r,p-1}} = Y_{\cR_{\bar{r},p-1}}\vert_{\cM(p)\otimes\cP_{r,p-1}},
\end{equation*}
so that we have a commutative diagram
\begin{equation}\label{diag:proj_comm}
\xymatrixcolsep{4pc}
 \xymatrix{
 \cF_{\alpha_{r,p-1}} \ar[r] \ar[d] & \cP_{r,p-1} \ar[r] \ar[d] & \cF_{\alpha_{r-1,1}} \ar[d] \\
 V_{\alpha_{\bar{r},p-1}+L} \ar[r]  & \cR_{\bar{r},p-1} \ar[r] & V_{\alpha_{3-\bar{r},1}+L} \\
 }
\end{equation}
of $\cM(p)$-module homomorphisms.

To see that $\cP_{r,p-1}$ is indeed an $\cM(p)$-module, note first that \eqref{newvirasoro} shows
\begin{equation}\label{eqn:om_on_P}
 Y_{\cP_{r,p-1}}(\omega, x)(w_1,w_2)  =\big( Y(\omega,x)w_1, Y(\omega,x)w_2+x^{-1}Y(e^{-\alpha/p},x)w_1\big)
\end{equation}
for $w_1\in\cF_{\alpha_{r-1,1}}$, $w_2\in\cF_{\alpha_{r,p-1}}$, and the components of $Y(e^{-\alpha/p},x)$ indeed map the Fock space $\cF_{\alpha_{r-1,1}}$ into $\cF_{\alpha_{r,p-1}}$ since
\begin{equation*}
 \big(p(1-(r-1))-(1-1)\big)\frac{\alpha}{2p}-\frac{\alpha}{p} =\big(p(1-r)-(1-(p-1))\big)\frac{\alpha}{2p}
\end{equation*}
(recall \eqref{eqn:L_into_RR}). As for the second generator $H=S_{2p-1}(\alpha)\one$ of $\cM(p)$, we can use formulas from \cite{A} to calculate
\begin{align}\label{eqn:H_on_P}
Y_{\cP_{r, p-1}} & (H, x)(w_1,  w_2) = \left(Y_W(H, x) + \sum_{n=1}^{\infty}\frac{(-1)^n}{-n}x^{-n}Y_W(e^{-\a/p}_nH,x)\right)(w_1,w_2)\nonumber\\
& = \left(Y(H, x)w_1, Y(H,x)w_2 - \sum_{n=1}^{2p-2}\frac{(-1)^n p}{p-1}x^{-n}Y(S_{2p-2-n}(\a)e^{-\a/p},x)w_1\right)
\end{align}
for $w_1\in\cF_{\alpha_{r-1,1}}$, $w_2\in\cF_{\alpha_{r,p-1}}$. Since $\omega$ and $H$ generate $\cM(p)$ as a vertex algebra, \eqref{eqn:om_on_P} and \eqref{eqn:H_on_P} show that indeed $Y_{\cP_{r,p-1}}$ preserves $\cP_{r,p-1}$ for $r\in \ZZ$; they also show that \eqref{exactproj} is non-split exact, and that $\cP_{r,p-1}$ is a logarithmic $\cM(p)$-module.

It was conjectured in \cite[Section 5]{CGR} that $\cP_{r,p-1}$ is the projective cover of $\cM_{r,p-1}$ in a suitable subcategory of grading-restricted generalized $\cM(p)$-modules. We shall prove this in Section \ref{subsec:loc_fin_sing_cats}.

\section{Tensor categories of singlet modules}

In this section, we will construct several tensor categories of generalized $\cM(p)$-modules. Although by the Main Theorem of \cite{Mi_C1}, the category of  $C_1$-cofinite (grading-restricted generalized) modules for a vertex operator algebra is closed under the vertex algebraic tensor product of \cite{HLZ3}, this is not sufficient for existence of associativity isomorphisms as in \cite{HLZ6}. If in addition the $C_1$-cofinite module category agrees with the category of finite-length modules, however, then the method in \cite{CJORY} shows that this category has braided tensor category structure. For the singlet algebra, all finite-length modules are $C_1$-cofinite by \cite[Theorem 13]{CMR}, but unfortunately the converse remains open. Thus here, we will need to restrict our attention to $\cM(p)$-modules which live in the direct limit completion of the $C_1$-cofinite module category for the Virasoro vertex operator algebra $L(c_p,0)$, which we now discuss.

\subsection{The Virasoro tensor category and its direct limit completion}\label{subsec:Vir}

Let $\cO_p$ denote the category of $C_1$-cofinite grading-restricted generalized modules for the simple Virasoro vertex operator algebra $L(c_p,0)$. Its simple objects are the modules $L(c_p,h_{r,s})$ for $r\in\ZZ_+$, $1\leq s\leq p$. It was shown in \cite{CJORY} that all modules in $\cO_p$ have finite length and that $\cO_p$ has vertex algebraic braided tensor category structure as in \cite{HLZ1}-\cite{HLZ8}.

From \eqref{sing} and \eqref{tripdecomp}, we see that the singlet and triplet algebras are not objects of $\cO_p$, but they are objects of the direct limit completion $\Ind(\cO_p)$ introduced in \cite{CMY}. This is the full subcategory of generalized $L(c_p,0)$-modules which are unions (equivalently sums) of their $C_1$-cofinite submodules. By \cite[Example 7.2]{CMY}, $\Ind(\cO_p)$ has vertex algebraic braided tensor category structure extending that on $\cO_p$, and then \cite[Theorem 7.5]{CMY} says that $\cM(p)$ is a commutative algebra in the braided tensor category $\Ind(\cO_p)$. This means that we can apply the extension theory of \cite{CKM1} to study $\cM(p)$-modules: in particular, \cite[Theorem 7.7]{CMY} (see also \cite[Theorem 3.65]{CKM1}) says that the category $\mathrm{Rep}^0\,\cM(p)$ of generalized $\cM(p)$-modules in $\Ind(\cO_p)$ has vertex algebraic braided tensor category structure. From \eqref{eqn:Mrs_decomp_pos} and \eqref{eqn:Mrs_decomp_neg}, we see that the irreducible $\cM(p)$-modules $\cM_{r,s}$ for $r\in\ZZ$, $1\leq s\leq p$ are objects of $\mathrm{Rep}^0\,\cM(p)$, although the irreducible Fock modules $\cF_\lambda$ for $\lambda\notin L^\circ$ are not.

Now since the triplet algebra is in $\Ind(\cO_p)$ and is a module for its singlet subalgebra, we see that $\cW(p)$ is a commutative algebra in the braided tensor category $\mathrm{Rep}^0\,\cM(p)$. This means we have a braided tensor category $\mathrm{Rep}^0\,\cW(p)$ of generalized $\cW(p)$-modules in $\mathrm{Rep}^0\,\cM(p)$, as well as the larger tensor category $\mathrm{Rep}\,\cW(p)$ of not-necessarily-local $\cW(p)$-modules in  $\mathrm{Rep}^0\,\cM(p)$. From \cite{KO, CKM1}, there is an induction functor
\begin{align*}
 \cF_{\cW(p)}: \mathrm{Rep}^0\,\cM(p) & \rightarrow \mathrm{Rep}\,\cW(p)\nonumber\\
  M & \mapsto \cW(p)\boxtimes M\nonumber\\
  f & \mapsto \id_{\cW(p)}\boxtimes f.
\end{align*}
Induction is a monoidal functor, so that in particular $\cF_{\cW(p)}(M_1\boxtimes M_2)\cong\cF_{\cW(p)}(M_1)\boxtimes_{\cW(p)}\cF_{\cW(p)}(M_2)$ for generalized modules $M_1$, $M_2$ in $\mathrm{Rep}^0\,\cM(p)$.

We are mainly interested in locally-finite categories of $\cM(p)$-modules, and $\Rep^0\,\cM(p)$ is too large for these purposes (for example, it is closed under arbitrary direct sums). Probably the largest interesting ``small'' subcategory we could consider is the category of $C_1$-cofinite modules in $\Rep^0\,\cM(p)$. Unfortunately, we do not know that modules in this category necessarily have finite length. So we will instead need to consider possibly smaller locally-finite tensor subcategories of $\Rep^0\,\cM(p)$. However, since we do expect that the $C_1$-cofinite category should ultimately be the ``right'' tensor category of $\cM(p)$-modules, we will at least show here that the tensor product $\boxtimes_{C_1}$ of \cite{HLZ3, Mi_C1} for $C_1$-cofinite modules agrees with the tensor product $\boxtimes$ on $\Rep^0\,\cM(p)$ for $C_1$-cofinite modules in $\Rep^0\,\cM(p)$. That is, if in the future it is established that the category of $C_1$-cofinite $\cM(p)$-modules has vertex algebraic braided tensor category structure, then the categories we construct here will be tensor subcategories:
\begin{prop}\label{prop:C_embeds}
 If $M_1$, $M_2$ are $C_1$-cofinite $\cM(p)$-modules in $\Rep^0\,\cM(p)$, then $M_1\boxtimes M_2\cong M_1\boxtimes_{C_1} M_2$.
\end{prop}
\begin{proof}
 It is enough to show that the pair $(M_1\boxtimes M_2,\cY_{M_1,M_2})$ satisfies the intertwining operator universal property of a tensor product in the category of $C_1$-cofinite $\cM(p)$-modules, where $\cY_{M_1,M_2}$ is the tensor product intertwining operator of type $\binom{M_1\boxtimes M_2}{M_1\,M_2}$ in $\Rep^0\,\cM(p)$. Thus suppose $M_3$ is any $C_1$-cofinite $\cM(p)$-module and $\cY$ is an intertwining operator of type $\binom{M_3}{M_1\,M_2}$.

 Since $M_3$ is $C_1$-cofinite as an $\cM(p)$-module, it is in particular an $\NN$-gradable weak $L(c_p,0)$-module. Thus for any $C_1$-cofinite $L(c_p,0)$-submodules $W_1\subseteq M_1$ and $W_2\subseteq M_2$, the Key Theorem of \cite{Mi_C1} implies that $\mathrm{im}\,\cY\vert_{W_1\otimes W_2}$ is a $C_1$-cofinite $L(c_p,0)$-submodule of $M_3$. Then since $M_1$ and $M_2$ are objects of $\Rep^0\,\cM(p)$, they are the unions of their $C_1$-cofinite $L(c_p,0)$-submodules. Thus $\mathrm{im}\,\cY$ is the union of its $C_1$-cofinite $L(c_p,0)$-submodules, which means that $\mathrm{im}\,\cY$ is an object of $\Rep^0\,\cM(p)$.

 Now the universal property of the tensor product in $\Rep^0\,\cM(p)$ shows that there is a unique $\cM(p)$-module homomorphism
 $$f: M_1\boxtimes M_2\rightarrow\mathrm{im}\,\cY\subseteq M_3$$
 such that $f\circ\cY_{M_1,M_2}=\cY$. Since $\cY_{M_1,M_2}$ is a surjective intertwining operator, $f$ is the unique map $M_1\boxtimes M_2\rightarrow M_3$ with this property, so $(M_1\boxtimes M_2,\cY_{M_1,M_2})$ satisfies the universal property for a tensor product of $C_1$-cofinite $\cM(p)$-modules.
\end{proof}

Now here are the categories of $\cM(p)$-modules that we will consider in the remainder of the paper:
\begin{defn}
 The category $\cC_{\cM(p)}$ is the full subcategory of finite-length grading-restricted generalized $\cM(p)$-modules whose composition factors come from the $\cM_{r,s}$ for $r\in\ZZ$, $1\leq s\leq p$. The category $\cC_{\cM(p)}^0$ is the full subcategory of generalized $\cM(p)$-modules $M$ in $\Rep^0\,\cM(p)$ such that $\cF_{\cW(p)}(M)$ is an object of $\cC_{\cW(p)}$ (that is, a grading-restricted generalized $\cW(p)$-module).
\end{defn}

We will show that $\cC_{\cM(p)}^0$ is a subcategory of $\cC_{\cM(p)}$,  and that both are braided tensor subcategories of $\Rep^0\,\cM(p)$,    in Section \ref{subsec:loc_fin_sing_cats} below after deriving some fusion rules in $\Rep^0\,\cM(p)$. For now, we show:

\begin{prop}\label{prop:indop_exts}
The category $\cC_{\cM(p)}$ is a full subcategory of $\Rep^0\,\cM(p)$.
\end{prop}
\begin{proof}
We need to show that every module in $\cC_{\cM(p)}$ is an object of the Virasoro direct limit completion $\Ind(\cO_p)$. Since modules in $\cC_{\cM(p)}$ are finite length by definition, we can prove this by induction on the length. For the base case, the irreducible modules $\cM_{r,s}$ in $\cC_{\cM(p)}$ are objects of $\Ind(\cO_p)$ by \eqref{eqn:Mrs_decomp_pos} and \eqref{eqn:Mrs_decomp_neg}.

For the inductive step, we note that because all modules in $\cC_{\cM(p)}$ have finite length, they are in particular grading-restricted generalized $L(c_p,0)$-modules. Thus it is sufficient to prove the following claim: For any exact sequence
\[
0 \longrightarrow A \longrightarrow W \longrightarrow B \longrightarrow 0
\]
of grading-restricted generalized $L(c_p,0)$-modules with $A$ and $B$ objects of $\Ind(\cO_p)$, the extension $W$ is also an object of $\Ind(\cO_p)$.

To prove the claim, we may assume that the conformal weights of $W$ are contained in $h+\NN$ for some $h\in\CC$, since in general $W$ is a direct sum of such modules. Take an $L(0)$-eigenvector $w_1 \in W$ with minimal conformal weight, and let $W_1$ be the Virasoro submodule of $W$ generated by $w_1$. Then $W_1$ is an extension of  $W_1/A\cap W_1$ by $A \cap W_1$, which are submodules of $B$ and $A$, respectively. Since $\Ind(\cO_p)$ is closed under submodules, these two are in $\Ind(\cO_p)$. They are also subquotients of a Verma module for the Virasoro algebra, and since non-zero submodules of Virasoro Verma modules are never unions of $C_1$-cofinite submodules, $A\cap W_1$ and $W_1/A\cap W_1$ must be submodules of proper quotients of Verma modules. Such proper subquotients have finite length, so they are modules in $\cO_p$. Thus $W_1$ is also a finite-length $L(c_p,0)$-module in $\cO_p$.

Now the quotient $W/W_1$ is an extension of $W/(A+W_1) \cong (W/A)/((A+W_1)/A)$ by $(A+W_1)/W_1 \cong A/A\cap W_1$, which are quotient modules of $B$ and $A$, respectively. Since $\Ind(\cO_p)$ is closed under taking quotients, these two are objects of $\Ind(\cO_p)$. So we can use the above argument again to obtain a finite-length $L(c_p, 0)$-submodule $W_2/W_1$ of $W/W_1$ generated by an $L(0)$-eigenvector with minimal conformal weight. Continuing this way, we obtain an ascending chain of Virasoro submodules
\[
0 \subseteq W_1 \subseteq W_2 \subseteq \cdots \subseteq W_i \subseteq \cdots
\]
of $W$ such that the $W_{i+1}/W_i$ are finite-length $L(c_{p}, 0)$-modules with lowest conformal weights (non-strictly) increasing as $i$ increases. So each $W_i$ has finite length and is an object of $\cO_{p}$. Also, because $W$ has finite-dimensional conformal weight spaces, any vector $w \in W$ is contained in $W_i$ for $i$ sufficiently large, that is, $W=\bigcup_{i\geq 1} W_i$. Since each $W_i$ is in $\cO_p$, this means $W$ is in $\Ind(\cO_{p})$.
\end{proof}

\begin{remark}
 The same argument using \eqref{tripmoduledecomp} shows that the category $\cC_{\cW(p)}$ of grading-restricted generalized $\cW(p)$-modules is a full subcategory of $\Rep^0\,\cW(p)$, since every module in $\cC_{\cW(p)}$ has finite length by \cite[Proposition 4.3]{H5}. Moreover, \cite[Theorem 3.65]{CKM1} shows that $\cC_{\cW(p)}$ is also a tensor subcategory of $\Rep^0\,\cW(p)$; in particular, the fusion rules of Theorem \ref{TW} are valid in $\Rep^0\,\cW(p)$.
\end{remark}

\begin{remark}
 The exact sequences \eqref{exactsing} and \eqref{exactproj} show that the $\cM(p)$-modules $\cF_{\alpha_{r,s}}$ for $r\in\ZZ$, $1\leq s\leq p$ and $\cP_{r,p-1}$ for $r\in\ZZ$ are objects of $\cC_{\cM(p)}$ and thus also of $\Rep^0\,\cM(p)$.
\end{remark}

\subsection{Some fusion rules involving irreducible singlet modules}\label{subsec:fusion}

In this section, we calculate some fusion products of $\cM(p)$-modules in the tensor category $\Rep^0\,\cM(p)$. We start with the action by vertex operator algebra automorphisms of $\RR/L^\circ\cong U(1)$ on $V_L$ such that for $u\in\RR$ and $\lambda=n\sqrt{2p}\in L$,
\begin{equation*}
\left. \left(u+\frac{1}{\sqrt{2p}} \ZZ\right)\right\vert_{\cF_\lambda} = e^{2\pi i u\lambda}\id_{\cF_\lambda}.
\end{equation*}
This restricts to a $U(1)$-action on $\cW(p)$, and since $\cF_0$ equals the $U(1)$-fixed points in $V_L$,
\[
\cM(p) = \cW(p)^{U(1)}.
\]
Then from \cite{DLM}, $\cW(p)$ decomposes as a direct sum of irreducible $U(1)\times \cM(p)$-modules:
\[
\cW(p) = \bigoplus_{n \in \ZZ}\CC_{\a_{2n+1,1}}\otimes \cM_{2n+1,1},
\]
where $\alpha_{2n+1,1}$ represents the irreducible $U(1)$-character corresponding to $-n\alpha\in L$.

Because $U(1)$ acts by automorphisms on $\cW(p)$ and because $\cW(p)$ is a simple vertex operator algebra, the triplet vertex operator restricts to a non-zero $\cM(p)$-module intertwining operator of type $\binom{\cM_{2(n+n')+1,1}}{\cM_{2n+1,1}\,\cM_{2n'+1,1}}$ for each $n,n'\in\ZZ$. Thus the universal property of the tensor product in $\mathrm{Rep}^0\,\cM(p)$ induces a non-zero $\cM(p)$-module homomorphism
\begin{equation}\label{eqn:simple_current_surj}
 \cM_{2n+1,1}\boxtimes\cM_{2n'+1,1}\longrightarrow\cM_{2(n+n')+1,1}
\end{equation}
which is surjective because $\cM_{2(n+n')+1,1}$ is simple. Perhaps surprisingly, this map is also injective as a consequence of the tensor category structure on $\mathrm{Rep}^0\,\cM(p)$, in particular the associativity isomorphisms and the pentagon and triangle identities. This is proved in \cite[Appendix A]{CKLR}, which generalizes to compact abelian automorphism groups of a simple vertex operator algebra the corresponding result for finite cyclic groups proved in \cite[Proposition 20]{Mi} and \cite[Theorem 4.2]{CM}:
\begin{lemma}
For $n\in \ZZ$, the $\cM(p)$-modules $\cM_{2n+1, 1}$ are simple currents with fusion products
\begin{equation}\label{simplewithsimple}
\cM_{2n+1, 1} \btimes \cM_{2n'+1, 1} = \cM_{2(n+n')+1, 1}.
\end{equation}
\end{lemma}

\begin{remark}
 A more general result for non-abelian compact automorphism groups is proved in Theorem 4.7 and Corollary 4.8 of \cite{McR1}: Suppose $V$ is a simple vertex operator algebra, $G$ is a compact Lie group acting continuously on $V$ by automorphisms, and the irreducible submodules occurring in the decomposition of $V$ as a $V^G$-module are objects of some vertex algebraic braided tensor category $\cC$ of $V^G$-modules. Then these irreducible submodules are the simple objects of a tensor subcategory of $\cC$ that is braided tensor equivalent to the symmetric tensor category of finite-dimensional continuous $G$-modules. In particular, the fusion products of these irreducible $V^G$-modules correspond to tensor products of simple finite-dimensional $G$-modules; moreover, these fusion products are independent of the choice of tensor category $\cC$.
 
 The proof of \cite[Theorem 4.7]{McR1}, which gives the injectivity of \eqref{eqn:simple_current_surj}, can be viewed as a non-abelian generalization of the argument in \cite[Appendix A]{CKLR}, but it is purely tensor-categorical in nature. The same tensor-categorical injectivity result was proved earlier in \cite[Theorem 2.5]{Ki}, but under the assumptions that $G$ is finite and $\cC$ is a rigid tensor category. Theorem 4.7 of \cite{McR1}, on the other hand, uses rigidity only for the category of finite-dimensional $G$-modules and not for the tensor category $\cC$ of $V^G$-modules.
\end{remark}

\begin{remark}
 In the case $p=2$, \eqref{simplewithsimple} was proved at the level of intertwining operators in \cite{AM5}. More generally, \cite[Theorem 3.4]{AM5} implies that for any $n\in\ZZ$ and simple $\cM(2)$-module $M_1$, there is a unique simple module $M_2$ such that the space of intertwining operators of type $\binom{M_2}{\cM_{2n+1,1}\,M_1}$ is non-zero (and one-dimensional). When $M_1=\cM_{2n'+1,1}$, clearly $M_2$ must be $\cM_{2(n+n')+1,1}$.
\end{remark}

As a consequence of \eqref{simplewithsimple}, we prove:
\begin{prop}\label{prop:ind_exact}
 The induction functor $\cF_{\cW(p)}: \Rep^0\,\cM(p)\rightarrow\Rep\,\cW(p)$ is exact.
\end{prop}
\begin{proof}
 Since induction is the functor $\cW(p)\boxtimes\bullet$, it is right exact by \cite[Proposition 4.26]{HLZ3}. So we just need to show that if $f: M_1\rightarrow M_2$ is an injective $\cM(p)$-module homomorphism in $\Rep^0\,\cM(p)$, then $\id_{\cW(p)}\boxtimes f$ is still injective. As a homomorphism of $\cM(p)$-modules, $\id_{\cW(p)}\boxtimes f$ can be naturally identified with
 \begin{equation*}
  \bigoplus_{n\in\ZZ} \id_{\cM_{2n+1,1}}\boxtimes f : \bigoplus_{n\in\ZZ} \cM_{2n+1,1}\boxtimes M_1\rightarrow\bigoplus_{n\in\ZZ} \cM_{2n+1,1}\boxtimes M_2.
 \end{equation*}
Since each $\cM_{2n+1,1}$ is a necessarily rigid simple current and since tensoring with rigid objects is exact, each $\id_{\cM_{2n+1,1}}\boxtimes f$ is injective. Thus $\id_{\cW(p)}\boxtimes f$ is injective as well.
\end{proof}

Next we look at how the simple currents $\cM_{2n+1,1}$ fuse with other simple modules in $\Rep^0\,\cM(p)$, and we calculate inductions:

\begin{prop}\label{induction_simple}
For $r \in \ZZ$, $1 \leq s \leq p$, and $n \in \ZZ$,
\begin{equation}\label{simplewithatypical}
\cM_{2n+1, 1}\btimes \cM_{r,s} = \cM_{2n+r,s},
\end{equation}
and for $r\in\ZZ$, $1\leq s\leq p$,
\begin{equation}
\label{inductionsimple} \cF_{\cW(p)}(\cM_{r,s}) = \cW_{\bar{r},s}
\end{equation}
where $\bar{r} = 1$ or $2$ according as $r$ is odd or even.
\end{prop}
\begin{proof}
If we compose the injections $\cM_{2n+1,1}\rightarrow\cF_{\alpha_{2n+1,1}}$ and $\cM_{r,s}\rightarrow\cF_{\alpha_{r,s}}$ of \eqref{exactsing} with the tensor product intertwining operator for Heisenberg Fock modules, we get an $\cM(p)$-module intertwining operator of type $\binom{\cF_{\alpha_{2n+r,s}}}{\cM_{2n+1,1}\,\cM_{r,s}}$. Since Fock modules are irreducible as $\cH$-modules, \cite[Proposition 11.9]{DL} shows that this intertwining operator is non-zero, so we get a non-zero $\cM(p)$-module homomorphism $\cM_{2n+1,1}\boxtimes\cM_{r,s}\rightarrow\cF_{\alpha_{2n+r,s}}$. Since $\cM_{2n+1,1}$ is a simple current, $\cM_{2n+1,1}\boxtimes\cM_{r,s}$ is simple \cite[Proposition 2.5(3)]{CKLR}. So the image of this homomorphism is a simple submodule of $\cF_{\alpha_{2n+r,s}}$, which must be $\cM_{2n+r,s}=\mathrm{Soc}(\cF_{\alpha_{2n+r,s}})$. This proves \eqref{simplewithatypical}.

Now (\ref{simplewithatypical}) says that the $\cM_{2n+1,1}\boxtimes\cM_{r,s}$ for $n\in\ZZ$ are distinct, so by \cite[Proposition~4.4]{CKM1}, $\cF_{\cW(p)}(\cM_{r,s})$ is an irreducible $\cW(p)$-module. Then Frobenius reciprocity applied to the embedding $\cM_{r,s}\rightarrow\cW_{\bar{r},s}$ implies there is a non-zero $\cW(p)$-homomorphism $\cF_{\cW(p)}(\cM_{r,s})\rightarrow\cW_{\bar{r},s}$, which must be an isomorphism.
\end{proof}

Now we use Proposition \ref{induction_simple} to calculate fusion of the simple currents with reducible Fock modules:
\begin{prop}\label{prop:induction_Fock}
 For $r\in\ZZ$, $1\leq s\leq p-1$, and $n\in\ZZ$,
 \begin{equation}\label{simplewithtypical}
\cM_{2n+1,1}\btimes \cF_{\alpha_{r,s}} = \cF_{\alpha_{2n+r,s}},
\end{equation}
and for $r\in\ZZ$, $1\leq s\leq p-1$,
\begin{equation}
 \label{inductionverma} \cF_{\cW(p)}(\cF_{\a_{r,s}}) = V_{\a_{\bar{r},s}+L}
\end{equation}
where $\bar{r}=1$ or $2$ according as $r$ is odd or even.
\end{prop}
\begin{proof}
We use the diagram of $\cM(p)$-module homomorphisms
\begin{equation*}
 \xymatrixcolsep{3pc}
 \xymatrix{
 0 \ar[r] & \cW(p)\boxtimes\cM_{r,s} \ar[r] \ar[d] & \cW(p)\boxtimes\cF_{\alpha_{r,s}} \ar[r] \ar[d] & \cW(p)\boxtimes \cM_{r+1,p-s} \ar[r] \ar[d] & 0 \\
 0 \ar[r] & \cW(p)\boxtimes\cW_{\bar{r},s} \ar[r] \ar[d] & \cW(p)\boxtimes V_{\alpha_{\bar{r},s}+L} \ar[r] \ar[d] & \cW(p)\boxtimes\cW_{3-\bar{r},p-s} \ar[r] \ar[d] & 0 \\
 0 \ar[r] & \cW_{\bar{r},s} \ar[r] & V_{\alpha_{\bar{r},s}+L} \ar[r] & \cW_{3-\bar{r},p-s} \ar[r] & 0 \\
 }
\end{equation*}
Here the upper vertical arrows are induced by the obvious inclusions and the lower vertical arrows are induced by the $\cW(p)$-module vertex operators (which are also $\cM(p)$-module intertwining operators). All three rows are exact, the first two by Proposition \ref{prop:ind_exact}. The upper squares commute by \eqref{diag:screening}, and the lower squares commute because the lower vertical arrows come from $\cW(p)$-module vertex operators while the bottom two rows come from $\cW(p)$-module homomorphisms.

Now, the compositions of vertical arrows in the diagram are the $\cW(p)$-module homomorphisms induced by Frobenius reciprocity. Since the first and third are isomorphisms by Proposition \ref{induction_simple}, the middle homomorphism $\cF_{\cW(p)}(\cF_{\alpha_{r,s}})\rightarrow V_{\alpha_{\bar{r},s}+L}$ is also an isomorphism by the short five lemma, proving \eqref{inductionverma}. For \eqref{simplewithtypical}, we project all homomorphisms in the commutative diagram to the summands involving $\cM_{2n+1,1}\subseteq \cW(p)$, leading to the commutative diagram
\begin{equation*}
 \xymatrixcolsep{2pc}
 \xymatrix{
0 \ar[r] & \cM_{2n+1, 1}\btimes \cM_{r,s} \ar[r] \ar[d]^{\cong} &\cM_{2n+1, 1}\btimes \cF_{\a_{r,s}} \ar[d] \ar[r] & \cM_{2n+1,1}\btimes \cM_{r+1, p-s} \ar[r] \ar[d]^{\cong} \ar[r] & 0 \\
0 \ar[r] &\cM_{2n+r, s} \ar[r] & \cF_{\a_{2n+r, s}} \ar[r] & \cM_{2n+r+1, p-s} \ar[r] & 0 \\
 }
\end{equation*}
where the first and third vertical arrows are isomorphisms by \eqref{simplewithatypical}. So $\cM_{2n+1,1}\boxtimes\cF_{\alpha_{r,s}}\cong \cF_{\alpha_{2n+r,s}}$ again by the short five lemma.
\end{proof}

Using Proposition \ref{prop:induction_Fock}, the commutative diagram \eqref{diag:proj_comm}, and an argument similar to the proof of Proposition \ref{prop:induction_Fock}, we can prove:
\begin{prop}
For $r, n\in\ZZ$,
\begin{equation}\label{simplewithprojective}
\cM_{2n+1, 1}\btimes \cP_{r,p-1} = \cP_{2n+r,p-1},
\end{equation}
and for $r\in\ZZ$,
\begin{eqnarray}
\label{inductionprojective} &&\cF_{\cW(p)}(\cP_{r,p-1}) = \cR_{\bar{r},p-1}
\end{eqnarray}
where $\bar{r}=1$ or $2$ according as $r$ is even or odd.
\end{prop}

Now we present the singlet algebra analogue of Theorem \ref{TW}. However, in the spirit of the proof of Theorem \ref{TW} given in \cite{TW}, which used rigidity of $\cW_{1,2}$ to prove the $s=p$ case of Theorem \ref{TW}(2), we shall defer the proof of the $s=p$ case of \eqref{more1} to Section \ref{subsubsec:s=p}, after we have proved that $\cM_{1,2}$ is rigid and $\cP_{r,p-1}$ is projective in $\cC_{\cM(p)}^0$.
\begin{theorem}\label{fusionrules}
In the category $\Rep^0\,\cM(p)$,
\begin{itemize}
\item[(1)] The simple module $\cM_{2,1}$ is a simple current such that
\begin{equation}\label{more}
\cM_{2,1}\btimes \cM_{r,s} = \cM_{r+1, s}
\end{equation}
for $r\in\ZZ$, $1\leq s\leq p$.
\item[(2)]For $r \in \ZZ$ and $1 \leq s \leq p$,
\begin{equation}\label{more1}
\cM_{1,2}\btimes \cM_{r,s} =
\begin{cases}
\cM_{r,2} \;\;\; &{\rm if}\; s = 1\\
\cM_{r,s-1}\oplus \cM_{r, s+1}\;\;\; &{\rm if}\; 2 \leq s \leq p-1\\
\cP_{r, p-1}\;\;\; &{\rm if}\;  s = p.
\end{cases}
\end{equation}
\end{itemize}
\end{theorem}
\begin{proof}
As in the proof of Proposition \ref{induction_simple}, there is a non-zero $\cM(p)$-module homomorphism $f: \cM_{2,1}\boxtimes\cM_{r,s}\rightarrow\cF_{\alpha_{r+1},s}$ induced by a non-zero intertwining operator
\begin{equation*}
 \cM_{2,1}\otimes\cM_{r,s}\hookrightarrow\cF_{\alpha_{2,1}}\otimes\cF_{\alpha_{r,s}}\rightarrow \cF_{\alpha_{r+1,s}}\lbrace x\rbrace.
\end{equation*}
As $\mathrm{im}\,f$ is a non-zero submodule of $\cF_{\alpha_{r+1,s}}$, it is either $\cM_{r+1,s}$ or $\cF_{\a_{r+1,s}}$. Since induction is a tensor functor,
\[
\cF_{\cW(p)}(\cM_{2, 1}\btimes \cM_{r,s}) \cong \cF_{\cW(p)}(\cM_{2,1})\btimes_{\cW(p)} \cF_{\cW(p)}(\cM_{r,s}) \cong \cW_{2,1}\btimes_{\cW(p)}\cW_{\bar{r},s} \cong \cW_{3-\bar{r},s}
\]
by Proposition \ref{induction_simple} and Theorem \ref{TW}. Thus because induction is exact, we get an exact $\cW(p)$-module sequence
\begin{equation*}
 0\longrightarrow\cF_{\cW(p)}(\ker f)\longrightarrow\cW_{3-\bar{r},s}\longrightarrow \cW_{3-\bar{r},s}\;or\;V_{\alpha_{3-\bar{r},s}+L} \longrightarrow 0.
\end{equation*}
This forces $\cF_{\cW(p)}(\ker f)=0$ and $\cF_{\cW(p)}(\mathrm{im}\,f)=\cW_{3-\bar{r},s}$, so $f$ is an isomorphism onto $\cM_{r+1,s}$, proving \eqref{more}.
%

Now the $s=1$ case of \eqref{more1} follows from \eqref{simplewithatypical} and \eqref{more}, as well as the associativity and commutativity of $\boxtimes$.

For $2\leq s\leq p-1$, we have a non-zero $\cM(p)$-module homomorphism $f:\cM_{1,2}\boxtimes\cM_{r,s}\rightarrow\cF_{\alpha_{r,s+1}}$ induced by an intertwining operator
\begin{equation*}
 \cM_{1,2}\otimes\cM_{r,s}\hookrightarrow\cF_{\alpha_{1,2}}\otimes\cF_{\alpha_{r,s}}\rightarrow\cF_{\alpha_{r,s+1}}\lbrace x\rbrace,
\end{equation*}
with $\mathrm{im}\,f$ either $\cF_{\a_{r,s+1}}$ or $\cM_{r, s+1}$. Since induction is monoidal, Theorem \ref{TW} implies
\[
\cF_{\cW(p)}(\cM_{1, 2}\btimes \cM_{r,s}) \cong \cF_{\cW(p)}(\cM_{1, 2})\btimes_{\cW(p)}\cF_{\cW(p)}(\cM_{r, s}) \cong \cW_{\bar{r}, s-1}\oplus \cW_{\bar{r}, s+1}.
\]
So using exactness of $\cF_{\cW(p)}$, there is an exact sequence of $\cW(p)$-modules
\begin{equation*}
 0\rightarrow\cF_{\cW(p)}(\ker f)\rightarrow\cW_{\bar{r},s-1}\oplus\cW_{\bar{r},s+1}\rightarrow \cW_{\bar{r},s+1}\;or\;V_{\alpha_{\bar{r},s+1}+L}\rightarrow 0.
\end{equation*}
This forces $\mathrm{im}\,f = \cM_{r, s+1}$ and $\cF_{\cW(p)}(\ker(f)) \cong \cW_{\bar{r}, s-1}$. Again because $\cF_{\cW(p)}$ is exact, $\ker(f)$ has to be a simple $\cM(p)$-module and therefore equals  $\cM_{r+2n, s-1}$ for some $n \in \ZZ$. Since $h_{r+2n, s-1} - h_{r, s+1} \equiv \frac{s}{p}$ $({\rm mod}\; \ZZ)$, any extension of  $\cM_{r, s+1}$ by $\cM_{r+2n, s-1}$ for $2 \leq s \leq p-1$ splits, so $\cM_{1,2}\btimes \cM_{r,s} = \cM_{r,s+1}\oplus \cM_{r+2n,s-1}$.

Now we show that $n=0$ by constructing a non-zero map $\cM_{1,2}\boxtimes\cM_{r,s}\rightarrow\cM_{r,s-1}$. First, we have the $\cM(p)$-module map $f: \cM_{1,2}\btimes \cF_{\a_{r-1,p-s}}\rightarrow \cF_{\a_{r-1, p-s+1}}$ induced by a non-zero intertwining operator
\begin{equation*}
 \cM_{1,2}\otimes \cF_{\a_{r-1,p-s}}\hookrightarrow\cF_{\alpha_{1,2}}\otimes \cF_{\a_{r-1,p-s}}\rightarrow\cF_{\a_{r-1, p-s+1}}\lbrace x\rbrace.
\end{equation*}
Since $\cM_{1,2}$ is generated by a highest-weight vector for the Heisenberg algebra $\cH$, $\mathrm{im}\,f$ contains a highest-weight vector in $\cF_{\a_{r-1, p-s+1}}$. If $r\leq 0$, then $\cM_{r-1,p-s+1}$ does not contain a highest-weight vector for $\cH$, so $f$ is surjective in this case. Consequently, we get a surjective homomorphism
$$\bar{f}: \cM_{1,2}\btimes \cF_{\a_{r-1,p-s}} \rightarrow  \cF_{\a_{r-1, p-s+1}} \twoheadrightarrow \cM_{r, s-1}$$
using \eqref{exactsing}. Moreover, because we have seen any $\cM(p)$-module map $\cM_{1,2}\btimes \cM_{r-1, p-s} \rightarrow \cM_{r,s-1}$ is zero, right exactness of $\cM_{1,2}\boxtimes\bullet$ means $\bar{f}$ induces a non-zero $\cM(p)$-module map $g: \cM_{1,2}\btimes \cM_{r,s} \rightarrow \cM_{r, s-1}$ as in the diagram:
\begin{equation*}
  \xymatrixcolsep{2pc}
  \xymatrix{
  \cM_{1,2}\btimes \cM_{r-1, p-s} \ar[d] \ar[rd]^{0} & \\
  \cM_{1,2}\btimes \cF_{\a_{r-1, p-s}} \ar[d] \ar[r]^{\bar{f}} & \cM_{r,s-1} \\
  \cM_{1,2}\btimes \cM_{r,s} \ar@{-->}[ur]_{g} \ar[d] &\\
  0 & \\
  }
 \end{equation*}
This proves the $r\leq 0$, $2\leq s\leq p-1$ case of \eqref{more1}. The $r > 0$ case follows from \eqref{simplewithatypical} together with the $r\leq 0$ case.

We will prove the $s=p$ case of \eqref{more1} in Section \ref{subsubsec:s=p} below.
\end{proof}

\subsection{Locally-finite tensor categories of singlet modules}\label{subsec:loc_fin_sing_cats}

We have now derived enough fusion rules and induction relationships to show that $\cC_{\cM(p)}$ and $\cC_{\cM(p)}^0$ are tensor subcategories of $\Rep^0\,\cM(p)$. In the next section, we will use the fusion rules of Theorem \ref{fusionrules} to show that these tensor subcategories are rigid.

\begin{theorem}\label{thm:C0_tens_cat}
 The category $\cC_{\cM(p)}^0$ is a tensor subcategory of $\Rep^0\,\cM(p)$ and a subcategory of $\cC_{\cM(p)}$.
\end{theorem}
\begin{proof}
 The category $\cC_{\cM(p)}^0$ is closed under the tensor product on $\Rep^0\,\cM(p)$ because induction is monoidal and because $\cC_{\cW(p)}$ is a tensor subcategory of $\Rep\,\cW(p)$. This means that $\cC_{\cM(p)}^0$ is a monoidal subcategory of $\Rep^0\,\cM(p)$. To show that $\cC_{\cM(p)}^0$ is an abelian category and therefore a tensor subcategory of $\Rep^0\,\cM(p)$, it is enough to show that $\cC_{\cM(p)}^0$ is closed under submodules and quotients. Consider an exact sequence
 \begin{equation*}
  0\rightarrow N_1\rightarrow M\rightarrow N_2\rightarrow 0
 \end{equation*}
in $\Rep^0\,\cM(p)$ where $M$ is a module in $\cC^0_{\cM(p)}$. Because induction is exact by Proposition \ref{prop:ind_exact}, $\cF_{\cW(p)}(N_1)$ is a $\cW(p)$-submodule of $\cF_{\cW(p)}(M)$ and $\cF_{\cW(p)}(N_2)$ is a quotient. Since $\cC_{\cW(p)}$ is closed under submodules and quotients, this means that indeed $N_1$ and $N_2$ are modules in $\cC_{\cM(p)}^0$.

Now to show that $\cC^0_{\cM(p)}$ is a subcategory of $\cC_{\cM(p)}$, we need to show that every module $M$ in $\cC^0_{\cM(p)}$ has finite length with composition factors $\cM_{r,s}$ for $r\in\ZZ$, $1\leq s\leq p$. We first show that $M$ contains a simple submodule. Indeed, since $\cF_{\cW(p)}(M)$ is a grading-restricted generalized $\cW(p)$-module, it contains an irreducible submodule $\cW_{r,s}$. Since $\cW_{r,s}$ contains $\cM_{r,s}$ as an $\cM(p)$-submodule and since $\cF_{\cW(p)}(M)\cong\bigoplus_{n\in\ZZ} \cM_{2n+1,1}\boxtimes M$ as an $\cM(p)$-module, we have a non-zero (necessarily injective) homomorphism
\begin{equation*}
 \cM_{r,s}\rightarrow \cM_{2n+1,1}\boxtimes M
\end{equation*}
for some $n\in\ZZ$. Then because tensoring with the simple current $\cM_{-2n+1,1}$ is exact, we can use \eqref{simplewithatypical} and the associativity of $\boxtimes$ to get an injection $\cM_{r-2n,s}\rightarrow M$.

Now since $\cC_{\cM(p)}^0$ is closed under quotients, we can iterate to get an ascending chain of submodules
\begin{equation}\label{eqn:M_comp_ser}
 0\subseteq M_1\subseteq M_2\subseteq M_3\subseteq ...\subseteq M
\end{equation}
such that for each $i$, $M_{i+1}/M_i\cong\cM_{r_i,s_i}$ for $r_i\in\ZZ$, $1\leq s_i\leq p$. Because $\cF_{\cW(p)}$ is exact, inducing yields an ascending chain of submodules in $\cF_{\cW(p)}(M)$ such that
\begin{equation*}
 \cF_{\cW(p)}(M_{i+1})/\cF_{\cW(p)}(M_i)\cong\cF_{\cW(p)}(M_{i+1}/M_i)\cong\cF_{\cW(p)}(\cM_{r_i,s_i})\cong \cW_{\bar{r}_i,s_i}.
\end{equation*}
Since $\cF_{\cW(p)}(M)$ is a finite-length $\cW(p)$-module in $\cC_{\cW(p)}$, it follows that the chain \eqref{eqn:M_comp_ser} must terminate at some finite $i$, so that $M$ has a finite-length composition series with composition factors $\cM_{r_i,s_i}$.
\end{proof}

\begin{theorem}\label{thm:C_tens_cat}
 The category $\cC_{\cM(p)}$ is a tensor subcategory of $\Rep^0\,\cM(p)$.
\end{theorem}
\begin{proof}
 Because $\cC_{\cM(p)}$ is the category of all finite-length grading-restricted generalized $\cM(p)$-modules with composition factors $\cM_{r,s}$ for $r\in\ZZ$, $1\leq s\leq p$, it is closed under submodules and quotients. This means that $\cC_{\cM(p)}$ is an abelian category. We also showed in Proposition \ref{prop:indop_exts} that $\cC_{\cM(p)}$ is a full subcategory of $\mathrm{Rep}^0\,\cM(p)$. To show that $\cC_{\cM(p)}$ is a tensor subcategory, therefore, we just need to show that it is closed under the tensor product on $\Rep^0\,\cM(p)$.

 Since every module $M$ in $\cC_{\cM(p)}$ has finite length $\ell(M)$, we can use induction on $\ell(M_1)+\ell(M_2)$ to prove that $M_1\boxtimes M_2$ has finite length with composition factors $\cM_{r,s}$ for any two modules $M_1$, $M_2$ in $\cC_{\cM(p)}$. For the base case $\ell(M_1)=\ell(M_2)=1$, both $M_1$ and $M_2$ are simple modules, and \eqref{inductionsimple} shows that they are modules in $\cC_{\cM(p)}^0$. Thus Theorem \ref{thm:C0_tens_cat} shows that $M_1\boxtimes M_2$ is a (finite-length) module in $\cC_{\cM(p)}$ in this case.

 For the inductive step, assume without loss of generality that $\ell(M_1)\geq 2$, so that there is an exact sequence
 \begin{equation*}
  0\rightarrow A\xrightarrow{f} M_1\xrightarrow{g} B\rightarrow 0
 \end{equation*}
with modules $A$, $B$ in $\cC_{\cM(p)}$ satisfying $\ell(A),\ell(B)<\ell(M_1)$. Then since $\bullet\boxtimes M_2$ is right exact, we have an exact sequence
\begin{equation*}
 A\boxtimes M_2\xrightarrow{f\boxtimes\id_{M_2}} M_1\boxtimes M_2 \xrightarrow{g\boxtimes\id_{M_2}} B\boxtimes M_2\rightarrow 0.
\end{equation*}
Thus $M_1\boxtimes M_2$ is an $\cM(p)$-module extension of $B\boxtimes M_2$ by $A\boxtimes M_2/\ker(f\boxtimes\id_{M_2})$. Since $A\boxtimes M_2$ and $B\boxtimes M_2$ are modules in $\cC_{\cM(p)}$ by the inductive hypothesis, it follows that $M_1\boxtimes M_2$ is in $\cC_{\cM(p)}$ as well.
\end{proof}

The reason we have introduced the category $\cC_{\cM(p)}^0$ in addition to $\cC_{\cM(p)}$ is that, as we shall show here and in Section \ref{sec:proj} below, the irreducible modules $\cM_{r,s}$ have projective covers in $\cC_{\cM(p)}^0$ but not in $\cC_{\cM(p)}$. The following lemma relates projective objects in $\cC_{\cW(p)}$ to projective objects in $\cC_{\cM(p)}^0$:
\begin{lemma}\label{lem:ind_proj}
 If $\cF_{\cW(p)}(P)$ is projective in $\cC_{\cW(p)}$ for some module $P$ in $\cC_{\cM(p)}^0$, then $P$ is projective in $\cC_{\cM(p)}^0$.
\end{lemma}
\begin{proof}
 Consider a surjection $p: M\twoheadrightarrow N$ in $\cC_{\cM(p)}^0$ and a homomorphism $q: P\rightarrow N$. Since all three modules here are objects of $\cC_{\cM(p)}^0$, we can induce to a diagram of homomorphisms in $\cC_{\cW(p)}$. Since $\cF_{\cW(p)}$ is exact, $\cF_{\cW(p)}(p)$ is still surjective and projectivity of $\cF_{\cW(p)}(P)$ in $\cC_{\cW(p)}$ implies there is a homomorphism $f: \cF_{\cW(p)}(P)\rightarrow\cF_{\cW(p)}(M)$ such that the diagram
 \begin{equation*}
  \xymatrixcolsep{3pc}
  \xymatrix{
  & \cF_{\cW(p)}(P) \ar[ld]_{f} \ar[d]^{\cF_{\cW(p)}(q)} \\
  \cF_{\cW(p)}(M) \ar[r]_{\cF_{W(p)}(p)} & \cF_{W(p)}(N) \\
  }
 \end{equation*}
commutes.

Now for a module $X$ in $\cC_{\cM(p)}^0$, since $\cF_{\cW(p)}(X)\cong\bigoplus_{n\in\ZZ} \cM_{2n+1,1}\boxtimes X$ as an $\cM(p)$-module, we have $\cM(p)$-module homomorphisms
\begin{equation*}
 \iota^{(n)}_X: \cM_{2n+1,1}\boxtimes X\rightarrow \cF_{\cW(p)}(X),\qquad \pi^{(n)}_X: \cF_{\cW(p)}(X)\rightarrow\cM_{2n+1,1}\boxtimes X
\end{equation*}
such that
\begin{equation*}
 \pi^{(m)}_X\circ\iota_X^{(n)} =\delta_{m,n}\id_{\cM_{2n+1,1}\boxtimes X},\qquad\sum_{n\in\ZZ} \iota_X^{(n)}\circ\pi_{X}^{(n)}=\id_{\cF_{\cW(p)}(X)}.
\end{equation*}
Note that the infinite sum here is well defined since it is finite when acting on any vector in $\cF_{\cW(p)}(X)$.

Since $\cF_{\cW(p)}(g)=\id_{\cW(p)}\boxtimes g$ for any morphism $g$ in $\cC_{\cM(p)}^0$, we have
\begin{equation*}
 \pi^{(0)}_N\circ\cF_{\cW(p)}(p)\circ\iota_M^{(n)} =\delta_{0,n}(\id_{\cM_{1,1}}\boxtimes p),\qquad
  \pi^{(0)}_N\circ\cF_{\cW(p)}(q)\circ\iota_P^{(0)} =\id_{\cM_{1,1}}\boxtimes q.
\end{equation*}
Thus if we set $\widetilde{f}=\pi^{(0)}_{M}\circ f\circ\iota_P^{(0)}$, we get
\begin{align*}
 \id_{\cM_{1,1}}\boxtimes q = \pi_N^{(0)}\circ\cF_{\cW(p)} & (q)\circ\iota_P^{(0)} =\pi_N^{(0)}\circ\cF_{\cW(p)}(p)\circ f\circ\iota_P^{(0)}\nonumber\\
 &= \sum_{n\in\ZZ} \pi^{(0)}_N\circ\cF_{\cW(p)}(p)\circ\iota_M^{(n)}\circ\pi^{(n)}_M\circ f\circ\iota_P^{(0)} =(\id_{\cM_{1,1}}\boxtimes p)\circ\widetilde{f}.
\end{align*}
Since we can identify $\widetilde{f}$ with a homomorphism $g: M\rightarrow N$ such that $q=p\circ g$, this shows that $P$ is projective in $\cC_{\cM(p)}^0$.
\end{proof}

As a consequence:
\begin{cor}\label{cor:first_proj_p}
 For $r\in\ZZ$, the $\cM(p)$-modules $\cP_{r,p}:=\cF_{\alpha_{r,p}}=\cM_{r,p}$ and $\cP_{r,p-1}$ are projective in $\cC_{\cM(p)}^0$.
\end{cor}
\begin{proof}
We have $\cF_{\cW(p)}(\cF_{\a_{r,p}}) \cong V_{\a_{r,p}+L}$ by \eqref{inductionverma} and $\cF_{\cW(p)}(\cP_{r,p-1})\cong\cR_{\bar{r},p-1}$ by \eqref{inductionprojective}, both of which are projective in $\cC_{\cW(p)}$ according to \cite[Section 5.1]{NT} (see also Proposition 7.7 and Theorem 7.9 in \cite{MY}). So $\cM_{r,p}$ and $\cP_{r,p-1}$ are projective in $\cC_{\cM(p)}^0$ by Lemma \ref{lem:ind_proj}.
\end{proof}

Since $\cM_{r,p}$ is irreducible and projective in $\cC_{\cM(p)}^0$, it is obvious that it is its own projective cover in $\cC_{\cM(p)}^0$. However, we now show that $\cM_{r,p}$ is not projective, and in fact has no projective cover, in $\cC_{\cM(p)}$. For any $n\in\ZZ_+$, let $\cF_{\alpha_{r,p}}^{(n)}$ denote the indecomposable Heisenberg module induced from an $n$-dimensional lowest conformal weight space on which $\alpha(0)$ acts by the indecomposable Jordan block
\begin{equation*}
 A_n =\sqrt{2p} \left[\begin{array}{cccc}
        \alpha_{r,p} & 1 & \cdots & 0\\
        0 & \alpha_{r,p} & \ddots & \vdots \\
        \vdots & \vdots & \ddots & 1 \\
        0 & 0 & \cdots & \alpha_{r,p} \\
       \end{array}
\right].
\end{equation*}
Thus for $n\geq 2$, we have a non-split extension of $\cH$-modules
\begin{equation}\label{eqn:self-ext}
 0\longrightarrow \cF_{\alpha_{r,p}}^{(n-1)}\longrightarrow\cF_{\alpha_{r,p}}^{(n)}\longrightarrow\cF_{\alpha_{r,p}}\longrightarrow 0.
\end{equation}
Now we consider $\cF_{\alpha_{r,p}}^{(n)}$ as an $\cM(p)$-module (see \cite[Theorem 6.1]{AM_log_intw} for the $n=2$ case). On the lowest conformal weight space, $L(0)$ and the degree-preserving component $H(0)$ of $Y(H,x)$ act by the matrices
\begin{equation*}
 L(0)=\frac{1}{4p} A_n^2-\frac{p-1}{2p} A_n,\qquad H(0)=\binom{A_n}{2p-1}.
 \end{equation*}
Analysis of these matrices shows that $L(0)$ acts indecomposably for $r\neq 1$, while for $r=1$, $H(0)$ is a nilpotent indecomposable matrix. Thus $\cF_{\alpha_{r,p}}^{(n)}$ is also singly-generated indecomposable as an $\cM(p)$-module.

Since each $\cF_{\alpha_{r,p}}^{(n)}$ is an object of $\cC_{\cM(p)}$, the non-split exact sequence \eqref{eqn:self-ext} shows that $\cM_{r,p}=\cF_{\alpha_{r,p}}$ is not projective in $\cC_{\cM(p)}$. But since $\cM_{r,p}$ is projective in $\cC_{\cM(p)}^0$, this means that for $n\geq 2$, $\cF_{\alpha_{r,p}}^{(n)}$ is an object of $\cC_{\cM(p)}$ that is not in $\cC_{\cM(p)}^0$. In particular, $\cF_{\cW(p)}(\cF_{\alpha_{r,p}}^{(n)})$ must be a non-local finite-length $\cW(p)$-module that has local composition factors. We can also see that $\cM_{r,p}$ fails to have a projective cover in $\cC_{\cM(p)}$ because a projective cover would have to surject onto $\cF_{\alpha_{r,p}}^{(n)}$ for all $n$, which is impossible for a finite-length $\cM(p)$-module.

We now discuss properties of $\cP_{r,p-1}$ in $\cC_{\cM(p)}^0$:
\begin{prop}\label{prop:p-1_proj_cover}
 For $r\in\ZZ$, $\cP_{r,p-1}$ is a projective cover of $\cM_{r,p-1}$ in $\cC_{\cM(p)}^0$, and it has Loewy diagram
 \begin{equation*}
 \begin{tikzpicture}[->,>=latex,scale=1.5]
\node (b1) at (1,0) {$\cM_{r,p-1}$};
\node (c1) at (-1, 1){$\cP_{r, p-1}$:};
   \node (a1) at (0,1) {$\cM_{r+1, 1}$};
   \node (b2) at (2,1) {$\cM_{r-1, 1}$};
    \node (a2) at (1,2) {$\cM_{r,p-1}$};
\draw[] (b1) -- node[left] {} (a1);
   \draw[] (b1) -- node[left] {} (b2);
    \draw[] (a1) -- node[left] {} (a2);
    \draw[] (b2) -- node[left] {} (a2);
\end{tikzpicture}.
\end{equation*}
\end{prop}

\begin{proof}
From \eqref{exactproj} and \eqref{exactsing}, $\cP_{r,p-1}$ has a submodule/subquotient structure as illustrated by the diagram
\begin{equation}\label{diag:P_structure}
 \xymatrixcolsep{2pc}
 \xymatrixrowsep{1pc}
 \xymatrix{
  & 0 \ar[d] & & 0 \ar[d] & \\
  & \cM_{r,p-1} \ar[rd] \ar[d] & & \cM_{r-1,1} \ar[d] & \\
 0 \ar[r] & \cF_{\alpha_{r,p-1}} \ar[d] \ar[r] & \cP_{r,p-1} \ar[rd] \ar[r] & \cF_{\alpha_{r-1,1}} \ar[r] \ar[d] & 0 \\
 & \cM_{r+1,1} \ar[d] & & \cM_{r,p-1} \ar[d] & \\
 & 0 & & 0 & \\
 }
\end{equation}
In addition to the submodules here, we have two additional submodules obtained by intersecting $\cW(p)$-submodules of $\cR_{\bar{r},p-1}$ with $\cP_{r,p-1}$: $\cG_{r,p-1}$ whose underlying vector space is $\cM_{r,p-1}\oplus\cM_{r-1,1}$ and which contains $\cM_{r,p-1}$ as a submodule, and $\cZ_{r,p-1}=\cY_{r,p-1}\cap\cP_{r,p-1}$.

By \eqref{inductionprojective} and exactness of induction, semisimple submodules of $\cP_{r,p-1}$ induce to semisimple submodules of $\cR_{\bar{r},p-1}$ with the same length. Since $\mathrm{Soc}(\cR_{\bar{r},p-1})=\cW_{\bar{r},p-1}$, this means that $\mathrm{Soc}(\cP_{r,p-1})=\cM_{r,p-1}$. Next, restricting the exact sequences \eqref{eqn:yrw_seq} and \eqref{eqn:wyww_seq} to $\cM(p)$-submodules, we get exact sequences
\begin{equation*}
 0\longrightarrow \cZ_{r,p-1}\longrightarrow\cP_{r,p-1}\longrightarrow\cM_{r,p-1}\longrightarrow 0
\end{equation*}
and
\begin{equation*}
 0\longrightarrow\cM_{r,p-1}\longrightarrow\cZ_{r,p-1}\longrightarrow\cM_{r+1,1}\oplus\cM_{r-1,1}\longrightarrow 0.
\end{equation*}
The first sequence does not split because induction is exact and $\cR_{\bar{r},p-1}\cong\cF_{\cW(p)}(\cP_{r,p-1})$ is indecomposable. This together with the second sequence implies
\begin{equation*}
 \mathrm{Soc}(\cP_{r,p-1}/\cM_{r,p-1})=\cZ_{r,p-1}/\cM_{r,p-1}\cong\cM_{r+1,1}\oplus\cM_{r-1,1},
\end{equation*}
and then clearly
$$\mathrm{Soc}(\cP_{r,p-1}/\cZ_{r,p-1})=\cP_{r,p-1}/\cZ_{r,p-1}\cong\cM_{r,p-1}.$$
This verifies the row structure of the Loewy diagram for $\cP_{r,p-1}$. To complete the verification of the arrow structure, we need to show that $\cG_{r,p-1}$ and $\cP_{r,p-1}/\cG_{r,p-1}$ are indecomposable. For $\cG_{r,p-1}$, we note that a similar proof to that of \eqref{inductionverma} shows that $\cF_{\cW(p)}(\cG_{r,p-1})$ is isomorphic to the corresponding indecomposable $\cW(p)$-submodule of $\cR_{\bar{r},p-1}$ (called $\cM_1$ in \cite[Lemma 5.2]{AM2}); this means $\cG_{r,p-1}$ is indecomposable. Then $\cP_{r,p-1}/\cG_{r,p-1}$ is indecomposable because it induces to the indecomposable $\cW(p)$-module $\cR_{\bar{r},p-1}/\cF_{\cW(p)}(\cG_{r,p-1})$.

Now we can show that $\cP_{r,p-1}$ is a projective cover of $\cM_{r,p-1}$. Note that the Loewy diagram implies that $\cP_{r,p-1}$ is generated by any vector not in $\cZ_{r,p-1}$, that is, $\cZ_{r,p-1}$ is the unique maximal proper submodule of $\cP_{r,p-1}$. Also, there is a surjection $q:\cP_{r,p-1}\rightarrow\cM_{r,p-1}$ such that $\ker q=\cZ_{r,p-1}$. Now suppose $P$ is any projective object in $\cC_{\cM(p)}^0$ with surjective map $\widetilde{q}: P\rightarrow\cM_{r,p-1}$. Then there is a map $f: P\rightarrow\cP_{r,p-1}$ such that the diagram
\begin{equation*}
 \xymatrixcolsep{3pc}
 \xymatrix{
 & P \ar[ld]_{f} \ar[d]^{\widetilde{q}} \\
 \cP_{r,p-1} \ar[r]_q & \cM_{r,p-1} \\
 }
\end{equation*}
commutes. We need to show that $f$ is surjective. Indeed, otherwise we would have $\mathrm{im}\,f\subseteq\cZ_{r,p-1}$, so that $q\circ f=0$, contradicting the surjectivity of $\widetilde{q}$.
\end{proof}

\section{Rigidity}\label{sec:rigidity}

In this section, we prove that $\cC_{\cM(p)}$ is a rigid tensor category, using the method of \cite{TW} for proving rigidity of $\cC_{\cW(p)}$. The steps of the proof are the following:
\begin{itemize}
 \item[(1)] First prove that $\cM_{1,2}$ is rigid and self-dual using the fusion rules \eqref{more1} and Belavin-Polyakov-Zamolodchikov differential equations. This proof is exactly the same as the rigidity proof in \cite{TW} for the $\cW(p)$-module $\cW_{1,2}$, but we provide a more detailed exposition, especially for $p=2$, in which case we do not assume the fusion rule \eqref{more1}.

 \item[(2)] Next, use rigidity of $\cM_{1,2}$ and of the simple currents $\cM_{r,1}$, together with the fusion rules from the previous section, to show that all simple modules in $\cC_{\cM(p)}$ are rigid.

 \item[(3)] Finally, use rigidity of simple modules and induction on length to prove that all modules in $\cC_{\cM(p)}$ are rigid. Here, our proof differs from that of \cite{TW}, which used a lemma from the Appendix of \cite{KL5} whose proof is valid only in categories with enough rigid projective objects.

\end{itemize}
We begin with a discussion of the BPZ equations satisfied by compositions of intertwining operators involving the module $\cM_{1,2}$.

\subsection{Belavin-Polyakov-Zamolodchikov equations}

Let $\cY_1$ and $\cY_2$ be $\cM(p)$-module intertwining operators of types $\binom{\cM_{1,2}}{\cM_{1,2}\,M}$ and $\binom{M}{\cM_{1,2}\,\cM_{1,2}}$, respectively, for some $\cM(p)$-module $M$. Also, let $\cY^1$ and $\cY^2$ be  $\cM(p)$-module intertwining operators of types $\binom{\cM_{1,2}}{N\,\cM_{1,2}}$ and $\binom{N}{\cM_{1,2}\,\cM_{1,2}}$, respectively, for some $\cM(p)$-module $N$. We then define
\begin{eqnarray*}
&& \varphi(x)=\langle v',\cY_1(v,1)\cY_2(v,x)v\rangle\\
&& \psi(x) = \langle v', \cY^1(\cY^2(v,1-x)v,x)v\rangle,
\end{eqnarray*}
where $v$ is a highest-weight vector in $\cM_{1,2}$, of conformal weight $h_{1,2}$, and $v'$ is a highest-weight vector in the contragredient module $\cM_{1,2}'$. (It turns out that $\cM_{1,2}'\cong\cM_{1,2}$, so later we will identify $\langle\cdot,\cdot\rangle$ with a non-degenerate invariant bilinear form on $\cM_{1,2}$.) We can view $\varphi(x)$ and $\psi(x)$ either as formal series in $x$ and $1-x$, respectively, or, replacing $x$ with $z$, as analytic functions on the (simply-connected) regions
\begin{equation*}
U_1 = \lbrace z\in\CC\,\vert\,\vert z\vert<1\rbrace\setminus(-1,0]
\end{equation*}
and
\begin{equation*}
U_2 = \lbrace z\in\CC\,\vert\,\vert 1-z\vert<\vert z\vert\rbrace\setminus[1,\infty) =\lbrace z\in\CC\,\vert\,\mathrm{Re}\,z>1/2\rbrace\setminus [1,\infty),
\end{equation*}
respectively.

It is well known that singular vectors in the Virasoro Verma module of conformal weight $h_{1,2}$ lead to a differential equation for $\varphi(x)$ and $\psi(x)$ \cite{BPZ, Hu_Vir}. To derive this equation, we need the following consequences of the intertwining operator Jacobi identity:
\begin{lemma}
Let $\cY$ be a logarithmic intertwining operator among $V$-modules of type $\binom{W_3}{W_1\,W_2}$. Then for $u\in V$, $w_1 \in W_1$, and $n \in \ZZ$,
\begin{eqnarray}\label{eqn:commutator}
[u_n, \cY(w_1,x)] = \sum_{k\geq 0}\binom{n}{k}x^{n-k}\cY(u_kw_1,x),
\end{eqnarray}
and
\begin{eqnarray}\label{eqn:iterate}
\cY(u_{n}w,z) = \sum_{k \geq 0}\binom{n}{k}(-x)^k u_{n-k}\cY(w_1,x)-\sum_{k\geq 0}\binom{n}{k}(-x)^{n-k}\cY(w_1,x)u_k.
\end{eqnarray}
\end{lemma}

\begin{prop}\label{prop}
The series $\varphi(x)$ and $\psi(x)$ satisfy the differential equation
\begin{equation}\label{eqn:original}
px(1-x)f''(x)+(1-2x)f'(x)-\frac{h_{1,2}}{x(1-x)}f(x)  = 0.
 \end{equation}
\end{prop}

\begin{proof}
We first derive a partial differential equation for the formal series
\begin{equation*}
 \Phi(x_1,x_2)=\langle v',\cY_1(v,x_1)\cY_2(v,x_2)v\rangle
\end{equation*}
 using the fact that
 \begin{equation*}
  \left(L(-1)^2-\frac{1}{p}L(-2)\right)v=0
 \end{equation*}
in the irreducible Virasoro module $L(c_{p},h_{1,2})\subseteq\cM_{1,2}$. Thus by the $L(-1)$-derivative property for intertwining operators, \eqref{eqn:commutator}, and \eqref{eqn:iterate}, we calculate
\begin{align*}
p\,&\partial_{x_2}^2\Phi  (x_1,x_2)  =\langle v',\cY_1(v,x_1)\cY_2(L(-2)v,x_2)v\rangle \\
 & =\sum_{k\geq 0} x_2^k\langle v', \cY_1(v,x_1)\omega_{-k-1}\cY_2(v,x_2)v\rangle+\sum_{k\geq 0} x_2^{-k-1}\langle v',\cY_1(v,x_1)\cY_2(v,x_2)\omega_k v\rangle\nonumber\\
 & =\sum_{k\geq 0} x_2^k\langle v',\omega_{-k-1}\cY_1(v,x_1)\cY_2(v,x_2)v\rangle\nonumber\\
 &\hspace{2em}-\sum_{k\geq 0}\sum_{i\geq 0}\binom{-k-1}{i} x_1^{-k-i-1} x_2^k\langle v',\cY_1(\omega_i v,x_1)\cY_2(v,x_2)v\rangle\nonumber\\
 &\hspace{2em}+\left(x_2^{-1}\langle v',\cY_1(v,x_1)\cY_2(v,x_2)L(-1)v\rangle+x_2^{-2}\langle v',\cY_1(v,x_1)\cY_2(v,x_2)L(0)v\rangle\right)\nonumber\\
 & = -\sum_{k\geq 0} x_1^{-k-1} x_2^k\langle v',\cY_1(L(-1)v,x_1)\cY_2(v,x_2)v\rangle\nonumber\\
 &\hspace{2em} +\sum_{k\geq 0} (k+1)x_1^{-k-2} x_2^k\langle v',\cY_1(L(0)v,x_1)\cY_2(v,x_2)v\rangle\nonumber\\
 &\hspace{2em} -x_2^{-1}\left(\langle v',\cY_1(v,x_1)\cY_2(L(-1)v,x_2)v\rangle +\langle v',\cY_1(L(-1)v,x_1)\cY_2(v,x_2)v\rangle\right)\nonumber\\
 &\hspace{2em}+h_{1,2} x_2^{-2}\langle v',\cY_1(v,x_1)\cY_2(v,x_2)v\rangle\nonumber\\
 & =-\left((x_1-x_2)^{-1}+x_2^{-1}\right)\partial_{x_1}\Phi(x_1,x_2)- x_2^{-1}\partial_{x_2}\Phi(x_1,x_2)\nonumber\\
 &\hspace{2em}+h_{1,2}\left(-\partial_{x_1}(x_1-x_2)^{-1}+x_2^{-2}\right)\Phi(x_1,x_2)\nonumber\\
 & =-\left(\frac{1}{x_1-x_2}+\frac{1}{x_2}\right)\partial_{x_1}\Phi(x_1,x_2)-\frac{1}{x_2}\partial_{x_2}\Phi(x_1,x_2)+h_{1,2}\left(\frac{1}{(x_1-x_2)^2}+\frac{1}{x_2^2}\right)\Phi(x_1,x_2).
\end{align*}
Now, the $L(0)$-conjugation formula for intertwining operators implies that
\begin{equation*}
 \Phi(x_1,x_2)=x_1^{-2h_{1,2}}\varphi(x_2/x_1).
\end{equation*}
Thus, using the branch of logarithm $\log 1=0$,
\begin{equation*}
 \Phi(1,x)=\varphi(x),\quad\partial_{x_2}\Phi(x_1,x_2)\vert_{x_1=1,x_2=x} =\varphi'(x),\quad\partial_{x_2}^2\Phi(x_1,x_2)\vert_{x_1=1,x_2=x} =\varphi''(x),
\end{equation*}
and
\begin{align*}
 \partial_{x_1}\Phi(x_1,x_2)\vert_{x_1=1,x_2=x} & = \left.\left(-2h_{1,2} x_1^{-2h_{1,2}-1}\varphi(x_2/x_1)-x_1^{-2h_{1,2}-2}x_2\varphi'(x_2/x_1)\right)\right\vert_{x_1=1,x_2=x}\nonumber\\
 & =-2h_{1,2}\varphi(x)-x\varphi'(x).
\end{align*}
Plugging these relations into the partial differential equation for $\Phi(x_1,x_2)$, we get
\begin{align*}
 p\,\varphi''(x) & =\left((1-x)^{-1}+x^{-1}\right)(2h_{1,2}\varphi(x)+x\varphi'(x))-x^{-1}\varphi'(x)+h_{1,2}\left((1-x)^{-2}+x^{-2}\right)\varphi(x)\nonumber\\
 & =\left(\frac{x}{1-x}+1-\frac{1}{x}\right)\varphi'(x)+h_{1,2}\left(\frac{2}{1-x}+\frac{2}{x}
 +\frac{1}{(1-x)^2}+\frac{1}{x^2}\right)\varphi(x),
\end{align*}
which means
\begin{align*}
 p x(1-x)\varphi_s''(x)=(2x-1)\varphi_s'(x)+\frac{h_{1,2}}{x(1-x)}\varphi_s(x)
\end{align*}
as required.

For $\psi(x)$, we consider the formal series
\begin{equation*}
\Psi(x_0,x_1) = \langle v', \cY^1(\cY^2(v,x_0)v,x_1-x_0)v\rangle.
\end{equation*}
So using the $L(-1)$-derivative and commutator formulas for intertwining operators,
\begin{align*}
 p\,\partial_{x_0}^2 & \Psi(x_0,x_1)\nonumber\\
 & = p\,\partial_{x_0}\left(\langle v',\cY^1(\cY^2(L(-1)v,x_0)v,x_1-x_0)v\rangle -\langle v',\cY^1(L(-1)\cY^2(v,x_0)v,x_1-x_0)v\rangle\right)\nonumber\\
 & = -p\,\partial_{x_0}\langle v',\cY^1(\cY^2(v,x_0)L(-1)v,x_1-x_0)v\rangle\nonumber\\
 & = p\,\langle v',\cY^1(\cY^2(v,x_0)L(-1)^2 v,x_1-x_0)v\rangle\nonumber\\
 & = \langle v',\cY^1(\cY^2(v,y_0)L(-2)v,y_2)v\rangle\vert_{y_0=x_0, y_2=x_1-x_0}.
\end{align*}
Then a calculation using \eqref{eqn:commutator} and \eqref{eqn:iterate} as before leads to
\begin{align*}
 p\,\partial_{x_0}^2\Psi(x_0,x_1) & = -y_0^{-1}\partial_{y_0}\Psi(y_0,y_2+y_0)\vert_{y_0=x_0,y_2=x_1-x_0} -x_2^{-1}\partial_{y_2}\Psi(y_0,y_2+y_0)\vert_{y_0=x_0,y_2=x_1-x_0}\nonumber\\
 &\hspace{4em} +h_{1,2}(y_0^{-2}+y_2^{-2})\Psi(y_0,y_2+y_0)\vert_{y_0=x_0, y_2=x_1-x_0}\nonumber\\
 & = -x_0^{-1}\partial_{x_0}\Psi(x_0,x_1) -\left( x_0^{-1}+(x_1-x_0)^{-1}\right)\partial_{x_1}\Psi(x_0,x_1) \nonumber\\
 &\hspace{4em}+h_{1,2}\left(x_0^{-2}+(x_1-x_0)^{-2}\right)\Psi(x_0,x_1).
\end{align*}
Now defining $\widetilde{\psi}(x)=\psi(1-x)$, the $L(0)$-conjugation property for intertwining operators implies
\begin{equation*}
\Psi(x_0,x_1) = x_1^{-2h_{1,2}}\widetilde{\psi}(x_0/x_1).
\end{equation*}
Then as before, we get
\begin{equation*}
 \Psi(x,1)=\widetilde{\psi}(x),\quad\partial_{x_0}\Psi(x_0,x_1)\vert_{x_0=x,x_1=1} =\widetilde{\psi}'(x),\quad\partial_{x_0}^2\Psi(x_0,x_1)\vert_{x_0=x,x_1=1} =\widetilde{\psi}''(x),
\end{equation*}
and
\begin{equation*}
\partial_{x_1}\Psi_s(x_0,x_1)\vert_{x_0=x,x_1=1} = -2h_{1,2}\widetilde{\psi}(x)-x\widetilde{\psi}'(x).
\end{equation*}
Plugging these relations into the partial differential equation for $\widetilde{\Psi}(x_0,x_1)$, we get
\begin{align*}
p\, \widetilde{\psi}''(x) & = -x^{-1}\widetilde{\psi}'(x) -(x^{-1}+(1-x)^{-1})(-2h_{1,2}\widetilde{\psi}(x)-x\widetilde{\psi_s}'(x)) \\
&\hspace{4em} +h_{1,2}(x^{-2}+(1-x)^{-2})\widetilde{\psi}(x)\\
& = \left(\frac{x}{1-x}+1-\frac{1}{x}\right)\widetilde{\psi}'(x)+h_{1,2}
\left(\frac{2}{1-x}+\frac{2}{x}+\frac{1}{(1-x)^2}+\frac{1}{x^2}\right)\widetilde{\psi}(x).
\end{align*}
This simplifies to
\begin{align*}
 p\,x(1-x)\widetilde{\psi}''(x)=(2x-1)\widetilde{\psi}'(x)+\frac{h_{1,2}}{x(1-x)}\widetilde{\psi}(x),
\end{align*}
so substituting $x\mapsto 1-x$ shows that $\psi(x)$ satisfies \eqref{eqn:original}.
\end{proof}

To solve the differential equation \eqref{eqn:original}, set
\[
g(x) = x^{-1/2p}(1-x)^{-1/2p}f(x).
\]
Then a straightforward calculation shows that  $g(x)$ satisfies the hypergeometric differential equation
\begin{equation}\label{hypergeo}
p\,x(1-x)g''(x)+2(1-2x)g'(x)+(1-3/p)g(x) = 0,
\end{equation}
as in \cite[Equation 4.83]{TW}, whose solutions on $U_1$ and $U_2$ can be obtained by the method of Frobenius and are well known.

For $p\geq 3$, the indicial equation of \eqref{hypergeo} has two distinct roots, and it follows that the solutions to \eqref{eqn:original} on $U_1$ have the basis:
\begin{align*}
 \varphi_1(x) & =x^{1/2p} (1-x)^{1/2p} \,_2F_1\left(\frac{1}{p}, \frac{3}{p}-1, \frac{2}{p};x\right)\in x^{1/2p}\CC[[x]],\\
\varphi_2(x) & = x^{1-3/2p} (1-x)^{1/2p}\,_2F_1\left(1-\frac{1}{p}, \frac{1}{p}, 2-\frac{2}{p};x\right)\in x^{1-3/2p}\CC[[x]].
\end{align*}
On $U_2$, \eqref{eqn:original} has basis of solutions
\begin{align*}
 \psi_1(x) & = x^{1/2p}(1-x)^{1/2p}\,_2F_1\left(\frac{1}{p}, \frac{3}{p}-1, \frac{2}{p};1-x\right)\in(1-x)^{1/2p}\CC[[1-x]],\\
\psi_2(x) & = x^{1/2p}(1-x)^{1-3/2p}\,_2F_1\left(1-\frac{1}{p}, \frac{1}{p}, 2-\frac{2}{p};1-x\right)\in(1-x)^{1-3/2p}\CC[[1-x]].
\end{align*}
As in \cite[Equation 4.85]{TW}, these two bases of solutions are related on the intersection
\begin{equation*}
U_1\cap U_2 =\lbrace z\in\CC\,\vert\,1>\vert z\vert>\vert 1-z\vert >0\rbrace
\end{equation*}
 by the hypergeometric function connection formulas:
\begin{align}
 \varphi_1(z) = \frac{1}{2\cos \frac{\pi}{p}} \psi_1(z) + \frac{3-p}{2-p}\frac{\Gamma(\frac{2}{p})^2}{\Gamma(\frac{1}{p})\Gamma(\frac{3}{p})} \psi_2(z) \label{eqn:p_conn_1}\\
 \psi_1(z) = \frac{1}{2\cos \frac{\pi}{p}} \varphi_1(z) + \frac{3-p}{2-p}\frac{\Gamma(\frac{2}{p})^2}{\Gamma(\frac{1}{p})\Gamma(\frac{3}{p})} \varphi_2(z)\label{eqn:p_conn_2}
\end{align}
when $1>\vert z\vert>\vert 1-z\vert >0$.

For $p=2$, on the other hand, the indicial equation of \eqref{hypergeo} has a repeated root, so one of the basis solutions for \eqref{eqn:original} on $U_1$ is logarithmic:
\begin{align*}
 \varphi_1(x) & = x^{1/4} (1-x)^{1/4} \,{}_2 F_1\left(\frac{1}{2},\frac{1}{2},1; x\right),\nonumber\\
 \varphi_2(x) & =\varphi_1(x)\log x +x^{1/4}(1-x)^{1/4}\,G(x)
\end{align*}
for a power series $G(x)$ with constant term $0$. As in the $p\geq 3$ case, the basis solutions on $U_2$ can be obtained via the substitution $x\mapsto 1-x$, that is, $\psi_i(x)=\varphi_i(1-x)$ for $i=1,2$. This time, we get
\begin{equation}\label{eqn:2_conn}
 \varphi_1(z) =\frac{\ln 4}{\pi}\psi_1(z) -\frac{1}{\pi} \psi_2(z)
\end{equation}
for $ 1>\vert z\vert>\vert 1-z\vert >0 $ (see \cite[Equation 4.92]{TW}).

\subsection{Rigidity of \texorpdfstring{$\cM_{1,2}$}{M_{1,2}}}

In this section, we prove that the simple $\cM(p)$-module $\cM_{1,2}$ is rigid. Since \eqref{more1} shows that the identity of $\cM_{1,2}\boxtimes\cM_{1,2}$ depends on $p$, the proof is divided into two cases.

\subsubsection{The case $p\geq 3$}

First we consider $p\geq 3$, in which case \eqref{more1} shows
\begin{equation*}
 \cM_{1,2}\boxtimes\cM_{1,2}\cong\cM_{1,1}\oplus\cM_{1,3}.
\end{equation*}
This direct sum decomposition means that for $s=1,3$, we have homomorphisms
\begin{equation*}
 i_s:\cM_{1,s}\rightarrow\cM_{1,2}\boxtimes\cM_{1,2},\qquad p_s: \cM_{1,2}\boxtimes\cM_{1,2}\rightarrow\cM_{1,s}
\end{equation*}
such that
\begin{equation*}
 p_s\circ i_s=\id_{\cM_{1,s}}
\end{equation*}
for $s=1,3$ and
\begin{equation*}
 i_1\circ p_1+i_3\circ p_3=\id_{\cM_{1,2}\boxtimes\cM_{1,2}}.
\end{equation*}
If we take $\cY_{2\boxtimes 2}$ to be the tensor product intertwining operator of type $\binom{\cM_{1,2}\boxtimes\cM_{1,2}}{\cM_{1,2}\,\cM_{1,2}}$, then
\begin{equation*}
 \cY_{22}^s=p_s\circ\cY_{2\boxtimes 2}
\end{equation*}
for $s=1,3$ is the unique (up to scale) non-zero intertwining operator of type $\binom{\cM_{1,s}}{\cM_{1,2}\,\cM_{1,2}}$. Moreover,
\begin{equation}\label{eqn:Y13}
 \cY_{2\boxtimes 2}=i_1\circ\cY^1_{22}+i_3\circ\cY^3_{22}.
\end{equation}
We may take $i_1$ and $p_1$ to be preliminary candidates for the coevaluation and evaluation, respectively, for $\cM_{1,2}$. Then to show $\cM_{1,2}$ is rigid, it is sufficient to show that the homomorphisms $f,g:\cM_{1,2}\rightarrow\cM_{1,2}$ defined by the commutative diagrams
\begin{equation}\label{diag:rigid_f}
\xymatrixcolsep{4pc}
 \xymatrix{
 \cM_{1,2} \ar[d]_{f} \ar[r]^(.4){r^{-1}_{\cM_{1,2}}} & \cM_{1,2}\boxtimes\cM_{1,1} \ar[r]^(.4){\id_{\cM_{1,2}}\boxtimes i_1} & \cM_{1,2}\boxtimes(\cM_{1,2}\boxtimes\cM_{1,2}) \ar[d]^{\cA_{\cM_{1,2},\cM_{1,2},\cM_{1,2}}} \\
 \cM_{1,2} & \ar[l]^(.6){l_{\cM_{1,2}}} \cM_{1,1}\boxtimes\cM_{1,2} & \ar[l]^(.6){p_1\boxtimes\id_{\cM_{1,2}}} (\cM_{1,2}\boxtimes\cM_{1,2})\boxtimes\cM_{1,2} \\
 }
\end{equation}
and
\begin{equation}\label{diag:rigid_g}
 \xymatrixcolsep{4pc}
 \xymatrix{
 \cM_{1,2} \ar[d]_{g} \ar[r]^(.4){l^{-1}_{\cM_{1,2}}} & \cM_{1,1}\boxtimes\cM_{1,2} \ar[r]^(.4){i_1\boxtimes\id_{\cM_{1,2}}} & (\cM_{1,2}\boxtimes\cM_{1,2})\boxtimes\cM_{1,2} \ar[d]^{\cA^{-1}_{\cM_{1,2},\cM_{1,2},\cM_{1,2}}} \\
 \cM_{1,2} & \ar[l]^(.6){r_{\cM_{1,2}}} \cM_{1,2}\boxtimes\cM_{1,1} & \ar[l]^(.6){\id_{\cM_{1,2}}\boxtimes p_1} \cM_{1,2}\boxtimes(\cM_{1,2}\boxtimes\cM_{1,2}) \\
 }
\end{equation}
are non-zero multiples of the identity. In fact, since $\cM_{1,2}$ is simple, it is sufficient to show that $f$ and $g$ are non-zero.

To prove $f\neq 0$, we need to show that the intertwining operator
\begin{align*}
 \cY^2_{21}& = l_{\cM_{1,2}}\circ(p_1\boxtimes\id_{\cM_{1,2}})\circ\cA_{\cM_{1,2},\cM_{1,2},\cM_{1,2}}\circ(\id_{\cM_{1,2}}\boxtimes i_1)\circ\cY_{2\boxtimes 1} \nonumber\\
 &=l_{\cM_{1,2}}\circ(p_1\boxtimes\id_{\cM_{1,2}})\circ\cA_{\cM_{1,2},\cM_{1,2},\cM_{1,2}}\circ\cY_{2\boxtimes(2\boxtimes 2)}\circ(\id_{\cM_{1,2}}\otimes i_1)
\end{align*}
of type $\binom{\cM_{1,2}}{\cM_{1,2}\,\cM_{1,1}}$ is non-zero. For this, it is sufficient to show that
\begin{equation*}
 \langle v',\cY^2_{21}(v,1)\cY^1_{22}(v,x)v\rangle\neq 0
\end{equation*}
for non-zero highest weight vectors $v\in\cM_{1,2}$, $v'\in\cM_{1,2}'\cong\cM_{1,2}$, and for some $x\in\RR$ such that $1>x>1-x>0$. We define the intertwining operator $\cY^2_{23}$ of type $\binom{\cM_{1,2}}{\cM_{1,2}\,\cM_{1,3}}$ similarly, using $i_3$ instead of $i_1$. Then we use \eqref{eqn:Y13} to calculate
\begin{align}\label{eqn:hyp_geo_expansion}
\langle v', \cY^2_{21}(v, & 1)  \cY^1_{22}(v,x)v\rangle +\langle v', \cY^2_{23}(v,1)\cY^3_{22}(v,x)v\rangle\nonumber\\
& = \left\langle v', \overline{l_{\cM_{1,2}}\circ(p_1\boxtimes\id_{\cM_{1,2}})\circ\cA_{\cM_{1,2},\cM_{1,2},\cM_{1,2}}}\left(\cY_{2\boxtimes(2\boxtimes 2)}(v,1)\cY_{2\boxtimes 2}(v,x)v\right)\right\rangle\nonumber\\
&  = \left\langle v', \overline{l_{\cM_{1,2}}\circ(p_1\boxtimes\id_{\cM_{1,2}})}\left(\cY_{(2\boxtimes 2)\boxtimes 2}(\cY_{2\boxtimes 2}(v,1-x)v,x)v\right)\right\rangle\nonumber\\
&  = \left\langle v', \overline{l_{\cM_{1,2}}}\left(\cY_{1\boxtimes 2}(\cY^1_{22}(v,1-x)v,x)v\right)\right\rangle\nonumber\\
&  = \langle v', Y_{\cM_{1,2}}(\cY^1_{22}(v,1-x)v,x)v\rangle.
\end{align}
Now by Proposition \ref{prop}, both sides of \eqref{eqn:hyp_geo_expansion} are solutions to \eqref{eqn:original}. Moreover, since
\begin{equation*}
 \cY_{22}^s(v,x)v\in x^{h_{1,s}-2h_{1,2}}\cM_{1,2}[[x]] =\left\lbrace\begin{array}{rcl}
    x^{1-3/2p}\cM_{1,2}[[x]] & \text{if} & s=1 \\
    x^{1/2p} \cM_{1,2}[[x]] & \text{if} & s=3                                                                \end{array}
\right. ,
\end{equation*}
we have
\begin{align*}
\langle v', \cY^2_{23}(v,1)\cY^3_{22}(v,x)v\rangle & = c_1\varphi_1(x)\\
 \langle v', \cY^2_{21}(v, 1)  \cY^1_{22}(v,x)v\rangle & = c_2\varphi_2(x)
\end{align*}
for $c_1,c_2\in\CC$, using the notation of the previous subsection.

We can be more precise about the series expansion of $\cY^1_{22}$: since $\cM_{1,2}$ is simple and self-contragredient, $\cY^1_{22}$ is necessarily related to $Y_{\cM_{1,2}}$ by the skew-symmetry and adjoint intertwining operators of \cite{HLZ2}. In particular,
\begin{equation*}
 \cY_{22}^1(v,x)v\in x^{-2h_{1,2}}\left(\langle e^{(2r+1)\pi i L(0)} v,v\rangle\one +x\CC[[x]]\right)
\end{equation*}
for some $r\in\ZZ$, where $\langle\cdot,\cdot\rangle$ is a non-degenerate invariant bilinear form on $\cM_{1,2}$. Then it follows that
\begin{equation*}
 \langle v', Y_{\cM_{1,2}}(\cY^1_{22}(v,1-x)v,x)v\rangle \in (1-x)^{-2h_{1,2}}\left(\langle v',v\rangle\langle e^{(2r+1)\pi i L(0)} v,v\rangle+(1-x)\CC[[1-x]]\right),
\end{equation*}
which means that if we choose $v$ and $v'$ to be non-zero, then the right side of \eqref{eqn:hyp_geo_expansion} is a non-zero multiple $d \psi_2(x)$. Now in the case $p\geq 4$, we can solve \eqref{eqn:p_conn_1} and \eqref{eqn:p_conn_2} to show that $c_2 =- \frac{d}{2\cos(\pi/p)}\neq 0$, which means $\cY^2_{21}\neq 0$, as required. For $p=3$, \eqref{eqn:p_conn_1} and \eqref{eqn:p_conn_2} show that $\varphi_1(x)=\psi_1(x)$ for $1>x>1-x>0$, so that \eqref{eqn:hyp_geo_expansion} amounts to
\begin{equation*}
\langle v', \cY^2_{21}(v,  1)  \cY^1_{22}(v,x)v\rangle = -c_1\psi_1(x)+d\psi_2(x).
\end{equation*}
Since $\psi_1(x)$, $\psi_2(x)$ are linearly independent on the interval $(\frac{1}{2},1)$ and since $d\neq 0$, we again conclude $\cY^2_{21}\neq 0$. This proves that $f$ in \eqref{diag:rigid_f} is non-zero. The proof that $g$ in \eqref{diag:rigid_g} is non-zero is similar, so $\cM_{1,2}$ is rigid when $p\geq 3$.

\subsubsection{The case $p=2$}

Now for $p=2$, $\cM_{1,2}=\cF_{\alpha_{1,2}}$ is an irreducible Fock module, and we have a surjective intertwining operator of type $\binom{\cF_{\alpha_{0,1}}}{\cM_{1,2}\,\cM_{1,2}}$, which induces a surjective homomorphism
\begin{equation*}
 p_1: \cM_{1,2}\boxtimes\cM_{1,2}\rightarrow\cF_{\alpha_{0,1}}\rightarrow\cM_{1,1}.
\end{equation*}
Since $\cM_{1,2}$ is irreducible, this means $\cM_{1,2}$ is self-contragredient, and the intertwining operator $\cY_{22}^1=p_1\circ\cY_{2\boxtimes 2}$ must be related to $Y_{\cM_{1,2}}$ by skew-symmetry and adjoints. In particular, if $v\in\cM_{1,2}$ is a lowest-conformal-weight vector, then
\begin{equation}\label{eqn:y22_nonzero}
 \cY_{22}^1(v,x)v =e^{-(2r+1)h_{1,2}}\langle v,v\rangle\one\,x^{-2h_{1,2}} +\ldots
\end{equation}
for some $r\in\ZZ$ and non-degenerate invariant bilinear form $\langle\cdot,\cdot\rangle$ on $\cM_{1,2}$. We can rescale $v$ and/or $p_1$ so that the coefficient of $x^{-2h_{1,2}}$ in $\cY_{22}^1(v,x)v$ is $\one$. Then let $\widetilde{\one}$ denote the coefficient of $x^{-2h_{1,2}}$ in $\cY_{2\boxtimes 2}(v,x)v$, so $p_1(\widetilde{\one}) =\one$. Since conformal weights of modules in $\cC_{\cM(2)}$ are bounded below by $-\frac{1}{4p}(p-1)^2 =-\frac{1}{8}$, the lowest power of $x$ from the coset $-2h_{1,2}+\ZZ$ occurring in $\cY_{2\boxtimes 2}(v,x)v$ is $x^{-2h_{1,2}}$.

We have not yet shown that $\cM_{1,2}\boxtimes\cM_{1,2}\cong\cP_{1,1}$ as in \eqref{more1}, but we will use $\cP_{1,1}$ to get a coevaluation $i_1:\cM_{1,1}\rightarrow\cM_{1,2}\boxtimes\cM_{1,2}$. From the structure of $\cP_{1,1}$ illustrated in \eqref{diag:P_structure}, we have an injection $i_2:\cM_{1,1}\rightarrow\cP_{1,1}$ and a surjection $p_2: \cP_{1,1}\rightarrow\cM_{1,1}$. The module $\cP_{1,1}$ has a two-dimensional lowest conformal weight space of weight $0$, and since $\cP_{1,1}$ is logarithmic, $(\cP_{1,1})_{[0]}$ has a basis $\lbrace\overline{\one},L(0)\overline{\one}\rbrace$ such that $p_2(\overline{\one})=\one$ and (rescaling $i_2$ if necessary) $i_2(\one)=L(0)\overline{\one}$. Recall from Corollary \ref{cor:first_proj_p} that $\cP_{1,1}$ is projective in $\cC_{\cM(2)}^0$.

Now let $M$ be the $\cM(2)$-submodule of $\cM_{1,2}\boxtimes\cM_{1,2}$ generated by $\widetilde{\one}$. Since
\begin{equation*}
 \cF_{W(2)}(M)\hookrightarrow\cF_{\cW(2)}(\cM_{1,2}\boxtimes\cM_{1,2})\cong\cW_{1,2}\boxtimes_{\cW(2)}\cW_{1,2}
\end{equation*}
and $\cF_{\cW(2)}(\cM_{1,1})\cong\cW_{1,1}$ are objects of $\cC_{\cW(2)}$, $p_1\vert_{M}: M\rightarrow\cM_{1,1}$ is a (surjective) homomorphism in $\cC_{\cM(2)}^0$.  Thus because $\cP_{1,1}$ is projective in $\cC_{\cM(2)}^0$, we get a homomorphism $q: \cP_{1,1}\rightarrow M\hookrightarrow\cM_{1,2}\boxtimes\cM_{1,2}$, and then $i_1: \cM_{1,1}\rightarrow\cM_{1,2}\boxtimes\cM_{1,2}$, such that the diagram
\begin{equation*}
\xymatrixcolsep{3pc}
\xymatrix{
\cM_{1,1} \ar[r]^{i_2} \ar[rd]_{i_1} & \cP_{1,1} \ar[d]_{q} \ar[rd]^{p_2} & \\
& \cM_{1,2}\boxtimes\cM_{1,2} \ar[r]_(.6){p_1} & \cM_{1,1}\\
 }
\end{equation*}
commutes. We want to show that $\cM_{1,2}$ is rigid self-dual with evaluation (a multiple of) $p_1$ and coevaluation $i_1$. As in the $p\geq 3$ case, it suffices to show that the maps $f$ and $g$ defined by the diagrams \eqref{diag:rigid_f} and \eqref{diag:rigid_g} are non-zero. Here we will focus on showing $g\neq 0$, since the proof for $f$ is similar.

We first show that $L(0)\widetilde{\one}=i_1(\one)$. Note that since $q(\overline{\one})=\widetilde{\one}+m$ for some $m\in\ker p_1\vert_{M_{[0]}}$, we have
\begin{equation*}
 L(0)\widetilde{\one} =q(L(0)\overline{\one})-L(0)m=i_1(\one)-L(0)m.
\end{equation*}
So it is sufficient to show that $L(0)\ker p_1\vert_{M_{[0]}} =0$. Note that because conformal weights of modules in $\cC_{\cM(2)}$ are bounded below by $-\frac{1}{8}$, $M_{[0]}$ is contained in the top level of $M$ and thus is a finite-dimensional module for the Zhu algebra $A(\cM(2))$ computed in \cite{A}. As an $A(\cM(2))$-module, $M_{[0]}$ is generated by $\widetilde{\one}$, and it is spanned by monomials $H(0)^i L(0)^j\widetilde{\one}$ since $\cM(2)$  is generated by $H$ and $\omega$, and since $[H]$, $[\omega]$ commute in $A(\cM(2))$. Moreover, since $[H]^2$ in $A(\cM(2))$ equals a polynomial in $[\omega]$, and since $L(0)$ is nilpotent on $M_{[0]}$, we may take $i\leq 1$ and $j\leq N$ for some $N\in\NN$. Note that $\ker p_1\vert_{M_{[0]}}$ is spanned by monomials $H(0)^i L(0)^j\widetilde{\one}$ with $i+j>0$, since the kernel has codimension $1$ and since $L(0)$, $H(0)$ act trivially on $(\cM_{1,1})_{[0]}=\CC\one$.

Now we can write
\begin{equation}\label{eqn:q_one}
 q(\overline{\one})=\sum_{i=0}^1\sum_{j=0}^{N} a_{i,j} H(0)^i L(0)^j\widetilde{\one}
\end{equation}
with $a_{0,0}=1$. By the relations
\begin{equation*}
 L(0)^2\overline{\one}=0=H(0)L(0)\overline{\one}
\end{equation*}
and $a_{0,0}=1$, successive applications of
\begin{equation*}
 H(0)L(0)^N, L(0)^N, H(0)L(0)^{N-1},\ldots, L(0)^2, H(0)L(0)
\end{equation*}
to both sides of \eqref{eqn:q_one} show that $H(0)^i L(0)^j\widetilde{\one}=0$ for $i+j>1$. This means that
\begin{equation*}
 \ker p_1\vert_{M_{[0]}} =\text{span}\lbrace L(0)\widetilde{\one}, H(0)\widetilde{\one}\rbrace
\end{equation*}
and $L(0)\ker p_1\vert_{M_{[0]}} =0$, as required.

Now we begin applying the composition in \eqref{diag:rigid_g} to a lowest-conformal-weight vector $v\in\cM_{1,2}$:
\begin{align*}
 (i_1\boxtimes\id_{\cM_{1,2}})\circ l_{\cM_{1,2}}^{-1}(v) =(i_1\boxtimes\id_{\cM_{1,2}})(\cY_{1\boxtimes 2}(\one,1)v ) =\cY_{(2\boxtimes 2)\boxtimes 2}(i_1(\one),1)v,
\end{align*}
where $\cY_{1\boxtimes 2}$ and $\cY_{(2\boxtimes 2)\boxtimes 2}$ are the relevant tensor product intertwining operators. But since $i_1(\one)=L(0)\widetilde{\one}$ is the coefficient of $x^{-2h_{1,2}}$ in $L(0)\cY_{2\boxtimes 2}(v,x)v$, and since  $x^{-2h_{1,2}}$ is the lowest power of $x$ from the coset $-2h_{1,2}+\ZZ$ occurring in $\cY_{2\boxtimes 2}(v,x)v$, we see that $(i_1\boxtimes\id_{\cM_{1,2}})\circ l_{\cM_{1,2}}^{-1}(v)$ is the coefficient of $(1-x)^{-2h_{1,2}}$ in the expansion of
\begin{equation*}
 \cY_{(2\boxtimes 2)\boxtimes 2}(L(0)\cY_{2\boxtimes 2}(v,1-x)v, x)v =\cY_{(2\boxtimes 2)\boxtimes 2}(L(0)\cY_{2\boxtimes 2}(v,1-x)v,1-(1-x))v
\end{equation*}
as a series in $1-x$. Defining the intertwining operator
\begin{equation*}
 \cY_{2\boxtimes 2, 2}^2=r_{\cM_{1,2}}\circ(\id_{\cM_{1,2}}\boxtimes p_1)\circ\cA_{\cM_{1,2},\cM_{1,2},\cM_{1,2}}^{-1}\circ\cY_{(2\boxtimes 2)\boxtimes 2}
\end{equation*}
of type $\binom{\cM_{1,2}}{\cM_{1,2}\boxtimes\cM_{1,2}\,\cM_{1,2}}$, we will thus get $g\neq 0$ provided the coefficient of $(1-x)^{-2h_{1,2}}$ in the series expansion of
\begin{equation}\label{eqn:p2_rigidity_crit}
\left \langle v', \cY_{2\boxtimes 2,2}^2\left(L(0)\cY_{2\boxtimes 2}(v,1-x)v,x\right)v\right\rangle
\end{equation}
is non-zero.

To analyze \eqref{eqn:p2_rigidity_crit}, we first replace $1$ and $x$ with $z_1$ and $z_2$, respectively,
and use the $L(0)$-commutator formula
\begin{equation}\label{vaccumimage}
L(0)\cY_{2\boxtimes 2}(v,z)v = 2h_{1,2}\cY_{2\boxtimes 2}(v,z)v+z\cY_{2\boxtimes 2}(L(-1)v,z)
\end{equation}
(see \cite[Equation 3.28]{HLZ2}) to calculate, for $\vert z_1\vert>\vert z_2\vert>\vert z_1-z_2\vert>0$,
\begin{align*}
& \left\langle v, \cY^2_{2\boxtimes 2,2}\left(L(0)\cY_{2\boxtimes 2}(v,z_1-z_2)v,z_2\right)v\right\rangle\nonumber\\
 &= 2h_{1,2}\left\langle v,\overline{r_{\cM_{1,2}}\circ(\id_{\cM_{1,2}}\boxtimes p_1)\circ\cA_{\cM_{1,2},\cM_{1,2},\cM_{1,2}}^{-1}}\left(\cY_{(2\boxtimes 2)\boxtimes 2}(\cY_{2\boxtimes 2}(v,z_1-z_2)v,z_2)v\right)\right\rangle\nonumber\\
 &\hspace{1.5em} +(z_1-z_2)\cdot\nonumber\\
 &\hspace{2.5em}\cdot\left\langle v,\overline{r_{\cM_{1,2}}\circ(\id_{\cM_{1,2}}\boxtimes p_1)\circ\cA_{\cM_{1,2},\cM_{1,2},\cM_{1,2}}^{-1}}\left(\cY_{(2\boxtimes 2)\boxtimes 2}(\cY_{2\boxtimes 2}(L(-1)v,z_1-z_2)v,z_2)v\right)\right\rangle\nonumber\\
 &  = 2h_{1,2}\left\langle v,\overline{r_{\cM_{1,2}}\circ(\id_{\cM_{1,2}}\boxtimes p_1)}\left(\cY_{2\boxtimes (2\boxtimes 2)}(v,z_1)\cY_{2\boxtimes 2}(v,z_2)v\right)\right\rangle\nonumber\\
 & \hspace{1.5em} +(z_1-z_2)\left\langle v,\overline{r_{\cM_{1,2}}\circ(\id_{\cM_{1,2}}\boxtimes p_1)}\left(\cY_{2\boxtimes (2\boxtimes 2)}(L(-1)v,z_1)\cY_{2\boxtimes 2}(v,z_2)v\right)\right\rangle\nonumber\\
 & = \left(2h_{1,2}+(z_1-z_2)\partial_{z_1}\right)\left\langle v,\overline{r_{\cM_{1,2}}}\left(\cY_{2\boxtimes 1}(v,z_1)\cY^1_{22}(v,z_2)v\right)\right\rangle\nonumber\\
 & =\left(2h_{1,2}+(z_1-z_2)\partial_{z_1}\right)\left\langle v,\Omega(Y_{\cM_{1,2}})(v,z_1)\cY^1_{22}(v,z_2)v\right\rangle,
\end{align*}
where $\Omega$ represents the skew-symmetry operation on vertex operators. Now, by the $L(0)$-conjugation formula for intertwining operators and the fact that $\cY_{22}^1$ is not a logarithmic intertwining operator,
\begin{align*}
 \left\langle v,\Omega(Y_{\cM_{1,2}})(v,z_1)\cY^1_{22}(v,z_2)v\right\rangle =d z_1^{-2h_{1,2}}\varphi_1(z_2/z_1)
\end{align*}
for some $d\in\CC$, using the notation of the previous subsection. It follows from \eqref{eqn:y22_nonzero} that $d\neq 0$. So now
\begin{align*}
 \left\langle v, \cY^2_{2\boxtimes 2,2}\left(L(0)\cY_{2\boxtimes 2}(v,z_1-z_2)v,z_2\right)v\right\rangle & = \left(2h_{1,2}+(z_1-z_2)\partial_{z_1}\right) \left(d z_1^{-2h_{1,2}}\varphi_1(z_2/z_1)\right)\nonumber\\
 & \hspace{-5em}=d\,z_2\left(2 h_{1,2} z_1^{-2h_{1,2}-1} \varphi_1(z_2/z_1)-(z_1-z_2)z_1^{-2h_{1,2}-2}\varphi_1'(z_2/z_1)\right).
\end{align*}
Setting $z_1=1$ and $z_2=x$ for $x\in\RR$, $1>x>1-x>0,$ and using \eqref{eqn:2_conn}, we have
\begin{align*}
 &\left\langle v, \cY^2_{2\boxtimes 2,2}\left(L(0)\cY_{2\boxtimes 2}(v,1-x)v,x\right)v\right\rangle = d\,x\left(2h_{1,2}-(1-x)\dfrac{d}{dx}\right)\varphi_1(x)\nonumber\\
 &= x\left(2 h_{1,2}-(1-x)\dfrac{d}{dx}\right)(c_1\psi_1(x)+c_2\psi_2(x))\nonumber\\
 & =x\left(2 h_{1,2}-(1-x)\dfrac{d}{dx}\right)\left((c_1+c_2\log(1-x))\varphi_1(1-x)+c_2 (1-x)^{-2h_{1,2}}x^{-2h_{1,2}}G(1-x)\right)
\end{align*}
where $c_1=d\ln 4/\pi$ and $c_2=-d/\pi\neq 0$. We need to expand the right side as a series in $1-x$ and show that the coefficient of $(1-x)^{-2h_{1,2}}$ is non-zero.

Let us introduce the notation
\begin{align*}
 (1-x)^{-2h_{1,2}}\widetilde{\varphi}_1(1-x) & = \varphi_1(1-x)\nonumber\\
 (1-x)^{-2h_{1,2}}\widetilde{G}(1-x) & = (1-x)^{-2h_{1,2}} x^{-2h_{1,2}} G(1-x),
\end{align*}
so that $\widetilde{\varphi}_1(1-x)$, $\widetilde{G}(1-x)$ are power series in $1-x$. Thus we have
\begin{align*}
 &\left\langle v, \cY^2_{2\boxtimes 2,2}\left(L(0)\cY_{2\boxtimes 2}(v,1-x)v,x\right)v\right\rangle \nonumber\\
 &\hspace{.5em}=x\left(2 h_{1,2}-(1-x)\dfrac{d}{dx}\right)(1-x)^{-2h_{1,2}}\left((c_1+c_2\log(1-x))\widetilde{\varphi}_1(1-x)+c_2 \widetilde{G}(1-x)\right)\nonumber\\
 &\hspace{.5em} = -x(1-x)^{-2h_{1,2}+1}\dfrac{d}{dx}\left((c_1+c_2\log(1-x))\widetilde{\varphi}_1(1-x)+c_2 \widetilde{G}(1-x)\right)\nonumber\\
 &\hspace{.5em} =x(1-x)^{-2h_{1,2}+1}\left( c_2(1-x)^{-1}\widetilde{\varphi}_1(1-x)+(c_1+c_2\log(1-x))\widetilde{\varphi}_1'(1-x)+c_2\widetilde{G}'(1-x)\right)\nonumber\\
 &\hspace{.5em} \in (1-x)^{-2h_{1,2}}\big(c_2\widetilde{\varphi}_1(1-x)+(1-x)\CC[[1-x]][\log(1-x)]\big).
\end{align*}
Since $c_2\neq 0$ and $\widetilde{\varphi}_1(1-x) =x^{-2h_{1,2}}\,_2 F_1\left(\frac{1}{2},\frac{1}{2},1;1-x\right)$ is a power series in $1-x$ with constant term $1$, it follows that the coefficient of $(1-x)^{-2h_{1,2}}$ here is indeed non-zero. This completes the proof that $g\neq 0$. As mentioned previously, the proof that $f\neq 0$ is similar: it uses \eqref{eqn:2_conn} with $\psi$ and $\varphi$ solutions exchanged, which is valid because the basis solutions to \eqref{eqn:original} in $U_1\cap\RR$ and $U_2\cap\RR$ are related by the substitution $x\mapsto 1-x$. Thus $\cM_{1,2}$ is rigid and self-dual when $p=2$.

\subsubsection{The $s=p$ case of \eqref{more1}}\label{subsubsec:s=p}

Now that we know $\cM_{1,2}$ is rigid, we use the projectivity of $\cP_{r,p-1}$ and $\cM_{r,p}$ in $\cC_{\cM(p)}^0$ (recall Corollary \ref{cor:first_proj_p}) to prove the fusion rule \eqref{more1} for $s=p$. Since the tensor product of a rigid object with a projective object in any tensor category is projective (see for example Corollary 2 in the Appendix of \cite{KL5}), $\cM_{1,2}\boxtimes\cM_{r,p}$ is projective in $\cC_{\cM(p)}^0$. We also have an intertwining operator of type $\binom{\cM_{r,p-1}}{\cM_{1,2}\,\cM_{r,p}}$ obtained by restricting and projecting a non-zero intertwining operator among Fock modules. This intertwining operator is surjective at least as long as $r\leq 0$, so we get a surjective homomorphism $p_1: \cM_{1,2}\boxtimes\cM_{r,p}\rightarrow\cM_{r,p-1}$.

Letting $p_2:\cP_{r,p-1}\rightarrow\cM_{r,p-1}$ denote a surjective homomorphism, projectivity of $\cM_{1,2}\boxtimes\cM_{r,p}$ in $\cC_{\cM(p)}^0$ implies the existence of a homomorphism
\begin{equation*}
q: \cM_{1,2}\boxtimes\cM_{r,p}\rightarrow\cP_{r,p-1}
\end{equation*}
such that the diagram
\begin{equation*}
\xymatrix{
  \cP_{r,p-1}  \ar[rd]^{p_2} & \\
 \cM_{1,2}\boxtimes\cM_{r,p} \ar[u]^{q} \ar[r]_(.6){p_1} & \cM_{r,p-1} \\
 }
\end{equation*}
commutes. As $(\cP_{r,p-1},p_2)$ is a projective cover of $\cM_{r,p-1}$ by Proposition \ref{prop:p-1_proj_cover}, $q$ is surjective. Then projectivity of $\cP_{r,p-1}$ means that $\cP_{r,p-1}$ is a direct summand of $\cM_{1,2}\boxtimes\cM_{r,p}$:
\begin{equation*}
 \cM_{1,2}\boxtimes\cM_{r,p}\cong\cP_{r,p-1}\oplus Q
\end{equation*}
for some submodule $Q$. But if we induce, we get
\begin{align*}
 \cR_{\bar{r},p-1} & \cong\cW_{1,2}\boxtimes_{\cW(p)} \cW_{\bar{r},p}\cong\cF_{\cW(p)}(\cM_{1,2})\boxtimes_{\cW(p)}\cF_{\cW(p)}(\cM_{r,p})\nonumber\\
 &\cong \cF_{\cW(p)}(\cM_{1,2}\boxtimes\cM_{r,p}) \cong \cF_{\cW(p)}(\cP_{r,p-1})\oplus\cF_{\cW(p)}(Q)\cong \cR_{\bar{r},p-1}\oplus\cF_{\cW(p)}(Q).
\end{align*}
Since $\cR_{\bar{r},p-1}$ is indecomposable and non-zero, it follows that $\cF_{\cW(p)}(Q)=0$, and then $Q=0$ as well. So
\begin{equation*}
 \cM_{1,2}\boxtimes\cM_{r,p}\cong\cP_{r,p-1}.
\end{equation*}
This proves the $s=p$, $r\leq 0$ case of \eqref{more1}, and the $r>0$ case then follows from \eqref{simplewithatypical} and \eqref{simplewithprojective}.

\subsection{Rigidity for simple modules}

We now prove that all simple modules $\cM_{r,s}$ in the tensor category $\cC_{\cM(p)}$ are rigid, with duals given by the contragredient modules of \cite{FHL}.

Recall that for a vertex operator algebra $V$ and a grading-restricted generalized $V$-module $W=\bigoplus_{h\in\CC} W_{[h]}$, the contragredient is a module structure on the graded dual $W'=\bigoplus_{h\in\CC} W_{[h]}^*$. If $\cC$ is a category of grading-restricted $V$-modules closed under contragredients, then contragredients define a contravariant functor, with the contragredient of a morphism $f: W_1\rightarrow W_2$ defined by
\begin{equation*}
 \langle f'(w_2'),w_1\rangle =\langle w_2', f(w_1)\rangle
\end{equation*}
for $w_1\in W_1$, $w_2'\in W_2'$. Moreover, for any module $W$ in $\cC$, we have a natural isomorphism $\delta_W: W\rightarrow W''$ defined by
\begin{equation*}
 \langle \delta_W(w),w'\rangle =\langle w',w\rangle
\end{equation*}
for $w\in W$, $w'\in W'$. It is straightforward to prove:
\begin{prop}\label{prop:contra_exact}
 The contragredient functor on $\cC$ is exact, that is, if a sequence of $V$-modules $W_1\xrightarrow{f} W_2\xrightarrow{g} W_3$ is exact at $W_2$, then $W_3'\xrightarrow{f'} W_2'\xrightarrow{g'} W_1'$ is exact at $W_2'$.
\end{prop}

We use this proposition to prove:
\begin{prop}\label{prop:C_has_contras}
 The category $\cC_{\cM(p)}$ of $\cM(p)$-modules is closed under contragredients.
\end{prop}
\begin{proof}
 Since every module in $\cC_{\cM(p)}$ has finite length, we can use induction on length. For the base case, we need to show that the contragredient of $\cM_{r,s}$ for $r\in\ZZ$, $1\leq s\leq p$ is isomorphic to some $\cM_{r',s'}$. Since $\cM_{r,s}'$ is an irreducible $\cM(p)$-module by \cite[Proposition 5.3.2]{FHL}, and since the lowest conformal weight of $\cM_{r,s}'$ agrees with that of $\cM_{r,s}$, the only other possibility is that $\cM_{r,s}'\cong\cF_\lambda$ for some $\lambda\notin L^\circ$ such that $h_\lambda=h_{\alpha_{r,s}}$ or $h_\lambda=h_{\alpha_{2-r,s}}$ according as $r\geq 1$ or $r\leq 0$. However, we can see from \eqref{conformalgrading} and \eqref{eqn:h_rs} that $h_\lambda=h_{r,s}$ for some $r\geq 1$, $1\leq s\leq p$ only if
 \begin{equation*}
  \lambda =\left((p-1)\pm(pr-s)\right)\frac{\alpha}{2p}\in L^\circ.
 \end{equation*}
So in fact $\cF_\lambda\cong\cM_{r,s}'$ for $\lambda\notin L^\circ$ is impossible, proving the base case of the induction.

For the inductive step, suppose $M$ is a module in $\cC_{\cM(p)}$ with length $\ell(M)\geq 2$. Then we have an exact sequence
\begin{equation*}
 0\rightarrow N_1\rightarrow M\rightarrow N_2\rightarrow 0
\end{equation*}
with $N_1$, $N_2$ modules in $\cC_{\cM(p)}$ such that $\ell(N_1),\ell(N_2)<\ell(M)$. By Proposition \ref{prop:C_has_contras}, the sequence of $\cM(p)$-modules
\begin{equation*}
 0\rightarrow N_2'\rightarrow M'\rightarrow N_1'\rightarrow 0
\end{equation*}
is also exact, and by the inductive hypothesis, $N_1'$, $N_2'$ are modules in $\cC_{\cM(p)}$. Thus $M'$, as an extension of two modules in $\cC_{\cM(p)}$, is also a module in $\cC_{\cM(p)}$.
\end{proof}

\begin{remark}\label{rem:s_self_contra}
 By analyzing conformal weights using \eqref{eqn:h_rs} again, we can see that $h_{1,s}=h_{r',s'}$ for $r'\geq 1$, $1\leq s'\leq p$ if and only if $r'=1$, $s'=s$. This forces $\cM_{1,s}'\cong\cM_{1,s}$ for all $1\leq s\leq p$; in particular, $\cM(p)=\cM_{1,1}$ is a self-contragredient vertex operator algebra.
\end{remark}

When a category $\cC$ of grading-restricted generalized modules for a self-contragredient vertex operator algebra $V$ is closed under contragredients and has vertex algebraic braided tensor category structure as in \cite{HLZ1}-\cite{HLZ8}, we have an alternative characterization of the contragredient functor (see for example \cite{CKM2}). In particular, the isomorphisms of intertwining operator spaces
\begin{equation*}
 \cV_{V, W}^W\rightarrow\cV_{W,V}^W\rightarrow\cV_{W,W'}^{V'}\rightarrow\cV_{W',W}^V\rightarrow\Hom_V(W'\boxtimes W,V)
\end{equation*}
for a $V$-module $W$ in $\cC$ allow us to identify the vertex operator $Y_W$ with an evaluation homomorphism $e_W: W'\boxtimes W\rightarrow V$. The pair $(W',e_W)$ satisfies the following universal property: for any homomorphism $f: X\boxtimes W\rightarrow V$ in $\cC$, there is a unique homomorphism $\varphi: X\rightarrow W'$ such that the diagram
\begin{equation}\label{diag:contra_univ_prop}
 \xymatrixcolsep{4pc}
 \xymatrix{
 X\boxtimes W \ar[rd]^{f} \ar[d]_{\varphi\boxtimes\id_W} & \\
 W'\boxtimes W \ar[r]_{e_W} & V \\
 }
\end{equation}
commutes. This universal property allows us to characterize contragredient homomorphisms and the ribbon isomorphisms $\delta_W$ in terms of commutative diagrams:
\begin{lemma}\label{lem:contra}
 Suppose $\cC$ is a braided tensor category of modules for a self-contragredient vertex operator algebra that is closed under contragredients. Then we can choose the evaluations $e_W: W'\boxtimes W\rightarrow V$ for modules $W$ in $\cC$ in such a way that:
 \begin{itemize}
  \item[(1)] For any homomorphism $f: W_1\rightarrow W_2$ in $\cC$, $f': W_2'\rightarrow W_1'$ is the unique homomorphism such that the diagram
  \begin{equation}\label{diag:f'_def}
 \xymatrixcolsep{4pc}
 \xymatrix{
 W_2'\boxtimes W_1 \ar[r]^{\id_{W_2'}\boxtimes f} \ar[d]_{f'\boxtimes\id_{W_1}} & W_2'\boxtimes W_2 \ar[d]^{e_{W_2}} \\
 W_1'\boxtimes W_1 \ar[r]_{e_{W_1}} & V \\
 }
\end{equation}
commutes.

\item[(2)] For any module $W$ in $\cC$, $\delta_W: W\rightarrow W''$ is the unique homomorphism such that the diagram
\begin{equation}\label{diag:delta_def}
  \xymatrixcolsep{6pc}
  \xymatrix{
  W\boxtimes W' \ar[d]_{\delta_W\boxtimes\id_{W'}} \ar[r]^{\cR_{W,W'}\circ(\theta_W\boxtimes\id_{W'})} & W'\boxtimes W \ar[d]^{e_W} \\
W'' \boxtimes W' \ar[r]_{e_{W'}} & V \\
  }
 \end{equation}
commutes, where $\theta_W=e^{2\pi i L(0)}$ is the twist on $W$.
 \end{itemize}
\end{lemma}

\begin{proof}
 Fix an isomorphism $\kappa: V'\rightarrow V$ and for a module $W$ in $\cC$, let $\cE_W\in\cV_{W',W}^V$ denote the intertwining operator corresponding to a choice of $e_W$, that is,
 \begin{equation*}\
  \cE_W=e_W\circ\cY_{W',W}
 \end{equation*}
where $\cY_{W',W}\in\cV_{W',W}^{W'\boxtimes W}$ is the tensor product intertwining operator. Let us choose
\begin{equation*}
 \cE_W=\kappa\circ\Omega_0(A_0(\Omega(Y_W))),
\end{equation*}
using the notation of \cite{HLZ2}, for all modules $W$ in $\cC$. A calculation shows that this means $\cE_W$ satisfies
\begin{equation}\label{eqn:cE_char}
 \langle v', \cE_W(w',x)w\rangle =\langle e^{-x^{-1} L(1)} w', Y_W(e^{xL(1)}\widetilde{\kappa}(v'),x^{-1})e^{-xL(1)} e^{-\pi i L(0)} x^{-2 L(0)} w\rangle
\end{equation}
for $v'\in V'$, $w'\in W'$, and $w\in W$, where $\widetilde{\kappa}=\delta_V^{-1}\circ\kappa': V'\rightarrow V$.

Now for a $V$-homomorphism $f: W_1\rightarrow W_2$ in $\cC$, we use the definitions, \eqref{eqn:cE_char}, and properties of $V$-module homomorphisms to calculate
\begin{align*}
 &\left\langle v', (e_{W_1}\circ(f'\boxtimes\id_{W_1}))  \cY_{W_2',W_1}(w_2',x)w_1\right\rangle  =\left\langle v',\cE_{W_1}(f'(w_2'),x)w_1\right\rangle\nonumber\\
 &\qquad\qquad\qquad = \left\langle e^{-x^{-1} L(1)} f'(w_2'), Y_{W_1}(e^{xL(1)}\widetilde{\kappa}(v'),x^{-1})e^{-xL(1)} e^{-\pi i L(0)} x^{-2 L(0)} w_1\right\rangle\nonumber\\
 &\qquad\qquad\qquad =\left\langle e^{-x^{-1} L(1)} w_2', Y_{W_2}(e^{xL(1)}\widetilde{\kappa}(v'),x^{-1})e^{-xL(1)} e^{-\pi i L(0)} x^{-2 L(0)} f(w_1)\right\rangle\nonumber\\
 &\qquad\qquad\qquad =\left\langle v',\cE_{W_2}(w_2',x)f(w_1)\right\rangle\nonumber\\
 &\qquad\qquad\qquad = \left\langle v',\left(e_{W_2}\circ(\id_{W_2'}\boxtimes f)\right)\cY_{W_2',W_1}(w_2',x)w_1\right\rangle
\end{align*}
for $v'\in V'$, $w_2'\in W_2'$, and $w_1\in W_1$. Since the tensor product intertwining operator $\cY_{W_2',W_1}$ is surjective, this shows that \eqref{diag:f'_def} commutes.

Now for a module $W$ in $\cC$, we use the definitions and \eqref{eqn:cE_char} to calculate
\begin{align*}
 &\left\langle v', (e_W\circ\cR_{W,W'}\circ(\theta_W\boxtimes\id_{W'}))\cY_{W,W'}(w,x)w'\right\rangle \nonumber\\
 &\qquad\qquad\qquad = \left\langle v', e_W\left(e^{xL(-1)}\cY_{W',W}(w',e^{\pi i} x)e^{2\pi i L(0)} w\right)\right\rangle\nonumber\\
 &\qquad\qquad\qquad = \left\langle e^{xL(1)} v',\cE_W(w',e^{\pi i}x) e^{2\pi i L(0)} w\right\rangle\nonumber\\
 &\qquad\qquad\qquad =\left\langle e^{x^{-1}L(1)} w', Y_W(\widetilde{\kappa}(v'), -x^{-1}) e^{xL(1)} e^{-\pi i L(0)} x^{-2L(0)} w\right\rangle
\end{align*}
for $v'\in V'$, $w\in W$, and $w'\in W'$. On the other hand, the definitions and conjugation formulas from \cite[Section 5]{FHL} involving $L(0)$ and $L(1)$ imply
\begin{align*}
 \langle v', (e_{W'}\circ & (\delta_W\boxtimes\id_{W'}))  \cY_{W,W'}(w,x)w'\rangle = \langle v', \cE_{W'}(\delta_W(w),x)w'\rangle\nonumber\\
 & =\left\langle e^{-x^{-1}} \delta_W(w), Y_{W'}(e^{xL(1)} \widetilde{\kappa}(v'),x^{-1})e^{-xL(1)} e^{-\pi i L(0)} x^{-2L(0)}w'\right\rangle\nonumber\\
 & =\left\langle Y_{W'}(e^{xL(1)}\widetilde{\kappa}(v'),x^{-1})e^{-xL(1)}e^{-\pi i L(0)} x^{-2L(0)} w', e^{-x^{-1}L(1)}w\right\rangle\nonumber\\
 & =\left\langle e^{-xL(1)}e^{-\pi i L(0)} x^{-2L(0)} w', Y_W(e^{x^{-1} L(1)}(-x^2)^{L(0)} e^{xL(1)}\widetilde{\kappa}(v'),x)e^{-x^{-1}L(1)}w\right\rangle\nonumber\\
 & =\left\langle e^{-\pi i L(0)} x^{-2L(0)}e^{x^{-1}L(1)} w', Y_W((-x^2)^{L(0)}\widetilde{\kappa}(v'),x) e^{-x^{-1} L(1)}w\right\rangle\nonumber\\
 & =\left\langle e^{x^{-1}L(1)} w', Y_W(\widetilde{\kappa}(v'), -x^{-1}) e^{xL(1)} e^{-\pi i L(0)} x^{-2L(0)} w\right\rangle.
\end{align*}
Since $\cY_{W,W'}$ is surjective, this shows that \eqref{diag:delta_def} commutes.
\end{proof}

From now on, we assume that we have chosen the evaluations so that the conclusion of the above lemma holds.

If $W$ is a rigid object in a tensor category, then its dual $W^*$ satisfies the same universal property of the contragredient, with the morphism $\varphi$ of \eqref{diag:contra_univ_prop} given by the composition
\begin{equation*}
 X\xrightarrow{r_X^{-1}} X\boxtimes V\xrightarrow{\id_X\boxtimes i_W} X\boxtimes(W\boxtimes W')\xrightarrow{\cA_{X,W,W'}} (X\boxtimes W)\boxtimes W'\xrightarrow{f\boxtimes\id_{W'}} V\boxtimes W'\xrightarrow{l_{W'}} W'.
\end{equation*}
Thus if $\cC$ is a braided tensor category of modules for a self-contragredient vertex operator algebra $V$ that is closed under contragredients, the dual of a rigid object in $\cC$ is necessarily its contragredient. The following two lemmas are standard, and we will use them to prove that simple modules in the category $\cC_{\cM(p)}$ of $\cM(p)$-modules are rigid:
\begin{lemma}\label{lem:prod_of_rigid}
 If $W_1$ and $W_2$ are rigid objects in a tensor category, then $W_1\boxtimes W_2$ is also rigid (with dual $W_2^*\boxtimes W_1^*$).
\end{lemma}
\begin{lemma}\label{lem:summand_of_rigid}
 If $\cC$ is a braided tensor category of modules for a self-contragredient vertex operator algebra $V$ that is closed under contragredients, then any direct summand of a rigid module in $\cC$ is rigid.
\end{lemma}
\begin{proof}
 Suppose $W$ is a rigid module in $\cC$ and we have homomorphisms $q: M\rightarrow W$, $p: W\rightarrow M$ such that $p\circ q=\id_M$ for some module $M$ in $\cC$. Then it is straightforward to show using \eqref{diag:f'_def} that $M$ is rigid with dual $M'$, evaluation $e_M=e_W\circ(p'\boxtimes q)$, and coevaluation $i_M=(p\boxtimes q')\circ i_W$.
\end{proof}

We can now prove:
\begin{theorem}\label{thm:C_simple_rigid}
 Every simple module $\cM_{r,s}$ in $\cC_{\cM(p)}$ is rigid.
\end{theorem}
\begin{proof}
 Since the modules $\cM_{2n+1,1}$, $n\in\ZZ$, and $\cM_{2,1}$ are simple currents, they are rigid. Then for any $n\in\ZZ$, the module $\cM_{2n,1}\cong\cM_{2,1}\boxtimes\cM_{2n-1,1}$ (from \eqref{more}) is rigid by Lemma \ref{lem:prod_of_rigid}. Next, the rigidity of $\cM_{1,2}$ from the previous subsection, the fusion rules \eqref{more1} for $2\leq s\leq p-1$, and Lemma \ref{lem:summand_of_rigid} imply that $\cM_{1,s}$ is rigid for each $1\leq s\leq p$. Finally, the fusion rule $\cM_{r,s}\cong\cM_{r,1}\boxtimes\cM_{1,s}$ from \eqref{simplewithatypical} and \eqref{more} together with Lemma \ref{lem:prod_of_rigid} imply that every $\cM_{r,s}$ is rigid.
\end{proof}

\subsection{Rigidity for finite-length modules}

In this subsection, we extend rigidity of simple modules in $\cC_{\cM(p)}$ to finite-length modules. This will be a consequence of the following more general theorem that we expect will also be useful for future examples:
\begin{theorem}\label{thm:rigid_simple_to_fin_len}
 Assume that $V$ is a self-contragredient vertex operator algebra and $\cC$ is a category of grading-restricted generalized $V$-modules such that:
 \begin{itemize}
 \item[(1)] The category $\cC$ is closed under submodules, quotients, and contragredients, and every module in $\cC$ has finite length.

  \item[(2)] The category $\cC$ has braided tensor category structure as in \cite{HLZ8}.

  \item[(3)] Every simple module in $\cC$ is rigid.
 \end{itemize}
Then $\cC$ is a rigid tensor category.
\end{theorem}

As we have verified the conditions of the theorem in Theorem \ref{thm:C_tens_cat}, Proposition \ref{prop:C_has_contras}, and Theorem \ref{thm:C_simple_rigid}, we get:
\begin{cor}
 The tensor category $\cC_{\cM(p)}$ is rigid; moreover, it is a braided ribbon tensor category with natural twist isomorphism $\theta=e^{2\pi i L(0)}$.
\end{cor}

\begin{cor}
 The category $\cC_{\cM(p)}^0$ is closed under contragredients; in particular, $\cC_{\cM(p)}^0$ is a braided ribbon tensor category.
\end{cor}
\begin{proof}
 Suppose that $M$ is a module in $\cC_{\cM(p)}^0$, so that $\cF_{\cW(p)}(M)$ is in $\cC_{\cW(p)}$. Since $M'$ is a dual of $M$ in $\cC_{\cM(p)}$, it is easy to see that $\cF_{\cW(p)}(M')$ is a dual of $\cF_{\cW(p)}(M)$ in $\Rep\,\cW(p)$ (see for example \cite[Proposition 2.77]{CKM1}). But since $\cC_{\cW(p)}$ is a rigid tensor category, $\cF_{\cW(p)}(M)'$ is already a dual of $\cF_{\cW(p)}(M)$. Since duals are unique, $\cF_{\cW(p)}(M')\cong\cF_{\cW(p)}(M)'$ is an object of $\cC_{\cW(p)}$, and therefore $M'$ is an object of $\cC_{\cM(p)}^0$.
\end{proof}

To prove Theorem \ref{thm:rigid_simple_to_fin_len}, we will use the $V$-homomorphisms
\begin{equation*}
 \Phi_{W_1,W_2}: W_2'\boxtimes W_1'\rightarrow(W_1\boxtimes W_2)'
\end{equation*}
for modules $W_1$, $W_2$ in $\cC$ induced by the universal property of contragredients such that the diagram
\begin{equation*}
 \xymatrixcolsep{4.5pc}
 \xymatrix{
  W_2'\boxtimes (W_1'\boxtimes(W_1\boxtimes W_2)) \ar[r]^(.48){\id_{W_2'}\boxtimes\cA_{W_1',W_1,W_2}} & W_2'\boxtimes((W_1'\boxtimes W_1)\boxtimes W_2) \ar[rd]^(.6){\quad\id_{W_2'}\boxtimes(e_{W_1}\boxtimes\id_{W_2})}  & \\
  (W_2'\boxtimes W_1')\boxtimes(W_1\boxtimes W_2) \ar[u]_{\cA_{W_2',W_1',W_1\boxtimes W_2}^{-1}} \ar[d]^{\Phi_{W_1,W_2}\boxtimes\id_{W_1\boxtimes W_2}} & & W_2'\boxtimes(V\boxtimes W_2) \ar[d]_{\id_{W_2'}\boxtimes l_{W_2}} \\
  (W_1\boxtimes W_2)'\boxtimes(W_1\boxtimes W_2) \ar[r]_(.65){e_{W_1\boxtimes W_2}} & V & W_2'\boxtimes W_2 \ar[l]^{e_{W_2}} \\
 }
\end{equation*}
commutes. These homomorphisms determine a natural transformation:
\begin{prop}\label{prop:Phi_nat}
 For any homomorphisms $f_1: W_1\rightarrow X_1$ and $f_2: W_2\rightarrow X_2$ in $\cC$, the diagram
 \begin{equation*}
  \xymatrixcolsep{4pc}
  \xymatrix{
  X_2'\boxtimes X_1' \ar[r]^{f_2'\boxtimes f_1'} \ar[d]^{\Phi_{X_1,X_2}} & W_2'\boxtimes W_1' \ar[d]^{\Phi_{W_1,W_2}} \\
  (X_1\boxtimes X_2)' \ar[r]_{(f_1\boxtimes f_2)'} & (W_1\boxtimes W_2)' \\
  }
 \end{equation*}
commutes.
\end{prop}
\begin{proof}
 By the universal property of $((W_1\boxtimes W_2)', e_{W_1\boxtimes W_2})$, it is enough to show that
\begin{equation*}
 e_{W_1\boxtimes W_2}\circ([(f_1\boxtimes f_2)'\circ\Phi_{X_1,X_2}]\boxtimes\id_{W_1\boxtimes W_2}) = e_{W_1\boxtimes W_2}\circ([\Phi_{W_1,W_2}\circ(f_2'\boxtimes f_1')]\boxtimes\id_{W_1\boxtimes W_2}).
\end{equation*}
By \eqref{diag:f'_def},  the left side of this equation is
\begin{align}\label{eqn:Phi_nat}
   e_{X_1\boxtimes X_2}\circ(\Phi_{X_1,X_2}\boxtimes(f_1\boxtimes f_2)).
\end{align}
The right side is the composition
\begin{align*}
 (X_2'\boxtimes & X_1')  \boxtimes(W_1\boxtimes W_2)\xrightarrow{(f_2'\boxtimes f_1')\boxtimes\id_{W_1\boxtimes W_2}} (W_2'\boxtimes W_1')\boxtimes(W_1\boxtimes W_2)\nonumber\\ & \xrightarrow{\cA_{W_2',W_1',W_1\boxtimes W_2}^{-1}} W_2'\boxtimes(W_1'\boxtimes(W_1\boxtimes W_2)) \xrightarrow{\id_{W_2'}\boxtimes\cA_{W_1',W_1,W_2}} W_2'\boxtimes((W_1'\boxtimes W_1)\boxtimes W_2)\nonumber\\
 &  \xrightarrow{\id_{W_2'}\boxtimes(e_{W_1}\boxtimes\id_{W_2})} W_2'\boxtimes(V\boxtimes W_2)\xrightarrow{\id_{W_2'}\boxtimes l_{W_2}} W_2'\boxtimes W_2\xrightarrow{e_{W_2}} V.
\end{align*}
It is then straightforward to use the naturality of the associativity and unit isomorphisms applied to $f_1'$ and $f_2'$, the diagram \eqref{diag:f'_def} for $f_1'$ and $f_2'$, naturality of associativity and unit applied to $f_1$ and $f_2$, and the definition of $\Phi_{X_1,X_2}$ to reduce this to \eqref{eqn:Phi_nat}.
\end{proof}

If $W_1$ and $W_2$ are rigid objects of $\cC$, then $\Phi_{W_1,W_2}$ is an isomorphism with inverse given by the composition
\begin{align*}
 (W_1\boxtimes W_2)' & \xrightarrow{r_{(W_1\boxtimes W_2)'}^{-1}} (W_1\boxtimes W_2)'\boxtimes V\xrightarrow{\id_{(W_1\boxtimes W_2)'}\boxtimes i_{W_1}} (W_1\boxtimes W_2)'\boxtimes(W_1\boxtimes W_1')\nonumber\\
 &\xrightarrow{\id_{(W_1\boxtimes W_2)'}\boxtimes(r_{W_1}^{-1}\boxtimes \id_{W_1'})} (W_1\boxtimes W_2)'\boxtimes((W_1\boxtimes V)\boxtimes W_1')\nonumber\\
 &\xrightarrow{\id_{(W_1\boxtimes W_2)'}\boxtimes((\id_{W_1}\boxtimes i_{W_2})\boxtimes\id_{W_1'})} (W_1\boxtimes W_2)'\boxtimes((W_1\boxtimes(W_2\boxtimes W_2'))\boxtimes W_1')\nonumber\\
& \xrightarrow{assoc.} ((W_1\boxtimes W_2)'\boxtimes(W_1\boxtimes W_2))\boxtimes(W_2'\boxtimes W_1')\nonumber\\
& \xrightarrow{e_{W_1\boxtimes W_2}\boxtimes\id_{W_2'\boxtimes W_1'}} V\boxtimes(W_2'\boxtimes W_1')\xrightarrow{l_{W_2'\boxtimes W_1'}} W_2'\boxtimes W_1',
\end{align*}
where the arrow marked $assoc.$ represents a suitable composition of associativity isomorphisms. Conversely, we have:
\begin{prop}\label{prop:Phi_iso_rigid}
 If $\Phi_{W,W'}$ is an isomorphism, then $W$ is rigid.
\end{prop}
\begin{proof}
 Recall the ribbon isomorphism $\delta_W: W\rightarrow W''$ characterized by the commutative diagram \eqref{diag:delta_def}. Let us use $\widetilde{e}_W$ to denote $$e_{W'}\circ(\delta_W\boxtimes\id_{W'}): W\boxtimes W'\rightarrow V.$$
We also define a homomorphism $\varphi: V\rightarrow V'$ such that the diagram
\begin{equation*}
 \xymatrixcolsep{4pc}
 \xymatrix{
 V\boxtimes V \ar[rd]^{l_V=r_V} \ar[d]_{\varphi\boxtimes\id_V} & \\
 V'\boxtimes V \ar[r]_{e_V} & V \\
 }
\end{equation*}
commutes. Now assuming $\Phi_{W,W'}$ is an isomorphism, we define a coevaluation $i_W: V\rightarrow W\boxtimes W'$ by the composition
\begin{align*}
 V\xrightarrow{\varphi} V'\xrightarrow{\widetilde{e}_W'} (W\boxtimes W')'\xrightarrow{\Phi_{W,W'}^{-1}} W''\boxtimes W'\xrightarrow{\delta_W^{-1}\boxtimes\id_{W'}} W\boxtimes W'.
\end{align*}
To prove rigidity, we need to show that the endomorphism $\mathfrak{R}_W$ of $W$ defined by the composition
\begin{equation*}
 W\xrightarrow{l_W^{-1}} V\boxtimes W\xrightarrow{i_W\boxtimes\id_W} (W\boxtimes W')\boxtimes W\xrightarrow{\cA_{W,W',W}^{-1}} W\boxtimes(W'\boxtimes W)\xrightarrow{\id_W\boxtimes e_W} W\boxtimes V\xrightarrow{r_W} W
\end{equation*}
and the endomorphism $\mathfrak{R}_{W'}$ of $W'$ defined by the composition
\begin{equation*}
 W'\xrightarrow{r_{W'}^{-1}} W'\boxtimes V\xrightarrow{\id_{W'}\boxtimes i_W} W'\boxtimes(W\boxtimes W')\xrightarrow{\cA_{W',W,W'}} (W'\boxtimes W)\boxtimes W'\xrightarrow{e_W\boxtimes\id_{W'}} V\boxtimes W'\xrightarrow{l_{W'}} W'
\end{equation*}
are identities. Since $\delta_W$ is an isomorphism, $(W,\widetilde{e}_W)$ satisfies the universal property of the contragredient of $W'$. So $\mathfrak{R}_W=\id_W$ is equivalent to
\begin{equation}\label{eqn:RW_id}
 \widetilde{e}_W\circ(\mathfrak{R}_W\boxtimes\id_{W'})=\widetilde{e}_W.
\end{equation}
Similarly, $\mathfrak{R}_{W'}=\id_{W'}$ is equivalent to
\begin{equation}\label{eqn:RW'_id}
 e_W\circ(\mathfrak{R}_{W'}\boxtimes\id_W)=e_W.
\end{equation}
Here, we give the details for \eqref{eqn:RW'_id} only, since the proof of \eqref{eqn:RW_id} is similar but simpler.

We first insert $\cR_{V,W'}\circ\cR_{V,W'}^{-1}$ into $\mathfrak{R}_{W'}$ and apply naturality of the braiding isomorphisms:
\begin{align*}
 W' & \xrightarrow{r_{W'}^{-1}} W'\boxtimes V\xrightarrow{\cR_{V,W'}^{-1}} V\boxtimes W'\xrightarrow{(\Phi_{W,W'}^{-1}\circ\widetilde{e}_W'\circ\varphi)\boxtimes\id_{W'}} (W''\boxtimes W')\boxtimes W'\nonumber\\ &\xrightarrow{(\delta_W^{-1}\boxtimes\id_{W'})\boxtimes\id_{W'}} (W\boxtimes W')\boxtimes W'\xrightarrow{\cR_{W\boxtimes W',W'}} W'\boxtimes(W\boxtimes W') \xrightarrow{\cA_{W',W,W'}} (W'\boxtimes W)\boxtimes W'\nonumber\\
 &\xrightarrow{e_W\boxtimes\id_{W'}} V\boxtimes W'\xrightarrow{l_{W'}} W'.
\end{align*}
Next we use $\cR_{V,W}^{-1}\circ r_{W'}^{-1}=l_{W'}^{-1}$, the hexagon axiom, and the naturality of braiding and associativity:
\begin{align*}
 & W'  \xrightarrow{l_{W'}^{-1}} V\boxtimes W' \xrightarrow{(\Phi_{W,W'}^{-1}\circ\widetilde{e}_W'\circ\varphi)\boxtimes\id_{W'}} (W''\boxtimes W')\boxtimes W'\xrightarrow{\cA_{W'',W',W'}^{-1}} W''\boxtimes(W'\boxtimes W')\nonumber\\
 & \xrightarrow{\id_{W''}\boxtimes\cR_{W',W'}} W''\boxtimes(W'\boxtimes W')\xrightarrow{\cA_{W'',W',W'}} (W''\boxtimes W')\boxtimes W'\nonumber\\
 &\xrightarrow{\cR_{W'',W'}\boxtimes\id_{W'}} (W'\boxtimes W'')\boxtimes W' \xrightarrow{(\id_{W'}\boxtimes\delta_W^{-1})\boxtimes\id_{W'}} (W'\boxtimes W)\boxtimes W'\xrightarrow{e_W\boxtimes\id_{W'}} V\boxtimes W'\xrightarrow{l_{W'}} W'.
\end{align*}
Using \eqref{diag:delta_def} and the definitions of $\theta$ and contragredient homomorphisms, we then calculate
\begin{align*}
 e_W\circ(\id_{W'}\boxtimes\delta_W^{-1})\circ\cR_{W'',W'} & =e_{W'}\circ(\theta_{W''}^{-1}\boxtimes\id_{W'})\nonumber\\
 & = e_{W'}\circ((\theta_{W'}^{-1})'\boxtimes\id_{W'}) = e_{W'}\circ(\id_{W''}\boxtimes\theta_{W'}^{-1}).
\end{align*}
Returning this to the composition and using naturality, we see that $\mathfrak{R}_{W'}$ is the composition
\begin{align*}
 W' & \xrightarrow{l_{W'}^{-1}\circ\theta_{W'}^{-1}} V\boxtimes W' \xrightarrow{(\Phi_{W,W'}^{-1}\circ\widetilde{e}_W'\circ\varphi)\boxtimes\id_{W'}} (W''\boxtimes W')\boxtimes W'\xrightarrow{\cA_{W'',W',W'}^{-1}} W''\boxtimes(W'\boxtimes W')\nonumber\\
 & \xrightarrow{\id_{W''}\boxtimes\cR_{W',W'}} W''\boxtimes(W'\boxtimes W')\xrightarrow{\cA_{W'',W',W'}} (W''\boxtimes W')\boxtimes W' \xrightarrow{e_{W'}\boxtimes\id_{W'}} V\boxtimes W'\xrightarrow{l_{W'}} W'.
\end{align*}

Now we begin to analyze $e_W\circ(\mathfrak{R}_{W'}\boxtimes\id_W)$. We first use
\begin{equation*}
 l_{W'}^{-1}\boxtimes \id_W=\cA_{V,W',W}\circ(l_{W'\boxtimes W}^{-1})
\end{equation*}
and naturality of associativity to get
\begin{align*}
 W' &\boxtimes W  \xrightarrow{\theta_{W'}^{-1}\boxtimes\id_W} W'\boxtimes W\xrightarrow{l_{W'\boxtimes W}^{-1}} V\boxtimes(W'\boxtimes W)\nonumber\\ &\xrightarrow{(\Phi_{W,W'}^{-1}\circ\widetilde{e}_W'\circ\varphi)\boxtimes\id_{W'\boxtimes W}} (W''\boxtimes W')\boxtimes(W'\boxtimes W)\xrightarrow{\cA_{W''\boxtimes W',W',W}} ((W''\boxtimes W')\boxtimes W')\boxtimes W\nonumber\\
 & \xrightarrow{\cA_{W'',W',W'}^{-1}\boxtimes\id_W} (W''\boxtimes(W'\boxtimes W'))\boxtimes W\xrightarrow{(\id_{W''}\boxtimes\cR_{W',W'})\boxtimes\id_W} (W''\boxtimes(W'\boxtimes W'))\boxtimes W\nonumber\\
 & \xrightarrow{\cA_{W'',W',W'}\boxtimes\id_W} ((W''\boxtimes W')\boxtimes W')\boxtimes W \xrightarrow{(e_{W'}\boxtimes\id_{W'})\boxtimes\id_W} (V\boxtimes W')\boxtimes W\nonumber\\ & \xrightarrow{l_{W'}\boxtimes\id_W} W'\boxtimes W\xrightarrow{e_W} V.
\end{align*}
Next, we apply the pentagon axiom to the first two associativity arrows and apply naturality of associativity to $\cR_{W',W'}$:
\begin{align*}
 W' &\boxtimes W  \xrightarrow{\theta_{W'}^{-1}\boxtimes\id_W} W'\boxtimes W\xrightarrow{l_{W'\boxtimes W}^{-1}} V\boxtimes(W'\boxtimes W)\nonumber\\ &\xrightarrow{(\Phi_{W,W'}^{-1}\circ\widetilde{e}_W'\circ\varphi)\boxtimes\id_{W'\boxtimes W}} (W''\boxtimes W')\boxtimes(W'\boxtimes W)\xrightarrow{\cA_{W'',W',W'\boxtimes W}^{-1}} W''\boxtimes(W'\boxtimes(W'\boxtimes W))\nonumber\\
 &\xrightarrow{\id_{W''}\boxtimes\cA_{W',W',W}} W''\boxtimes((W'\boxtimes W')\boxtimes W)\xrightarrow{\id_{W''}\boxtimes(\cR_{W',W'}\boxtimes\id_W)} W''\boxtimes((W'\boxtimes W')\boxtimes W)\nonumber\\
 & \xrightarrow{\cA_{W'',W'\boxtimes W',W}} (W''\boxtimes(W'\boxtimes W'))\boxtimes W\xrightarrow{\cA_{W'',W',W'}\boxtimes\id_W} ((W''\boxtimes W')\boxtimes W')\boxtimes W\nonumber\\
 & \xrightarrow{(e_{W'}\boxtimes\id_{W'})\boxtimes\id_W} (V\boxtimes W')\boxtimes W \xrightarrow{l_{W'}\boxtimes\id_W} W'\boxtimes W\xrightarrow{e_W} V.
\end{align*}
Now we insert $\cR_{W,W'}\circ\cR_{W,W'}^{-1}$ between the first two associativity isomorphisms and apply various naturalities to the inverse braiding:
\begin{align*}
 W' &\boxtimes W\xrightarrow{\cR_{W,W'}^{-1}} W\boxtimes W'\xrightarrow{\id_W\boxtimes\theta_{W'}^{-1}} W\boxtimes W'\xrightarrow{l_{W\boxtimes W'}^{-1}} V\boxtimes(W\boxtimes W')\nonumber\\
 &\xrightarrow{(\Phi_{W,W'}^{-1}\circ\widetilde{e}_W'\circ\varphi)\boxtimes\id_{W\boxtimes W'}} (W''\boxtimes W')\boxtimes(W\boxtimes W')\xrightarrow{\cA_{W'',W',W\boxtimes W'}^{-1}} W''\boxtimes(W'\boxtimes(W\boxtimes W'))\nonumber\\
 & \xrightarrow{\id_{W''}\boxtimes(\id_{W'}\boxtimes\cR_{W,W'})} W''\boxtimes(W'\boxtimes(W'\boxtimes W)) \xrightarrow{\id_{W''}\boxtimes\cA_{W',W',W}} W''\boxtimes((W'\boxtimes W')\boxtimes W)\nonumber\\
 &\xrightarrow{\id_{W''}\boxtimes(\cR_{W',W'}\boxtimes\id_W)} W''\boxtimes((W'\boxtimes W')\boxtimes W) \xrightarrow{\cA_{W'',W'\boxtimes W',W}} (W''\boxtimes(W'\boxtimes W'))\boxtimes W\nonumber\\
 & \xrightarrow{\cA_{W'',W',W'}\boxtimes\id_W} ((W''\boxtimes W')\boxtimes W')\boxtimes W \xrightarrow{(e_{W'}\boxtimes\id_{W'})\boxtimes\id_W} (V\boxtimes W')\boxtimes W\nonumber\\ &\xrightarrow{l_{W'}\boxtimes\id_W} W'\boxtimes W\xrightarrow{e_W} V.
\end{align*}
We apply the hexagon axiom and then the pentagon axiom to the braiding and associativity isomorphisms in rows three, four, and five:
\begin{align}\label{eqn:RWprime_id_calc}
 W' &\boxtimes W\xrightarrow{\cR_{W,W'}^{-1}} W\boxtimes W'\xrightarrow{\id_W\boxtimes\theta_{W'}^{-1}} W\boxtimes W'\xrightarrow{l_{W\boxtimes W'}^{-1}} V\boxtimes(W\boxtimes W')\nonumber\\
 &\xrightarrow{(\Phi_{W,W'}^{-1}\circ\widetilde{e}_W'\circ\varphi)\boxtimes\id_{W\boxtimes W'}} (W''\boxtimes W')\boxtimes(W\boxtimes W')\xrightarrow{\cA_{W'',W',W\boxtimes W'}^{-1}} W''\boxtimes(W'\boxtimes(W\boxtimes W'))\nonumber\\
 & \xrightarrow{\id_{W''}\boxtimes\cA_{W',W,W'}} W''\boxtimes((W'\boxtimes W)\boxtimes W') \xrightarrow{\id_{W''}\boxtimes\cR_{W'\boxtimes W,W'}} W''\boxtimes(W'\boxtimes(W'\boxtimes W))\nonumber\\
 & \xrightarrow{\cA_{W'',W',W'\boxtimes W}} (W''\boxtimes W')\boxtimes(W'\boxtimes W)\xrightarrow{\cA_{W''\boxtimes W',W',W}} ((W''\boxtimes W')\boxtimes W')\boxtimes W\nonumber\\
 &\xrightarrow{(e_{W'}\boxtimes\id_{W'})\boxtimes\id_W} (V\boxtimes W')\boxtimes W\xrightarrow{l_{W'}\boxtimes \id_W} W'\boxtimes W\xrightarrow{e_W} V.
\end{align}
Now by naturality of associativity and properties of the left unit isomorphisms, the last six arrows reduce to
\begin{align*}
 W'' \boxtimes & ((W'\boxtimes W)\boxtimes W') \xrightarrow{\id_{W''}\boxtimes\cR_{W'\boxtimes W,W'}} W''\boxtimes(W'\boxtimes(W'\boxtimes W))\nonumber\\
 & \xrightarrow{\cA_{W'',W',W'\boxtimes W}} (W''\boxtimes W')\boxtimes(W'\boxtimes W)\xrightarrow{e_{W'}\boxtimes\id_{W'\boxtimes W}} V\boxtimes(W'\boxtimes W)\nonumber\\
 &\xrightarrow{l_{W'\boxtimes W}} W'\boxtimes W\xrightarrow{e_W} V.
\end{align*}
Applying naturality of unit, associativity, and braiding isomorphisms to $e_W$, we get
\begin{align*}
 W'' \boxtimes & ((W'\boxtimes W)\boxtimes W') \xrightarrow{ \id_{W''}\boxtimes(e_{W}\boxtimes\id_{W'})} W''\boxtimes(V\boxtimes W')\xrightarrow{\id_{W''}\boxtimes\cR_{V,W'}} W''\boxtimes(W'\boxtimes V)\nonumber\\
& \xrightarrow{\cA_{W'',W',V}} (W''\boxtimes W')\boxtimes V\xrightarrow{e_{W'}\boxtimes\id_V} V\boxtimes V\xrightarrow{l_V=r_V} V.
\end{align*}
Using naturality of the right unit isomorphism, and then using the relations between unit, associativity, and braiding isomorphisms, we get
\begin{align*}
 W''\boxtimes((W'\boxtimes W)\boxtimes W')\xrightarrow{\id_{W''}\boxtimes(e_W\boxtimes\id_{W'})} W''\boxtimes(V\boxtimes W')\xrightarrow{\id_{W''}\boxtimes l_{W'}} W''\boxtimes W'\xrightarrow{e_{W'}} V.
\end{align*}
Returning this to \eqref{eqn:RWprime_id_calc} and using the definition of $\Phi_{W,W'}$, we now find that $e_W\circ(\mathfrak{R}_{W'}\boxtimes\id_W)$ equals the composition
\begin{align*}
  W' \boxtimes W\xrightarrow{\cR_{W,W'}^{-1}} & W\boxtimes W'\xrightarrow{\id_W\boxtimes\theta_{W'}^{-1}} W\boxtimes W'\xrightarrow{l_{W\boxtimes W'}^{-1}} V\boxtimes(W\boxtimes W')\nonumber\\
 &\xrightarrow{(\widetilde{e}_W'\circ\varphi)\boxtimes\id_{W\boxtimes W'}} (W\boxtimes W')'\boxtimes(W\boxtimes W')\xrightarrow{e_{W\boxtimes W'}} V.
\end{align*}
The diagram \eqref{diag:f'_def} applied to $\widetilde{e}_W$ then yields
\begin{align*}
 W' \boxtimes W\xrightarrow{\cR_{W,W'}^{-1}}  W\boxtimes W'\xrightarrow{\id_W\boxtimes\theta_{W'}^{-1}} & W\boxtimes W'\xrightarrow{l_{W\boxtimes W'}^{-1}} V\boxtimes(W\boxtimes W')\nonumber\\
 &\xrightarrow{\varphi\boxtimes\id_{W\boxtimes W'}} V'\boxtimes(W\boxtimes W')\xrightarrow{\id_{V'}\boxtimes\widetilde{e}_W} V'\boxtimes V\xrightarrow{e_V} V.
\end{align*}
Then by the definition of $\varphi$ and naturality of the left unit isomorphisms, this composition simplifies to
\begin{equation*}
 W'\boxtimes W\xrightarrow{\cR_{W,W'}^{-1}} W\boxtimes W'\xrightarrow{\id_W\boxtimes\theta_{W'}^{-1}} W\boxtimes W'\xrightarrow{\widetilde{e}_W} V.
\end{equation*}
Finally, we find
\begin{align*}
 \widetilde{e}_W\circ(\id_W\boxtimes\theta_{W'}^{-1})\circ\cR_{W,W'}^{-1} & = e_{W'}\circ(\delta_W\boxtimes\id_{W'})\circ(\id_W\boxtimes\theta_{W'}^{-1})\circ\cR_{W,W'}^{-1}\nonumber\\
 & = e_W\circ(\theta_{W'}^{-1}\boxtimes\theta_W) = e_W
\end{align*}
using \eqref{diag:delta_def} and $\theta_{W'}^{-1} =(\theta_W^{-1})'$. This completes the proof that $\mathfrak{R}_{W'}=\id_{W'}$.
\end{proof}

Now we finish the proof of Theorem \ref{thm:rigid_simple_to_fin_len}:
\begin{proof}
 By Proposition \ref{prop:Phi_iso_rigid}, it is sufficient to prove that $\Phi_{W_1,W_2}$ is an isomorphism for all modules $W_1$ and $W_2$ in $\cC$. Since every module $W$ in $\cC$ has finite length $\ell(W)$, we can prove $\Phi_{W_1,W_2}$ is an isomorphism by induction on $\ell(W_1)+\ell(W_2)$. For the base case $\ell(W_1)=\ell(W_2)=1$, $W_1$ and $W_2$ are both simple and thus rigid by assumption. So $\Phi_{W_1,W_2}$ is an isomorphism by the discussion preceding Proposition \ref{prop:Phi_iso_rigid}.

 Now assume $\ell(W_1)+\ell(W_2)=N>2$ and $\Phi_{\tilW_1,\tilW_2}$ is an isomorphism whenever $\ell(\tilW_1)+\ell(\tilW_2)<N$. Without loss of generality, assume that $\ell(W_1)\geq 2$, so that there is an exact sequence
 \begin{equation*}
  0\rightarrow M\xrightarrow{g} W_1\xrightarrow{f} N\rightarrow 0
 \end{equation*}
with $M,N$ modules in $\cC$ (because $\cC$ is closed under submodules and quotients) and $\ell(M),\ell(N)<\ell(W_1)$. The diagram
\begin{equation*}
 \xymatrixcolsep{4pc}
 \xymatrix{
 & W_2'\boxtimes N' \ar[r]^{\id_{W_2'}\boxtimes f'} \ar[d]_{\Phi_{N,W_2}} & W_2'\boxtimes W_1' \ar[r]^{\id_{W_2'}\boxtimes g'} \ar[d]^{\Phi_{W_1,W_2}} & W_2'\boxtimes M' \ar[r] \ar[d]^{\Phi_{M,W_2}} & 0 \\
 0 \ar[r] & (N\boxtimes W_2)' \ar[r]_{(f\boxtimes\id_{W_2})'} & (W_1\boxtimes W_2)' \ar[r]_{(g\boxtimes\id_{W_2})'} & (M\boxtimes W_2)' &\\
 }
\end{equation*}
has exact rows because $W_2'\boxtimes\bullet$, $\bullet\boxtimes W_2$ are right exact and because the contragredient functor is exact. The diagram also commutes by Proposition \ref{prop:Phi_nat}. Since $\Phi_{N,W_2}$ and $\Phi_{M,W_2}$ are isomorphisms by the inductive hypothesis, the short five lemma diagram chase shows that $\Phi_{W_1,W_2}$ is also an isomorphism.
\end{proof}

\begin{remark}\label{rem:O1_rigid}
 Using Theorem \ref{thm:rigid_simple_to_fin_len}, we can also get rigidity for the category $\cO_1$ of $C_1$-cofinite grading-restricted generalized modules for the simple Virasoro vertex operator algebra $L(1,0)$ of central charge $1$. Conditions (1) and (2) of the theorem  were proved for this category in \cite{CJORY}, while it was shown in \cite{McR1} that all simple modules in $\cO_1$ are rigid.
\end{remark}

\begin{remark}
 Using the natural isomorphism $\Phi$ from the proof of Theorem \ref{thm:rigid_simple_to_fin_len}, we can determine the contragredients of all simple modules in $\cC_{\cM(p)}$. We have already seen in Remark \ref{rem:s_self_contra} that each $\cM_{1,s}$ is self-contragredient, and the fusion rule $\cM_{r,1}\boxtimes\cM_{2-r,1}\cong\cM_{1,1}$ for $r\in\ZZ$ shows that $\cM_{r,1}'\cong\cM_{2-r,1}$. Then we get
 \begin{equation*}
  \cM_{r,s}' \cong(\cM_{r,1}\boxtimes\cM_{1,s})'\cong\cM_{1,s}'\boxtimes\cM_{r,1}'\cong\cM_{1,s}\boxtimes\cM_{2-r,1}\cong\cM_{2-r,s}
 \end{equation*}
for $r\in\ZZ$, $1\leq s\leq p$.
\end{remark}

\section{Projective modules and fusion rules}\label{sec:proj}

In this final section, we construct projective covers of the remaining irreducible modules in $\cC_{\cM(p)}^0$, and we determine all fusion products involving irreducible modules and their projective covers.

\subsection{Construction of projective covers}\label{subsec:projectivity}

In Corollary \ref{cor:first_proj_p} and Proposition \ref{prop:p-1_proj_cover}, we proved that for $r\in\ZZ$,
\begin{itemize}
\item The module $\cP_{r,p} = \cF_{\a_{r,p}} = \cM_{r,p}$ is projective and is its own projective cover.
\item The module $\cP_{r, p-1} \cong \cM_{1,2}\btimes \cM_{r,p}$ is projective and is a projective cover of $\cM_{r, p-1}$.
\end{itemize}
In the proof of Proposition \ref{prop:p-1_proj_cover}, we also showed that $\cP_{r,p-1}$ has a unique maximal proper submodule $\cZ_{r,p-1}$, and we have the following non-split exact sequences:
\begin{equation}\label{zrp-1}
0 \rightarrow \cM_{r,p-1} \rightarrow \cZ_{r,p-1} \rightarrow \cM_{r-1, 1}\oplus \cM_{r+1, 1} \rightarrow 0,
\end{equation}
and
\begin{equation}\label{prp-1}
0 \rightarrow \cZ_{r,p-1} \rightarrow \cP_{r,p-1} \rightarrow \cM_{r,p-1} \rightarrow 0.
\end{equation}
Moreover, the Loewy diagrams of $\cP_{r,p-1}$ and $\cZ_{r,p-1}$ are
\[
\begin{tikzpicture}[->,>=latex,scale=1.5]
\node (b1) at (1,0) {$\cM_{r,p-1}$};
\node (c1) at (-1, 1){$\cP_{r, p-1}$:};
   \node (a1) at (0,1) {$\cM_{r-1, 1}$};
   \node (b2) at (2,1) {$\cM_{r+1, 1}$};
    \node (a2) at (1,2) {$\cM_{r,p-1}$};
\draw[] (b1) -- node[left] {} (a1);
   \draw[] (b1) -- node[left] {} (b2);
    \draw[] (a1) -- node[left] {} (a2);
    \draw[] (b2) -- node[left] {} (a2);
\end{tikzpicture},\;\;\;
\begin{tikzpicture}[->,>=latex,scale=1.5]
\node (b1) at (1,0) {$\cM_{r,p-1}.$};
\node (c1) at (-1, 1){$\cZ_{r, p-1}$:};
   \node (a1) at (0,1) {$\cM_{r-1, 1}$};
   \node (b2) at (2,1) {$\cM_{r+1, 1}$};
  \draw[] (b1) -- node[left] {} (a1);
   \draw[] (b1) -- node[left] {} (b2);
    \end{tikzpicture}
\]
We will also need the following lemma:
\begin{lemma}
 For $r\in\ZZ$, $\cM_{r,p}$ is injective in $\cC_{\cM(p)}^0$.
\end{lemma}
\begin{proof}
 Because $\cM_{2-r,p}$ is projective in $\cC_{\cM(p)}^0$, $\cM_{r,p}\cong\cM_{2-r,p}'$ is injective.
\end{proof}

Now for $p\geq 3$, we begin constructing more projective modules using the fusion rules \eqref{more1}. Since $\cM_{1,2}$ is rigid, the functor $\cM_{1,2}\btimes \bullet$ is exact. Applying $\cM_{1,2} \btimes \bullet $ to \eqref{zrp-1} and using \eqref{more1},
\begin{equation}
0 \rightarrow \cM_{r,p-2}\oplus \cM_{r,p} \rightarrow \cM_{1,2}\btimes \cZ_{r,p-1} \rightarrow \cM_{r-1, 2}\oplus \cM_{r+1, 2} \rightarrow 0
\end{equation}
is exact. Because $\cM_{r,p}$ is injective in $\cC_{\cM(p)}^0$, it is a direct summand of $\cM_{1,2}\btimes \cZ_{r,p-1}$. Let $\cZ_{r,p-2}$ be a submodule complement of $\cM_{r,p}$ in $\cM_{1,2}\btimes \cZ_{r,p-1}$, that is,
\begin{equation}\label{decompzrp-1}
\cM_{1,2}\btimes \cZ_{r,p-1} = \cM_{r,p}\oplus \cZ_{r,p-2}.
\end{equation}
It is easy to see that there is an exact sequence
\begin{equation}\label{zrp-2}
0 \rightarrow  \cM_{r, p-2}\rightarrow \cZ_{r,p-2} \rightarrow \cM_{r-1, 2}\oplus \cM_{r+1,2} \rightarrow 0.
\end{equation}
We claim that ${\rm Soc}(\cZ_{r,p-2}) = \cM_{r,p-2}$. Otherwise either $\cM_{r-1, 2}$ or $\cM_{r+1, 2}$ is a submodule of $\cZ_{r,p-2}$, and hence also a submodule of $\cM_{1,2}\btimes \cZ_{r,p-1}$. But in fact the rigidity of $\cM_{1,2}$, the fusion rules \eqref{more1}, and the Loewy diagram of $\cZ_{r,p-1}$ imply
\begin{align*}
 \Hom_{\cM(p)}(\cM_{r\pm 1,2}, \cM_{1,2}\btimes \cZ_{r,p-1}) & \cong \Hom_{\cM(p)}(\cM_{1,2}\btimes \cM_{r \pm 1,2}, \cZ_{r,p-1})\nonumber\\
 & \cong \Hom_{\cM(p)}(\cM_{r\pm1,1}\oplus\cM_{r,\pm1,3},\cZ_{r,p-1}) = 0,
\end{align*}
proving the claim. Thus the Loewy diagram of $\cZ_{r,p-2}$ is
\[
\begin{tikzpicture}[->,>=latex,scale=1.5]
\node (b1) at (1,0) {$\cM_{r,p-2}.$};
\node (c1) at (-1, 1){$\cZ_{r, p-2}$:};
   \node (a1) at (0,1) {$\cM_{r-1, 2}$};
   \node (b2) at (2,1) {$\cM_{r+1, 2}$};
  \draw[] (b1) -- node[left] {} (a1);
   \draw[] (b1) -- node[left] {} (b2);
    \end{tikzpicture}
\]

Now we apply $\cM_{1,2}\btimes \bullet$ to \eqref{prp-1} and use \eqref{more1} and the decomposition \eqref{decompzrp-1} of $\cZ_{r,p-1}$ to get
\begin{equation}
0 \rightarrow \cM_{r,p}\oplus \cZ_{r,p-2} \rightarrow \cM_{1,2}\btimes \cP_{r,p-1} \rightarrow \cM_{r,p-2}\oplus \cM_{r,p} \rightarrow 0.
\end{equation}
Because $\cM_{r,p}$ is both projective and injective in $\cC_{\cM(p)}^0$, $2\cdot \cM_{r,p}$ is a direct summand of $\cM_{1,2}\btimes \cP_{r,p-1}$. Defining $\cP_{r,p-2}$ to be a direct summand of $\cM_{1,2}\btimes \cP_{r,p-1}$ complementary to $2\cdot\cM_{r,p}$, we get an exact sequence
\begin{equation}\label{prp-2}
0 \rightarrow \cZ_{r,p-2} \rightarrow \cP_{r, p-2} \rightarrow \cM_{r,p-2} \rightarrow 0.
\end{equation}
We claim that ${\rm Soc}(\cP_{r,p-2}) =\cM_{r,p-2}$. Otherwise ${\rm Soc}(\cP_{r,p-2}) =2\cdot\cM_{r,p-2}$ and then
\[
\dim\Hom_{\cM(p)}(\cM_{1,2}\btimes \cM_{r,p-2}, \cP_{r,p-1}) = \dim\Hom_{\cM(p)}(\cM_{r,p-2}, \cM_{1,2}\btimes \cP_{r,p-1}) = 2,
\]
whereas in fact
\begin{equation*}
 \dim\Hom_{\cM(p)}(\cM_{1,2}\btimes \cM_{r,p-2}, \cP_{r,p-1}) =\dim\Hom_{\cM(p)}(\cM_{r,p-3}\oplus\cM_{r,p-1},\cP_{r,p-1})=1,
\end{equation*}
a contradiction. The claim follows.

The exact sequence \eqref{prp-2} gives
\begin{equation}\label{pmrp-2}
0 \rightarrow \cZ_{r,p-2}/\cM_{r,p-2} \rightarrow \cP_{r, p-2}/\cM_{r,p-2} \rightarrow \cM_{r,p-2} \rightarrow 0.
\end{equation}
We claim that $\cP_{r,p-2}/\cM_{r,p-2}$ is indecomposable, so that in particular \eqref{pmrp-2} does not split and
$${\rm Soc}(\cP_{r, p-2}/\cM_{r,p-2}) = \cZ_{r,p-2}/\cM_{r,p-2} = \cM_{r+1, 2}\oplus \cM_{r-1,2}.$$
Now if $\cP_{r,p-2}/\cM_{r,p-2}$ were decomposable, then because it has length $3$ and contains $\cM_{r-1,2}\oplus\cM_{r+1,2}$ as a submodule, it would have to contain either $\cM_{r-1,2}$ or $\cM_{r+1,2}$ as a summand. But using the rigidity of $\cM_{1,2}$, this would imply
\begin{align*}
\Hom_{\cM(p)}(\cP_{r,p-1}/\cM_{r,p-1},\cM_{1,2} & \btimes \cM_{r\pm1,2}) \cong \Hom_{\cM(p)}( \cM_{1,2}\btimes (\cP_{r,p-1}/\cM_{r,p-1}),\cM_{r\pm1,2})\\
& \cong\Hom_{\cM(p)}\big((\cM_{1,2}\btimes \cP_{r,p-1})/(\cM_{1,2}\btimes \cM_{r,p-1}),\cM_{r\pm1,2}\big)\\
& \cong \Hom_{\cM(p)}\big( (\cP_{r,p-2}\oplus 2\cdot\cM_{r,p})/(\cM_{r,p}\oplus \cM_{r,p-2}),\cM_{r\pm1,2}\big)\\
& \cong \Hom_{\cM(p)}\big( (\cP_{r,p-2}/\cM_{r,p-2})\oplus \cM_{r,p},\cM_{r\pm1,2}\big) \neq 0,
\end{align*}
whereas in fact the Loewy diagram for $\cP_{r,p-1}$ shows
\begin{align*}
 \Hom_{\cM(p)}(\cP_{r,p-1}/\cM_{r,p-1}, & \cM_{1,2}  \btimes \cM_{r\pm1,2})\nonumber\\
 &\cong\Hom_{\cM(p)}(\cP_{r,p-1}/\cM_{r,p-1},\cM_{r\pm1,1}\oplus\cM_{r\pm1,3}) = 0.
\end{align*}
This proves the claim, and it also finishes the verification of the Loewy diagram for $\cP_{r,p-2}$ in the next lemma:
\begin{lemma}\label{lem:p-2_proj}
For $r\in\ZZ$, $\cP_{r,p-2}$ is a projective cover of $\cM_{r,p-2}$ in $\cC_{\cM(p)}^0$, and it has Loewy diagram
\[
\begin{tikzpicture}[->,>=latex,scale=1.5]
\node (b1) at (1,0) {$\cM_{r,p-2}.$};
\node (c1) at (-1, 1){$\cP_{r, p-2}$:};
   \node (a1) at (0,1) {$\cM_{r-1, 2}$};
   \node (b2) at (2,1) {$\cM_{r+1, 2}$};
    \node (a2) at (1,2) {$\cM_{r,p-2}$};
\draw[] (b1) -- node[left] {} (a1);
   \draw[] (b1) -- node[left] {} (b2);
    \draw[] (a1) -- node[left] {} (a2);
    \draw[] (b2) -- node[left] {} (a2);
\end{tikzpicture}
\]
\end{lemma}
\begin{proof}
It remains to show that $\cP_{r,p-2}$ is a projective cover of $\cM_{r,p-2}$ in $\cC_{\cM(p)}^0$. Since $\cM_{1,2}$ is rigid and $\cP_{r,p-1}$ is projective in $\cC_{\cM(p)}^0$, $\cM_{1,2}\btimes \cP_{r,p-1}$ is also projective (see Corollary 2 from the Appendix of \cite{KL5}). Thus as a direct summand of $\cM_{1,2}\btimes \cP_{r,p-1}$, $\cP_{r,p-2}$ is  projective in $\cC_{\cM(p)}^0$, and we have a surjective map $p:\cP_{r,p-2}\rightarrow\cM_{r,p-2}$. Now suppose $P$ is projective and $q: P \rightarrow \cM_{r,p-2}$ is a surjective map. Then there exists $f: P \rightarrow \cP_{r,p-2}$ such that
\begin{equation*}
  \xymatrixcolsep{3pc}
  \xymatrix{
   & P \ar[ld]_{f} \ar[d]^q \\
  \cP_{r,p-2}  \ar[r]_{p} & \cM_{r,p-2}
  }
 \end{equation*}
 commutes. We need to show that $f$ is surjective. The Loewy diagram for $\cP_{r,p-2}$ shows that $\cZ_{r,p-2}$ is the unique maximal submodule of $\cP_{r,p-2}$, so if $f$ is not surjective, then $\mathrm{im}\,f\subseteq\cZ_{r,p-2}=\ker p$. But this contradicts $q\neq 0$, so in fact $f$ is surjective.
\end{proof}

Now that we have projective covers $\cP_{r,p-2}$, $\cP_{r,p-1}$, and $\cP_{r,p}$ for $r\in\ZZ$, we proceed to construct modules $\cP_{r,s}$ for $1\leq s\leq p-3$ recursively. Fix $s\in\lbrace 2,3,\ldots, p-2\rbrace$ and assume we have $\cP_{r,\sigma}$ for all $s\leq\sigma\leq p-1$ such that:
\begin{itemize}
\item The module $\cP_{r,\sigma}$ is a projective cover of $\cM_{r,\sigma}$ in $\cC_{\cM(p)}^0$, and
\item The Loewy diagram of $\cP_{r,\sigma}$ is
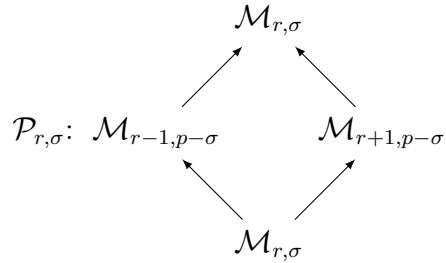
\[
\begin{tikzpicture}[->,>=latex,scale=1.5]
\node (b1) at (1,0) {$\cM_{r, \sigma}$};
\node (c1) at (-1, 1){$\cP_{r, \sigma}$:};
   \node (a1) at (0,1) {$\cM_{r-1, p-\sigma}$};
   \node (b2) at (2,1) {$\cM_{r+1, p-\sigma}$};
    \node (a2) at (1,2) {$\cM_{r,\sigma}$};
\draw[] (b1) -- node[left] {} (a1);
   \draw[] (b1) -- node[left] {} (b2);
    \draw[] (a1) -- node[left] {} (a2);
    \draw[] (b2) -- node[left] {} (a2);
\end{tikzpicture}
\]
\end{itemize}
We now define $\cP_{r,s-1}$ as follows. We have a surjection
\begin{align*}
 \cM_{1,2}\boxtimes\cP_{r,s}\rightarrow\cM_{1,2}\boxtimes\cM_{r,s}\cong\cM_{r,s-1}\oplus\cM_{r,s+1}\rightarrow\cM_{r,s+1}.
\end{align*}
Then because $\cM_{1,2}$ is rigid and $\cP_{r,s}$ is projective, $\cM_{1,2}\boxtimes\cP_{r,s}$ is also projective. So because $\cP_{r,s+1}$ is the projective cover of $\cM_{r,s+1}$, we get a a surjective map
\begin{equation*}
 \cM_{1,2}\boxtimes\cP_{r,s}\rightarrow \cP_{r,s+1}.
\end{equation*}
Since $\cP_{r,s+1}$ is projective, this surjection splits and $\cP_{r,s+1}$ is a direct summand of $\cM_{1,2}\boxtimes\cP_{r,s}$. We now define $\cP_{r,s-1}$ to be a complement of $\cP_{r,s+1}$:
\[
\cM_{1,2}\btimes \cP_{r, s} = \cP_{r,s+1} \oplus \cP_{r, s-1}.
\]
The module $\cP_{r,s-1}$ is in $\cC_{\cM(p)}^0$ because this category is closed under tensor products and submodules, and it is projective in $\cC_{\cM(p)}^0$ because it is a summand of a projective module. We can now prove:
\begin{theorem}
The module $\cP_{r,s-1}$ is a projective cover of $\cM_{r,s-1}$ in $\cC_{\cM(p)}^0$
with the following Loewy diagram:
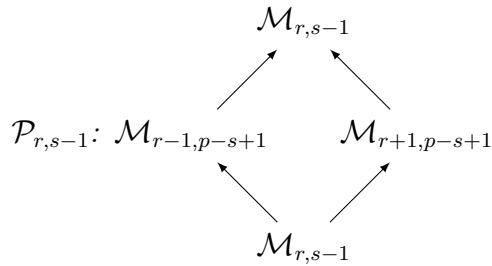
\[
\begin{tikzpicture}[->,>=latex,scale=1.5]
\node (b1) at (1,0) {$\cM_{r, s-1}$};
\node (c1) at (-1.2, 1){$\cP_{r, s-1}$:};
   \node (a1) at (0,1) {$\cM_{r-1, p-s+1}$};
   \node (b2) at (2,1) {$\cM_{r+1, p-s+1}$};
    \node (a2) at (1,2) {$\cM_{r,s-1}$};
\draw[] (b1) -- node[left] {} (a1);
   \draw[] (b1) -- node[left] {} (b2);
    \draw[] (a1) -- node[left] {} (a2);
    \draw[] (b2) -- node[left] {} (a2);
\end{tikzpicture}
\]
\end{theorem}
\begin{proof}
From the Loewy diagram for $\cP_{r,s}$, we see that $\cP_{r,s}$ has a unique maximal proper submodule $\cZ_{r,s}$ with Loewy diagram
\[
\begin{tikzpicture}[->,>=latex,scale=1.5]
\node (b1) at (1,0) {$\cM_{r,s}$};
\node (c1) at (-1, 1){$\cZ_{r, s}$:};
   \node (a1) at (0,1) {$\cM_{r-1, p-s}$};
   \node (b2) at (2,1) {$\cM_{r+1, p-s}$};
  \draw[] (b1) -- node[left] {} (a1);
   \draw[] (b1) -- node[left] {} (b2);
    \end{tikzpicture}
\]
Then the following non-split exact sequences are clear:
\begin{equation}\label{zrs}
0 \longrightarrow \cM_{r,s} \longrightarrow \cZ_{r,s} \longrightarrow \cM_{r-1,p-s}\oplus \cM_{r+1, p-s} \longrightarrow 0,
\end{equation}
and
\begin{equation}\label{prs}
0 \longrightarrow \cZ_{r,s} \longrightarrow \cP_{r,s} \longrightarrow \cM_{r,s} \longrightarrow 0.
\end{equation}
Applying $\cM_{1,2} \btimes \bullet $ to \eqref{zrs} and using the fusion rules \eqref{more1} yields the exact sequence
\begin{align*}
0 \longrightarrow \cM_{r,s-1}\oplus \cM_{r,s+1} &\longrightarrow \cM_{1,2}\btimes \cZ_{r,s}\nonumber\\ &\longrightarrow \cM_{r-1,p-s-1}\oplus \cM_{r-1,p-s+1}\oplus \cM_{r+1, p-s-1}\oplus \cM_{r+1,p-s+1} \longrightarrow 0.
\end{align*}
From this we see that the conformal weights of $\cM_{1,2}\boxtimes\cZ_{r,s}$ are contained in the two distinct cosets $h_{r,s\pm1}+\ZZ$, and thus $\cM_{1,2}\boxtimes\cZ_{r,s}$ decomposes as a direct sum of two modules, say $\widetilde{\cZ}_{r,s+1}$ and $\cZ_{r,s-1}$, satisfying the exact sequences
\begin{equation*}
0 \rightarrow \cM_{r,s-1} \rightarrow \cZ_{r,s-1} \rightarrow \cM_{r-1,p-s+1}\oplus \cM_{r+1,p-s+1} \rightarrow 0,
\end{equation*}
and
\begin{equation*}
0 \rightarrow \cM_{r,s+1} \rightarrow \widetilde{\cZ}_{r,s+1} \rightarrow \cM_{r-1,p-s-1}\oplus \cM_{r+1, p-s-1} \rightarrow 0.
\end{equation*}
If either $\cM_{r-1,p-s+1}$ or $\cM_{r+1,p-s+1}$ were a submodule of $\cZ_{r,s-1}$, and thus also a submodule of $\cM_{1,2}\btimes \cZ_{r,s}$, then rigidity of $\cM_{1,2}$ would imply
\[
\Hom_{\cM(p)}(\cM_{1,2}\btimes\cM_{r\pm 1,p-s+1}, \cZ_{r,s}) \cong \Hom_{\cM(p)}(\cM_{r\pm 1,p-s+1}, \cM_{1,2}\btimes \cZ_{r,s}) \neq 0.
\]
However, by the fusion rules \eqref{more1} and the Loewy diagram of $\cZ_{r,s}$, there is no non-zero homomorphism
\[
\cM_{1,2}\btimes\cM_{r\pm 1,p-s+1} \cong \cM_{r\pm 1, p-s}\oplus \cM_{r\pm 1, p-s+2} \rightarrow \cZ_{r,s}.
\]
As a result, ${\rm Soc}(\cZ_{r,s-1}) = \cM_{r,s-1}$.

Now we apply $\cM_{1,2}\btimes \bullet$ to  \eqref{prs} to get an exact sequence
\begin{equation*}
0 \rightarrow \cZ_{r,s-1}\oplus \widetilde{\cZ}_{r,s+1} \rightarrow \cM_{1,2}\btimes\cP_{r,s} \rightarrow \cM_{r,s-1}\oplus \cM_{r,s+1} \rightarrow 0.
\end{equation*}
Thus conformal weight considerations again show that $\cP_{r,s-1}$ satisfies the exact sequence
\begin{equation}\label{prs-1}
0 \rightarrow \cZ_{r,s-1} \rightarrow \cP_{r,s-1} \rightarrow \cM_{r,s-1} \rightarrow 0.
\end{equation}
We claim that ${\rm Soc}(\cP_{r,s-1}) = \cM_{r,s-1}$. If not, then ${\rm Soc}(\cP_{r,s-1}) = 2\cdot\cM_{r,s-1}$ and rigidity of $\cM_{1,2}$ would imply
\[
\dim \Hom_{\cM(p)}(\cM_{1,2}\btimes\cM_{r,s-1}, \cP_{r,s}) = \dim \Hom_{\cM(p)}(\cM_{r,s-1}, \cM_{1,2}\btimes \cP_{r,s}) = 2.
\]
However, this would contradict
\begin{align*}
\dim \Hom_{\cM(p)}(\cM_{1,2}\btimes\cM_{r,s-1}, \cP_{r,s}) &= \dim \Hom_{\cM(p)}\big([\cM_{r,s-2}\oplus]\cM_{r,s}, \cP_{r,s}\big) = 1,
\end{align*}
where the summand in brackets occurs for $s\geq 3$.

The exact sequence \eqref{prs-1} gives an exact sequence
\[
0 \rightarrow \cZ_{r,s-1}/\cM_{r,s-1} \rightarrow \cP_{r,s-1}/\cM_{r,s-1} \rightarrow \cM_{r,s-1} \rightarrow 0
\]
We claim that $\cP_{r,s-1}/\cM_{r,s-1}$ is indecomposable, so that in particular
$${\rm Soc}(\cP_{r,s-1}/\cM_{r,s-1}) = \cZ_{r,s-1}/\cM_{r,s-1} = \cM_{r-1, p-s+1}\oplus \cM_{r+1,p-s+1}.$$
If $\cP_{r,s-1}/\cM_{r,s-1}$ were decomposable, then because it contains $\cM_{r-1, p-s+1}\oplus \cM_{r+1,p-s+1}$ as a submodule and has length $3$, it would have to contain either $\cM_{r-1,p-s+1}$ or $\cM_{r+1,p-s+1}$ as a direct summand. Then rigidity of $\cM_{1,2}$ would imply
\begin{align*}
\Hom_{\cM(p)}( \cP_{r,s}/\cM_{r,s}, & \,\cM_{1,2}\btimes \cM_{r\pm1,p-s+1}) \cong \Hom_{\cM(p)}( \cM_{1,2}\btimes (\cP_{r,s}/\cM_{r,s}), \cM_{r\pm1,p-s+1})\\
& \cong \Hom_{\cM(p)}\big( (\cM_{1,2}\btimes \cP_{r,s})/(\cM_{1,2}\btimes \cM_{r,s}),\cM_{r\pm1,p-s+1}\big)\\
& \cong \Hom_{\cM(p)}\big( (\cP_{r,s-1}\oplus \cP_{r,s+1})/(\cM_{r,s-1}\oplus \cM_{r,s+1}),\cM_{r\pm1,p-s+1}\big)\\
& \cong \Hom_{\cM(p)}\big( (\cP_{r,s-1}/\cM_{r,s-1})\oplus (\cP_{r,s+1}/\cM_{r,s+1}),\cM_{r\pm1,p-s+1}\big) \neq 0.
\end{align*}
However, in fact
\begin{align*}
\Hom_{\cM(p)}(\cP_{r,s}/\cM_{r,s}, &\,\cM_{1,2}\btimes \cM_{r\pm1,p-s+1} )\nonumber\\
&= \Hom_{\cM(p)}(\cP_{r,s}/\cM_{r,s},\cM_{r\pm1,p-s}\oplus \cM_{r\pm1,p-s+2} ) = 0.
\end{align*}
Thus $\cP_{r,s-1}/\cM_{r,s-1}$ is indecomposable, and we have verified the Loewy diagram for $\cP_{r,s-1}$.

Now we show that the projective module $\cP_{r,s-1}$ is a projective cover of $\cM_{r,s-1}$ in $\cC_{\cM(p)}^0$. The Loewy diagram shows that $\cZ_{r,s-1}$ is the unique maximal proper submodule of $\cP_{r,s-1}$ and that there is a surjective map $p:\cP_{r,s-1}\rightarrow\cM_{r,s-1}$ with kernel $\cZ_{r,s-1}$. Thus for any surjective map $q:P\rightarrow\cM_{r,s-1}$ in $\cC_{\cM(p)}^0$ with $P$ projective, a map $f: P\rightarrow\cP_{r,s-1}$ such that the diagram
\begin{equation*}
 \xymatrixcolsep{3pc}
 \xymatrix{
 & P \ar[d]^{q} \ar[ld]_{f} \\
 \cP_{r,s-1} \ar[r]_{p} & \cM_{r,s-1} \\
 }
\end{equation*}
commutes is necessarily surjective, just as in the proof of Lemma \ref{lem:p-2_proj}
\end{proof}

The projective covers $\cP_{r,s}$ for $r \in \ZZ$, $1\leq s \leq p$ have the following fusion rules:
\begin{theorem}\label{thm:prs_fusion} \hspace{2em}
\begin{itemize}
\item[(1)]For $r \in \ZZ$, $1\leq s \leq p$,
\begin{equation}\label{moreprs}
\cM_{2,1}\btimes \cP_{r,s} = \cP_{r+1,s}.
\end{equation}
\item[(2)]For $p\geq 3$ and $r \in \ZZ$, $1\leq s \leq p-1$,
\begin{equation}\label{more1prs}
\cM_{1,2}\btimes \cP_{r,s} =
\begin{cases}
\cP_{r,2}\oplus \cP_{r+1,p}\oplus \cP_{r-1, p} \;\;\; &\mbox{if}\; s = 1\\
\cP_{r,s-1}\oplus \cP_{r,s+1},\;\;\; &\mbox{if}\; 1<s<p-1\\
\cP_{r,p-2}\oplus 2\cdot \cP_{r,p}, \;\;\; &\mbox{if}\; s = p-1.
\end{cases}
\end{equation}
\item[(3)] For $p=2$ and $r\in\ZZ$,
\begin{equation}\label{more2prs}
 \cM_{1,2}\boxtimes\cP_{r,1} = \cP_{r+1,2}\oplus 2\cdot\cP_{r,2}\oplus\cP_{r-1,2}.
\end{equation}
\end{itemize}
\end{theorem}
\begin{proof}
It remains to prove \eqref{moreprs}, the $s=1$ case of \eqref{more1prs}, and \eqref{more2prs}.

Using the fusion rules \eqref{more}, $\cM_{2,1}\btimes \cP_{r,s}$ is projective with a surjective map $\cM_{2,1}\btimes \cP_{r,s}\rightarrow \cM_{r+1,s}$. Because $\cP_{r+1,s}$ is a projective cover of $\cM_{r+1,s}$, it is a direct summand of $\cM_{2,1}\btimes \cP_{r,s}$. However, $\cM_{2,1}\boxtimes\cP_{r,s}$ is indecomposable because $\cM_{2,1}$ is a simple current and $\cP_{r,s}$ is indecomposable. So in fact $\cM_{1,2}\boxtimes\cP_{r,s} = \cP_{r+1,s}$, proving \eqref{moreprs}.

Now for the $s=1$ case of \eqref{more1prs},
the unique maximal proper submodule $\cZ_{r,1}$ of $\cP_{r,1}$ satisfies the exact sequence
\begin{equation*}
0 \rightarrow \cM_{r,1} \rightarrow \cZ_{r,1} \rightarrow \cM_{r-1, p-1}\oplus \cM_{r+1, p-1} \rightarrow 0.
\end{equation*}
Applying $\cM_{1,2}\btimes \bullet$ and using the fusion rules \eqref{more1}, we have
\begin{equation}\label{seq:more1prs_proof}
0 \rightarrow \cM_{r,2} \rightarrow \cM_{1,2}\btimes \cZ_{r,1} \rightarrow \cM_{r-1, p-2}\oplus \cM_{r+1, p-2}\oplus \cM_{r-1,p}\oplus \cM_{r+1,p} \rightarrow 0.
\end{equation}
Since both of $\cM_{r\pm1,p}$ are projective, $\cM_{r-1,p}\oplus \cM_{r+1,p}$ is a direct summand of $\cM_{1,2}\btimes \cZ_{r,1}$. Then the complement $\widetilde{\cZ}_{r, 2}$ of $\cM_{r-1,p}\oplus \cM_{r+1,p}$ satisfies the exact sequence
\begin{equation}\label{seq:z_til}
0 \rightarrow \cM_{r,2} \rightarrow \widetilde{\cZ}_{r, 2} \rightarrow \cM_{r-1, p-2}\oplus \cM_{r+1, p-2} \rightarrow 0.
\end{equation}

Now consider the exact sequence
\begin{equation*}
0 \rightarrow \cZ_{r,1} \rightarrow \cP_{r,1} \rightarrow \cM_{r,1} \rightarrow 0.
\end{equation*}
Applying $\cM_{1,2}\btimes \bullet$ and using the fusion rules \eqref{more1}, we have
\begin{equation*}
0 \rightarrow \widetilde{\cZ}_{r,2}\oplus \cM_{r-1,p}\oplus \cM_{r+1,p} \rightarrow \cM_{1,2}\btimes \cP_{r,1} \rightarrow \cM_{r,2} \rightarrow 0.
\end{equation*}
Since both of $\cM_{r\pm1,p}$ are injective, there exists a direct summand $\widetilde{\cP}_{r,2}$ of $\cM_{1,2}\btimes \cP_{r,1}$ complementary to $\cM_{r-1,p}\oplus \cM_{r+1,p}$ satisfying the exact sequence
\begin{equation}\label{seq:p_til}
0 \rightarrow \widetilde{\cZ}_{r,2} \rightarrow \widetilde{\cP}_{r,2} \rightarrow \cM_{r,2} \rightarrow 0.
\end{equation}
The module $\widetilde{P}_{r,2}$ is projective in $\cC_{\cM(p)}^0$ since it is a summand of a projective module. Since $\cP_{r,2}$ is a projective cover of $\cM_{r,2}$, there is thus a surjection $\widetilde{P}_{r,2}\rightarrow\cP_{r,2}$; but since \eqref{seq:z_til} and \eqref{seq:p_til} show that these two modules have the same length, we get $\widetilde{P}_{r,2}\cong\cP_{r,2}$. Therefore
\[
\cM_{1,2}\btimes \cP_{r,1} = \cP_{r,2} \oplus \cM_{r-1,p}\oplus \cM_{r+1,p},
\]
proving \eqref{more1prs} for $s = 1$.

Now when $p=2$, we need to replace the exact sequence \eqref{seq:more1prs_proof} with
\begin{equation*}
 0\longrightarrow\cM_{r,2}\longrightarrow\cM_{1,2}\boxtimes\cZ_{r,1}\longrightarrow \cM_{r-1,2}\oplus\cM_{r+1,2}\longrightarrow 0.
\end{equation*}
Since both $\cM_{r\pm1,2}=\cP_{r\pm1,2}$ are projective, this exact sequence splits. The exact sequence
\begin{equation*}
 0\longrightarrow \cM_{1,2}\boxtimes\cZ_{r,1}\longrightarrow\cM_{1,2}\boxtimes\cP_{r,1}\longrightarrow \cM_{1,2}\boxtimes\cM_{r,1}\longrightarrow 0
\end{equation*}
also splits because $\cM_{1,2}\boxtimes\cM_{r,1}\cong\cM_{r,2}$ is projective. Then these two split exact sequences together imply \eqref{more2prs}.
\end{proof}

\subsection{General fusion rules}
Finally, here are all fusion rules involving the simple modules $\cM_{r,s}$ and their projective covers in $\cC_{\cM(p)}^0$:
\begin{theorem}\label{generalfusionrules}
All fusion products of the $\cM(p)$-modules $\cM_{r,s}$ and $\cP_{r,s}$ are as follows, with sums taken to be empty if the lower bound exceeds the upper bound:
\begin{itemize}
\item[(1)] For $r, r' \in \ZZ$ and $1 \leq s, s' \leq p$,
\begin{align}\label{caser}
& \cM_{r,s}\btimes \cM_{r',s'}   = \bigg(\bigoplus_{\substack{\ell = |s-s'|+1 \\ \ell+s+s' \equiv 1\; (\mathrm{mod}\; 2)}}^{{\rm min}\{s+s'-1, 2p-1-s-s'\}}\cM_{r+r'-1, \ell}\bigg) \oplus \bigg(\bigoplus_{\substack{\ell = 2p+1-s-s'\\ \ell+s+s' \equiv 1\; (\mathrm{mod}\; 2)}}^{p}\cP_{r+r'-1, \ell}\bigg).
\end{align}
\item[(2)] For $r, r' \in \ZZ$, $1 \leq s \leq p-1$ and $1 \leq s' \leq p$,
\begin{align}\label{casePM}
\cP_{r,s}\btimes \cM_{r', s'} & = \bigg(\bigoplus_{\substack{\ell = |s-s'|+1\\ \ell+s+s' \equiv 1\; ({\rm mod}\; 2)}}^{{\rm min}\{s+s'-1, p\}}\cP_{r+r'-1, \ell}\bigg)\oplus \bigg(\bigoplus_{\substack{\ell = 2p+1-s-s'\\ \ell+s+s' \equiv 1\; ({\rm mod}\; 2)}}^{p}\cP_{r+r'-1, \ell}\bigg)\nonumber\\
& \qquad \oplus \bigoplus_{\substack{\ell = p+s-s'+1\\ \ell+p+s+s' \equiv 1\; ({\rm mod}\; 2)}}^{p}\big(\cP_{r+r', \ell}\oplus \cP_{r+r'-2, \ell}\big).
\end{align}
\item[(3)]For $r,r' \in \ZZ$ and $1\leq s,s' \leq p-1$,
\begin{align}\label{casePP}
 \cP_{r,s}\btimes \cP_{r',s'}
& = \bigg(2\cdot \bigoplus_{\substack{\ell = |s-s'|+1\\ \ell+s+s' \equiv 1\; ({\rm mod}\; 2)}}^{{\rm min}\{s+s'-1, p\}}\cP_{r+r'-1, \ell}\bigg)\oplus \bigg(2\cdot \bigoplus_{\substack{\ell = 2p+1-s-s'\\ \ell+s+s' \equiv 1\; ({\rm mod}\; 2)}}^{p}\cP_{r+r'-1, \ell}\bigg)\nonumber\\
&\qquad \oplus 2 \cdot\bigoplus_{\substack{\ell = p+s-s'+1\\ \ell+p+s+s' \equiv 1\; ({\rm mod}\; 2)}}^{p}\big(\cP_{r+r', \ell}\oplus \cP_{r+r'-2, \ell}\big)\nonumber\\
&\qquad \oplus \bigoplus_{\substack{\ell = |s+s'-p|+1\\ \ell+p+s+s' \equiv 1\; ({\rm mod}\; 2)}}^{{\rm min}\{s-s'+p-1, p\}}\big(\cP_{r+r', \ell}\oplus \cP_{r+r'-2, \ell}\big)\nonumber\\
&\qquad \oplus \bigoplus_{\substack{\ell = p-s+s'+1\\ \ell+p+s+s' \equiv 1\; ({\rm mod}\; 2)}}^{p}\big(\cP_{r+r', \ell}\oplus \cP_{r+r'-2, \ell}\big)\nonumber\\
&\qquad \oplus \bigoplus_{\substack{\ell = s+s'+1\\ \ell+s+s' \equiv 1\; ({\rm mod}\; 2)}}^{p}\big(\cP_{r+r'+1, \ell}\oplus 2\cdot\cP_{r+r'-1, \ell}\oplus \cP_{r+r'-3, \ell}\big).
\end{align}
 \end{itemize}
\end{theorem}

\begin{proof}
We will see that the fusion rules \eqref{caser} and \eqref{casePM} are completely determined by repeated applications of \eqref{more}, \eqref{more1}, \eqref{moreprs}, \eqref{more1prs}, and \eqref{more2prs}. As it can be seen that these recursion formulas agree with those for the unrolled quantum group of $\sl_2$ given in \cite{CGP}, the general fusion rules must therefore agree with those in \cite[Lemma 8.1, Proposition 8.2, Corollary 8.3, Proposition 8.4]{CGP}. Thus we shall give a detailed proof for \eqref{caser} only; the fusion rules \eqref{casePM} can be proved similarly, and then \eqref{casePP} will follow from \eqref{casePM} combined with rigidity and the projectivity of $\cP_{r,s}$ in $\cC_{\cM(p)}^0$.

We first use induction on $s$ to prove
\begin{equation}\label{caser_1}
\cM_{1,s}\btimes \cM_{r',s'} = \bigg(\bigoplus_{\substack{\ell = |s-s'|+1\\ \ell+s+s' \equiv 1\; (\mathrm{mod}\; 2)}}^{{\rm min}\{s+s'-1, 2p-1-s-s'\}}\cM_{r', \ell}\bigg) \oplus \bigg(\bigoplus_{\substack{\ell = 2p+1-s-s'\\ \ell+s+s' \equiv 1\; (\mathrm{mod}\; 2)}}^{p}\cP_{r', \ell}\bigg).
\end{equation}
From the fusion rules \eqref{more1}, \eqref{caser_1} is true for $s = 1, 2$. Now assume \eqref{caser_1} holds for some fixed $s\in\lbrace 2,3,\ldots p-1\rbrace$; we shall prove it holds for $s+1$. We first tensor   the left side of \eqref{caser_1} with $\cM_{1,2}$ to get
\begin{align*}
\cM_{1,2}\btimes (\cM_{1,s}\btimes \cM_{r',s'}) &= (\cM_{1,2}\btimes \cM_{1,s})\btimes \cM_{r',s'}\nonumber\\
& = (\cM_{1, s-1} \oplus \cM_{1,s+1})\btimes \cM_{r',s'}\nonumber \\
& = (\cM_{1, s-1} \btimes \cM_{r',s'}) \oplus (\cM_{1,s+1}\btimes \cM_{r',s'}).
\end{align*}
Then because all tensor product modules here have finite length, the Krull-Schmidt Theorem guarantees that we can determine the indecomposable summands of $\cM_{1,s+1}\boxtimes\cM_{r',s'}$ by subtracting those of $\cM_{1,s-1}\boxtimes\cM_{r',s'}$ from those of $\cM_{1,2}\boxtimes(\cM_{1,s}\boxtimes\cM_{r',r'})$. We consider the following three cases:
\begin{itemize}
\item If $s+s' \leq p$, by induction
\begin{equation*}
\cM_{1,s-1} \btimes {\cM_{r',s'}} = \bigoplus_{\substack{\ell = |s-s'-1|+1\\ \ell+s-1+s' \equiv 1\; (\mathrm{mod}\; 2)}}^{s+s'-2}\cM_{r', \ell},
\end{equation*}
and by induction and the fusion rules \eqref{more1}
\begin{equation*}
\cM_{1,2}\btimes (\cM_{1,s}\btimes \cM_{r',s'}) = \bigoplus_{\substack{\ell = |s-s'|+1\\ \ell+s+s' \equiv 1\; (\mathrm{mod}\; 2)}}^{s+s'-1}\big(\cM_{r', \ell+1} \oplus \cM_{r', \ell-1}\big),
\end{equation*}
where we define $\cM_{r',0}=0$ in case $s=s'$. Thus we can verify that
\begin{equation*}
\cM_{1,s+1}\btimes \cM_{r',s'} = \frac{\cM_{1,2}\btimes (\cM_{1,s}\btimes \cM_{r',s'})}{\cM_{1,s-1} \btimes {\cM_{r',s'}}}= \bigoplus_{\substack{\ell = |s-s'+1|+1\\ \ell+s+1+s' \equiv 1\; (\mathrm{mod}\; 2)}}^{s+s'}\cM_{r', \ell},
\end{equation*}
which is the same as \eqref{caser_1} in the case $s+1$. (This is trivial if $s+s' < p$, and we use $\cP_{r',p} = \cM_{r',p}$ if $s+s' = p$.)
\item If $s+s' = p+1$, by induction
\begin{equation*}
\cM_{1,s-1} \btimes {\cM_{r',s'}} = \bigoplus_{\substack{\ell = |s-s'-1|+1\\ \ell+p \equiv 1\; (\mathrm{mod}\; 2)}}^{p-1}\cM_{r', \ell},
\end{equation*}
and by induction and the fusion rules \eqref{more1}
\begin{equation*}
\cM_{1,2}\btimes (\cM_{1,s}\btimes \cM_{r',s'}) = \bigg(\bigoplus_{\substack{\ell = |s-s'|+1\\ \ell+p+1 \equiv 1\; (\mathrm{mod}\; 2)}}^{p-2}\big(\cM_{r', \ell+1} \oplus \cM_{r', \ell-1}\big)\bigg) \oplus \cP_{r', p-1}.
\end{equation*}
 Thus
\begin{align*}
\cM_{1,s+1}\btimes \cM_{r',s'} & = \frac{\cM_{1,2}\btimes (\cM_{1,s}\btimes \cM_{r',s'})}{\cM_{1,s-1} \btimes {\cM_{r',s'}}}\\
&= \bigg(\bigoplus_{\substack{\ell = |s-s'+1|+1\\ \ell+p+2 \equiv 1\; (\mathrm{mod}\; 2)}}^{p-3}\cM_{r', \ell}\bigg) \oplus \cP_{r', p-1},
\end{align*}
which is the same as \eqref{caser_1} for $s+1$.
\item If $p+2 \leq s+s' < 2p$, by induction,
\begin{equation*}\label{eqn:s-1_s'}
\cM_{1,s-1}\btimes \cM_{r',s'} = \bigg(\bigoplus_{\substack{\ell = |s-s'-1|+1\\ \ell+s-1+s' \equiv 1\; (\mathrm{mod}\; 2)}}^{2p-s-s'}\cM_{r', \ell}\bigg)\oplus \bigg(\bigoplus_{\substack{\ell = 2p+2-s-s'\\ \ell+s-1+s' \equiv 1\; (\mathrm{mod}\; 2)}}^{p}\cP_{r', \ell}\bigg),
\end{equation*}
and by induction and the fusion rules \eqref{more1}, \eqref{more1prs},
\begin{align*}\label{eqn:2_s_s'}
\cM_{1,2}\btimes (\cM_{1,s}\btimes \cM_{r',s'}) & = \bigoplus_{\substack{\ell = |s-s'|+1\\ \ell+s+s' \equiv 1\; (\mathrm{mod}\; 2)}}^{2p-1-s-s'}\big(\cM_{r', \ell+1} \oplus \cM_{r', \ell-1}\big)\nonumber\\
& \qquad \oplus \bigg(\bigoplus_{\substack{\ell = 2p+1-s-s'\\ \ell+s+s' \equiv 1\; (\mathrm{mod}\; 2)}}^{p}\big(\cP_{r', \ell-1}\oplus \cP_{r', \ell+1}\big)\bigg)\oplus \cA,
\end{align*}
where
\[
\cA =
\begin{cases}
\cP_{r', p} \;\;\; &{\rm if} \; p+s+s' \equiv 0 \; ({\rm mod} \; 2)\\
0 & {\rm otherwise}
\end{cases} ,
\]
and we define $\cP_{r', p+1} = 0$. From these, we can see that
\begin{align*}
\cM_{1,s+1}\btimes \cM_{r',s'} & = \frac{\cM_{1,2}\btimes (\cM_{1,s}\btimes \cM_{r',s'})}{\cM_{1,s-1} \btimes {\cM_{r',s'}}}\\
& = \bigg(\bigoplus_{\substack{\ell = |s-s'+1|+1\\ \ell+s+1+s' \equiv 1\; (\mathrm{mod}\; 2)}}^{2p-2-s-s'}\cM_{r', \ell}\bigg)\quad \oplus \bigg(\bigoplus_{\substack{\ell = 2p-s-s'\\ \ell+s+1+s' \equiv 1\; (\mathrm{mod}\; 2)}}^{p-1}\cP_{r', \ell}\bigg)\oplus \cA\\
& = \bigg(\bigoplus_{\substack{\ell = |s-s'+1|+1\\ \ell+s+1+s' \equiv 1\; (\mathrm{mod}\; 2)}}^{2p-2-s-s'}\cM_{r', \ell}\bigg) \oplus \bigg(\bigoplus_{\substack{\ell = p-s-s'\\ \ell+s+1+s' \equiv 1\; (\mathrm{mod}\; 2)}}^{p}\cP_{r', \ell}\bigg),
\end{align*}
which is the same as the fusion rule given by \eqref{caser_1} for $s+1$.
\end{itemize}
Finally, the fusion rule \eqref{caser} for an arbitrary $r \in \ZZ$ follows from \eqref{caser_1} together with the fusion rules \eqref{more} and \eqref{moreprs} involving $\cM_{2,1}$.

Now the fusion rules \eqref{casePM} can be proved similarly by induction on $s'$, using the fusion rules of Theorem \ref{thm:prs_fusion}; we omit the details since they are tedious. As for \eqref{casePP}, we first note that rigidity of $\cC_{\cM(p)}$ and projectivity of $\cP_{r,s}$ in $\cC_{\cM(p)}^0$ imply that the sequence
\begin{equation*}
 0\longrightarrow\cP_{r,s}\boxtimes\cM_{r',s'}\longrightarrow\cP_{r,s}\boxtimes\cZ_{r',s'}\longrightarrow\cP_{r,s}\boxtimes(\cM_{r'+1,p-s'}\oplus\cM_{r'-1,p-s'})\longrightarrow 0
\end{equation*}
is split exact for all $r'\in\ZZ$, $1\leq s'\leq p-1$. Similarly, the sequence
\begin{equation*}
 0\longrightarrow\cP_{r,s}\boxtimes\cZ_{r',s'}\longrightarrow\cP_{r,s}\boxtimes\cP_{r',s'}\longrightarrow\cP_{r,s}\boxtimes\cM_{r',s'}\longrightarrow 0
\end{equation*}
is split exact. Thus
\begin{equation*}
 \cP_{r,s}\boxtimes\cP_{r',s'}=2\cdot(\cP_{r,s}\boxtimes\cM_{r',s'})\oplus(\cP_{r,s}\boxtimes\cM_{r'+1,p-s'})\oplus(\cP_{r,s}\boxtimes\cM_{r'-1,p-s'}),
\end{equation*}
and then \eqref{casePP} follows from \eqref{casePM}.
\end{proof}

\bigskip

\noindent
Department of Mathematical and Statistical Sciences, University of Alberta, Edmonton, Alberta T6G 2R3, Canada\\
\emph{E-mail:} \texttt{creutzig@ualberta.ca, jinwei2@ualberta.ca}

\medskip

\noindent
Yau Mathematical Sciences Center, Tsinghua University, Beijing 100084, China\\
\emph{E-mail:} \texttt{rhmcrae@tsinghua.edu.cn}

\end{document}